\documentclass{amsart}
\usepackage{amsmath,amssymb,latexsym}
\usepackage{amscd}

\usepackage{fancyhdr}
\usepackage[all]{xy}
\usepackage{amsthm}
\usepackage{amsfonts}

\newtheorem{theorem}{Theorem}[section]
\newtheorem{lemma}[theorem]{Lemma}
\newtheorem{proposition}[theorem]{Proposition}
\newtheorem{corollary}[theorem]{Corollary}

 \theoremstyle{definition}
 \newtheorem{definition}[theorem]{Definition}
 \newtheorem{observation}[theorem]{Observation}

 \newtheorem{example}[theorem]{Example}

 \DeclareMathOperator{\Ext}{Ext}
 \DeclareMathOperator{\Hom}{Hom}
 \DeclareMathOperator{\Tor}{Tor}
 
 \DeclareMathOperator{\im}{Im}
 \DeclareMathOperator{\Ker}{Ker}
 \DeclareMathOperator{\Mod}{Mod}

\def\coOmega{\mathop{{\rm co}\Omega}\nolimits}
\def\cTr{\mathop{\rm cTr}\nolimits}
\def\acTr{\mathop{\rm acTr}\nolimits}
\def\Ker{\mathop{\rm Ker}\nolimits}
\def\Coker{\mathop{\rm Coker}\nolimits}
\def\End{\mathop{\rm End}\nolimits}
\def\pd{\mathop{\rm pd}\nolimits}
\def\fd{\mathop{\rm fd}\nolimits}
\def\id{\mathop{\rm id}\nolimits}
\def\mod{\mathop{\rm mod}\nolimits}
\def\Mod{\mathop{\rm Mod}\nolimits}
\def\Add{\mathop{\rm Add}\nolimits}

\def\FID{\mathop{\rm FID}\nolimits}
\def\FPD{\mathop{\rm FPD}\nolimits}

\def\Ecograde{\mathop{\rm {E{\text-}cograde}}\nolimits}
\def\Tcograde{\mathop{\rm {T{\text-}cograde}}\nolimits}
\def\sEcograde{\mathop{\rm {s.E{\text-}cograde}}\nolimits}
\def\sTcograde{\mathop{\rm {s.T{\text-}cograde}}\nolimits}

\def\cT{\mathop{{\rm c}\mathcal{T}}\nolimits}
\def\acT{\mathop{{\rm ac}\mathcal{T}}\nolimits}

\begin{document}

\title{\footnotesize{Cograde conditions and cotorsion pairs}}

\author{Xi Tang}
\address{College of Science, Guilin University of Technology, Guilin 541004, Guangxi Province, P.R. China \\}
\email{tx5259@sina.com.cn}

\author{Zhaoyong Huang}
\address{Department of Mathematics, Nanjing University, Nanjing 210093, Jiangsu Province, P.R. China}
\email{huangzy@nju.edu.cn}
\urladdr{http://math.nju.edu.cn/~huangzy/}



\date{}

\begin{abstract}
Let $R$ and $S$ be rings and $_R\omega_S$ a semidualizing bimodule.
We study when the double functor $\Tor^S_i(\omega, \Ext^i_{R}(\omega,-))$ preserves epimorphisms
and the double functor $\Ext_{R}^i(\omega, \Tor_i^{S}(\omega,-))$ preserves monomorphisms
in terms of the (strong) cograde conditions of modules. Under certain cograde condition of modules,
we construct two complete cotorsion pairs. In addition, we establish the relation between
some relative finitistic dimensions of rings and the right and left projective dimensions of $\omega$.
\end{abstract}

\maketitle

\section{\bf Introduction}

Let $R$ be a left and right Noetherian ring and $n\geqslant 1$. It was proved by Auslander that
the flat dimension of the $i$-th term in the minimal injective resolution of $R_R$ is at most $i$
for any $0\leqslant i <n$ if and only if the strong grade of $\Ext_R^i(M,R)$ is at least $i$
for any finitely generated left $R$-module $M$ and $1\leqslant i \leqslant n$; and this result is
left-right symmetric (\cite[Theorem 3.7]{FGR}). In this case, $R$ is called {\it Auslander $n$-Gorenstein}.
If $R$ is Auslander $n$-Gorenstein for all $n$, then it is said to satisfy the {\it Auslander condition}.
This condition is a non-commutative version of commutative Gorenstein rings. It has been known that
Auslander $n$-Gorenstein rings and the Auslander condition play a crucial role in homological algebra,
representation theory of artin algebras and non-commutative algebraic geometry,
see \cite{AR1, AR2, B, CSS, EHIS, FGR, H1, H3, HI, HQ, IS, Iy1, Iy2} and references therein.
In particular, Auslander $n$-Gorenstein rings and some generalized versions were characterized
in terms of the properties of the double functor $\Ext^i_{R^{op}}(\Ext^i_R(-,R),R)$ and certain
(strong) grade conditions of $\Ext$-modules, and a series of cotorsion pairs were constructed
under the Auslander condition (\cite{HI}).

It is well known that the (Auslander) transpose is one of the most powerful tools in representation
theory of artin algebras and Gorenstein homological algebra, see \cite{AB, ARS, EJ}. To dualize this
important and useful notion, we introduced in \cite{TH1} the notion of the
cotranspose of modules and then obtained many dual counterparts of interesting results
(\cite{TH1, TH2, TH3, TH4}). As a dual of the notion of the (strong) grade of modules,
we introduced in \cite{TH1, TH2} the notion of the (strong) cograde of modules, and obtained
the dual versions of some results about the (strong) grade of modules.
Let $R$ and $S$ be rings and $_R\omega_S$ a semidualizing bimodule.
In this paper, we will study when the double functor $\Tor^S_i(\omega, \Ext^i_{R}(\omega,-))$
preserves epimorphisms and the double functor $\Ext_{R}^i(\omega, \Tor_i^{S}(\omega,-))$
preserves monomorphisms in terms of the (strong) cograde conditions of modules and some related properties
of the cotranspose of modules, and also investigate the relationship between certain cograde conditions
of modules and complete cotorsion pairs. This paper is organized as follows.

In Section 2, we give some terminology and some preliminary results.

Let $R$ and $S$ be rings and $_R\omega_S$ a semidualizing bimodule. In Section 3,
we study when $\Tor^S_i(\omega, \Ext^i_{R}(\omega,-))$ preserves epimorphisms and $\Ext_{R}^i(\omega, \Tor_i^{S}(\omega,-))$
preserves monomorphisms in terms of the (strong) cograde conditions of modules. Let $n,k\geqslant 0$.
We prove that the $\Tor$-cograde of $\Ext_R^{i+k}(\omega,M)$ with respect to $\omega$ is at least $i$
for any left $R$-module $M$ and $1\leqslant i\leqslant n$ if and only if $\Tor^S_i(\omega, \Ext^i_{R}(\omega,f))$
is an epimorphism for any epimorphism of left $R$-modules $f: B\twoheadrightarrow C$ with $B,C$ being a $(k+1)$-cosyzygy
and $0\leqslant i\leqslant n-1$ (Theorem~ \ref{thm-3.5}); and that the $\Ext$-cograde of $\Tor^S_{i+k}(\omega,N)$
with respect to $\omega$ is at least $i$ for any left $S$-module $N$ and $1\leqslant i\leqslant n$ if and only if
$\Ext_{R}^i(\omega, \Tor_i^{S}(\omega,g))$ is a monomorphism for any monomorphism of left $S$-modules $g: B'\rightarrowtail C'$
with $B',C'$ being a $(k+1)$-yoke and $0\leqslant i\leqslant n-1$ (Theorem~ \ref{thm-3.7}).

Moreover, we prove that the strong $\Tor$-cograde of $\Ext_R^{i+k}(\omega,M)$ with respect to $\omega$ is at least $i$
for any left $R$-module $M$ and $1\leqslant i\leqslant n$ if and only if for any exact sequence of left $R$-modules
$0\to A\to B\stackrel{f}{\longrightarrow}C\to 0$ with $A$ an $(i-1)$-$\mathcal{P}_{\omega}(R)$-syzygy of an
$(i+k-1)$-cosyzygy, $\Tor^S_i(\omega, \Ext^i_{R}(\omega,f))$ is an epimorphism for any $0\leqslant i\leqslant n-1$
(Theorem~ \ref{thm-3.8}); and that the strong $\Ext$-cograde of $\Tor^S_{i+k}(\omega,N)$ with respect to $\omega$
is at least $i$ for any left $S$-module $N$ and $1\leqslant i\leqslant n$ if and only if for any exact sequence
of left $S$-modules $0\to A\stackrel{g}{\rightarrow} B\to C\to 0$ with $C$ an $(i-1)$-$\mathcal{I}_{\omega}(S)$-cosyzygy
of an $(i+k-1)$-yoke, $\Ext_{R}^i(\omega, \Tor_i^{S}(\omega,g))$ is a monomorphism for any $0\leqslant i\leqslant n-1$
(Theorem~ \ref{thm-3.9}).

In Section 4, we introduce the notion of $\omega$ satisfying the (quasi) $n$-cograde condition in terms of
the properties of the strong cograde of modules. By using the results obtained in Section 3, we give some
equivalent characterizations for $\omega$ satisfying such conditions (Theorems \ref{thm-4.8} and \ref{thm-4.14}).
In particular, the $n$-cograde condition is left-right symmetric, but the quasi $n$-cograde condition is not.
In addition, we prove that the $\Tor$-cograde of $\Ext^i_R(\omega,M)$ with respect to $\omega$ is at least
$i-1$ for any $M\in \Mod R$ and $1\leqslant i\leqslant n$ if and only if the $\Ext$-cograde of $\Tor_{i}^{S}(\omega,N)$
with respect to $\omega$ is at least $i-1$ for any $N\in \Mod S$ and $1\leqslant i\leqslant n$ (Theorem~ \ref{thm-4.19}).

In Section 5, we prove that if one of the equivalent conditions in Theorem~ \ref{thm-4.19} mentioned above
is satisfied, then the right $S$-projective dimension $\pd_{S^{op}}\omega$ of $\omega$ is at most $n-1$ if and only if
$(\mathcal{P}_\omega\text{-}\id^{\leqslant n-1}(R), {_R\omega^{\perp_{n}}})$ forms a complete cotorsion pair;
and the left $R$-projective dimension $\pd_{R}\omega$ of $\omega$ is at most $n-1$ if and only if
$({{\omega_S}^{\top_{n}}}, \mathcal{I}_\omega\text{-}\pd^{\leqslant n-1}(S))$ forms a complete cotorsion pair
(Theorem~ \ref{thm-5.6}); see Section 2 and 5 for the details of these notations. Then we apply this result to
right quasi $(n-1)$-Gorenstein artin algebras (Corollary ~\ref{cor-5.8}).

In Section 6, we introduce the finitistic $\mathcal{P}_\omega(R)$-injective dimension ${\rm F}\mathcal{P}_{\omega}{\text -}\id R$
of $R$ and the $\mathcal{I}_\omega(S)$-projective dimension ${\rm F}\mathcal{I}_{\omega}{\text -}\pd S$
of $S$. We prove that if the $\Tor$-cograde of $\Ext^{i+k}_R(\omega,M)$
with respect to $\omega$ is at least $i$ for any $M\in \Mod R$ and $i\geqslant 1$,
then ${\rm F}\mathcal{P}_{\omega}{\text -}\id R\leqslant \pd_R\omega\leqslant{\rm F}\mathcal{P}_{\omega}{\text -}\id R+k$;
and if the $\Ext$-cograde of $\Tor_{i+k}^S(\omega,N)$ with respect to $\omega$ is at least $i$ for any $N\in \Mod S$ and $i\geqslant 1$,
then ${\rm F}\mathcal{I}_{\omega}{\text -}\pd S\leqslant \pd_{S^{op}}\omega\leqslant{\rm F}\mathcal{I}_{\omega}{\text -}\pd S+k$
(Theorem~ \ref{thm-6.3}). As an application, we get that for an artin algebra $R$,
if $R$ satisfies the Auslander condition, then $\FPD R^{op}=\FID R^{op}={\id_{R^{op}}R}=\id_RR=\FPD R=\FID R$;
and if $R$ satisfies the right quasi Auslander condition, then
$\FPD R\leqslant\FID R={\id_{R^{op}}R}=\id_{R}R\leqslant\FPD R+1$,
where $\FID R$, $\FPD R$, ${\id_{R^{op}}R}$ and $\id_RR$ are the finitistic injective dimension,
the finitistic projective dimension, the right and left self-injective dimensions of $R$ respectively
(Corollary ~\ref{cor-6.9}).

\section {\bf Preliminaries}
Throughout this paper, all rings are associative rings with units. For a ring $R$, $\Mod R$ (resp. $\mod R$) are the class of left
(resp. finitely generated left) $R$-modules. Let $M\in \Mod R$, we use $\Add_RM$ to denote the subclass of
$\Mod R$ consisting of modules consisting of direct summands of direct sums of copies of $M$,
and use $\pd_RM$, $\fd_RM$ and $\id_RM$ to denote the projective, flat and injective dimensions of $M$ respectively.

\begin{definition} \label{def-2.1} {\rm (\cite{ATY, HW}).
Let $R$ and $S$ be rings. An ($R$-$S$)-bimodule $_R\omega_S$ is called
\textit{semidualizing} if the following conditions are satisfied.
\begin{enumerate}
\item[(a1)] $_R\omega$ admits a degreewise finite $R$-projective resolution.
\item[(a2)] $\omega_S$ admits a degreewise finite $S$-projective resolution.
\item[(b1)] The homothety map $_RR_R\stackrel{_R\gamma}{\rightarrow} \Hom_{S^{op}}(\omega,\omega)$ is an isomorphism.
\item[(b2)] The homothety map $_SS_S\stackrel{\gamma_S}{\rightarrow} \Hom_{R}(\omega,\omega)$ is an isomorphism.
\item[(c1)] $\Ext_{R}^{\geqslant 1}(\omega,\omega)=0.$
\item[(c2)] $\Ext_{S^{op}}^{\geqslant 1}(\omega,\omega)=0.$
\end{enumerate}}
\end{definition}

Wakamatsu in \cite{W1} introduced and studied the so-called {\it generalized tilting modules},
which are usually called {\it Wakamatsu tilting modules}, see \cite{BR, MR}. Note that
a bimodule $_R\omega_S$ is semidualizing if and only if it is Wakamatsu tilting (\cite[Corollary 3.2]{W3}).
Examples of semidualizing bimodules are referred to \cite{HW,S,TH2,TH4,W2}.

From now on, $R$ and $S$ are arbitrary rings and we fix a semidualizing bimodule $_R\omega_S$.
For convenience, We write
$$(-)_*:=\Hom(\omega,-)\ \text{and}\ (-)^*:=\Hom(-,\omega),$$
$${_R\omega^{\perp}}:=\{M\in\Mod R\mid\Ext_R^{\geqslant 1}(\omega,M)=0\},$$
$${{\omega_S}^{\top}}:=\{N\in\Mod S\mid\Tor^S_{\geqslant 1}(\omega,N)=0\}.$$
For any $n\geqslant 1$, we write
$${_R\omega^{\perp_n}}:=\{M\in\Mod R\mid\Ext_R^{1\leqslant i\leqslant n}(\omega,M)=0\},$$
$${{\omega_S}^{\top_n}}:=\{N\in\Mod S\mid\Tor^S_{1\leqslant i\leqslant n}(\omega,N)=0\};$$
in particular, $_R\omega^{\perp_0}=\Mod R$ and ${\omega_S}^{\top_0}=\Mod S$. Symmetrically,
${{\omega_S}^{\perp_n}}$ and ${{_R\omega}^{\top_n}}$ are defined.
Following \cite{HW}, set
$$\mathcal{F}_\omega(R):=\{\omega\otimes_SF\mid F\ {\rm \ is\ flat\ in}\ \Mod S\},$$
$$\mathcal{P}_\omega(R):=\{\omega\otimes_SP\mid P\ {\rm \ is\ projective\ in}\ \Mod S\},$$
$$\mathcal{I}_\omega(S):=\{I_*\mid I\ {\rm \ is\ injective\ in}\ \Mod R\}.$$
The modules in $\mathcal{F}_\omega(R)$, $\mathcal{P}_\omega(R)$ and $\mathcal{I}_\omega(S)$ are called \textit{$\omega$-flat},
\textit{$\omega$-projective} and \textit{$\omega$-injective} respectively.
Note that $\mathcal{P}_\omega(R)=\Add_R\omega$ (\cite[Proposition 3.4(2)]{TH2}).
Symmetrically, the classes of $\mathcal{F}_\omega(S^{op})$,
$\mathcal{P}_\omega(S^{op})$ and $\mathcal{I}_\omega(R^{op})$ are defined.

Let $M\in \Mod R$ and $N\in \Mod S$. Then we have the following two canonical valuation homomorphisms
$$\theta_M:\omega\otimes_SM_*\rightarrow M$$ defined by $\theta_M(x\otimes f)=f(x)$
for any $x\in \omega$ and $f\in M_*$; and
$$\mu_N: N\rightarrow (\omega\otimes_SN)_*$$ defined by $\mu_N(y)(x)=x\otimes y$
for any $y\in N$ and $x\in \omega$. Recall that a module $M\in \Mod R$ is called {\it $\omega$-cotorsionless}
(resp. {\it $\omega$-coreflexive}) if $\theta_M$ is an epimorphism (resp. an isomorphism) (\cite{TH1});
and a module $N\in \Mod S$ is called {\it adjoint $\omega$-cotorsionless}
(resp. {\it adjoint $\omega$-coreflexive}) if $\mu_N$ is a monomorphism (resp. an isomorphism) (\cite{TH3}).

\begin{definition} \label{def-2.2}(\cite{HW}).
\begin{enumerate}
\item[(1)] The {\it Auslander class} $\mathcal{A}_{\omega}(S)$ with respect to $\omega$ consists of all left $S$-modules $N$
satisfying the following conditions.
\begin{enumerate}
\item[(A1)] $N\in{{\omega_S}^{\top}}$.
\item[(A2)] $\omega\otimes_SN\in{_R\omega^{\perp}}$.
\item[(A3)] $\mu_{N}$ is an isomorphism in $\Mod S$.
\end{enumerate}
\end{enumerate}
\begin{enumerate}
\item[(2)] The {\it Bass class} $\mathcal{B}_\omega(R)$ with respect to $\omega$ consists of all left $R$-modules $M$
satisfying the following conditions.
\begin{enumerate}
\item[(B1)] $M\in{_R\omega^{\perp}}$.
\item[(B2)] $M_*\in{{\omega_S}^{\top}}$.
\item[(B3)] $\theta_M$ is an isomorphism in $\Mod R$.
\end{enumerate}
\end{enumerate}
\end{definition}

For a module $M\in \Mod R$, we use
$$0\rightarrow M \rightarrow I^0(M) \stackrel{g^0}{\longrightarrow} I^1(M)\stackrel{g^1}{\longrightarrow}\cdots
\stackrel{g^{i-1}}{\longrightarrow} I^i(M)\stackrel{g^{i}}{\longrightarrow} \cdots\eqno{(2.1)}$$
to denote the minimal injective resolution of $M$. For any $n\geqslant 1$,
$\coOmega^n(M):= \im g^{n-1}$ is called the {\it $n$-cosyzygy} of $M$; in particular, $\coOmega^{0}(M)=M$.
We use $\coOmega^n(R)$ to denote the subclass of $\Mod R$ consisting of $n$-cosyzygy modules.
Symmetrically, $\coOmega^n(S^{op})$ is defined.

\begin{definition}\label{def-2.3} (\cite{TH1}).
Let $M\in \Mod R$ and $n\geqslant 1$.
\begin{enumerate}
\item $\cTr_\omega M:=\Coker({g^0}_*)$ is called the \emph{cotranspose} of $M$ with respect to $_R\omega_S$, where $g^0$ is as in (2.1).
\item $M$ is called \emph{$n$-$\omega$-cotorsionfree} if $\cTr_\omega M\in{{\omega_S}^{\top_n}}$; and is called
\emph{$\infty$-$\omega$-cotorsionfree} if it is $n$-$\omega$-cotorsionfree for all $n$.
\end{enumerate}
\end{definition}

We use $\cT^n_{\omega}(R)$ (resp. $\cT_{\omega}(R)$) to denote the subclass of $\Mod R$ consisting of
$n$-$\omega$-cotorsionfree modules (resp. $\infty$-$\omega$-cotorsionfree modules).
Symmetrically, $\cT^n_{\omega}(S^{op})$ is defined. By \cite[Proposition 3.2]{TH1},
we have that a module in $\Mod R$ is $\omega$-cotorsionless (resp. $\omega$-coreflexive) if and only if it is in
$\cT^1_{\omega}(R)$ (resp. $\cT^2_{\omega}(R)$).

Recall from \cite{E} that a homomorphism $f:F\rightarrow N$ in $\Mod S$ with $F$ flat
is called a \textit{flat cover} of $N$ if $\Hom_S(F',f)$ is epic for any flat module $F'$ in $\Mod S$,
and an endomorphism $h:F\rightarrow F$ is an automorphism whenever $f=fh$.
Let $N\in \Mod S$. Bican, El Bashir and Enochs proved in \cite{BBE} that $N$ has a flat cover. We use
$$\cdots \buildrel {f_{n}} \over \longrightarrow F_n(N) \buildrel {f_{n-1}} \over \longrightarrow \cdots
\buildrel {f_1} \over \longrightarrow F_1(N) \buildrel {f_0} \over \longrightarrow F_0(N)
\rightarrow N \to 0 \eqno{(2.2)}$$
to denote the minimal flat resolution of $N$ in $\Mod S$, where each $F_i(N)\to\Coker f_i$ is a flat cover of
$\Coker f_i$. For any $n\geqslant 1$,
$\Omega^n_{\mathcal{F}}(N):= \im f_{n-1}$ is called the {\it $n$-yoke} of $N$; in particular, $\Omega^{0}_{\mathcal{F}}(N)=N$.
We use $\Omega^n_{\mathcal{F}}(S)$ to denote the subclass of $\Mod S$ consisting of $n$-yoke modules.
Symmetrically, $\Omega^n_{\mathcal{F}}(R^{op})$ is defined.

\begin{definition} \label{def-2.4} {\rm (\cite{TH3})
Let $N\in \Mod S$ and $n\geqslant 1$.
\begin{enumerate}
\item $\acTr_{\omega} N:=\Ker(1_\omega\otimes f_0)$ is called the \emph{adjoint cotranspose} of $N$
with respect to $_{R}\omega_{S}$, where $f_0$ is as in (2.2).
\item $N$ is called \emph{adjoint $n$-$\omega$-cotorsionfree} if $\acTr_{\omega}N\in{_R\omega^{\perp_n}}$; and is called
\emph{adjoint $\infty$-$\omega$-cotorsionfree} if it is adjoint $n$-$\omega$-cotorsionfree for all $n$.
\end{enumerate}}
\end{definition}

We use $\acT^n_{\omega}(S)$ (resp. $\acT{\omega}(S)$) to denote the subclass of $\Mod S$ consisting of
adjoint $n$-$\omega$-cotorsionfree modules (resp. adjoint $\infty$-$\omega$-cotorsionfree modules).
Symmetrically, $\acT^n_{\omega}(R^{op})$ is defined.
By \cite[Proposition 3.2]{TH3}, we have that a module in $\Mod S$ is adjoint $\omega$-cotorsionless
(resp. adjoint $\omega$-coreflexive) if and only if it is in $\acT^1_{\omega}(S)$ (resp. $\acT^2_{\omega}(S)$).

\begin{definition} (\cite{TH2}) \label{def-2.5}
{\rm \begin{enumerate}
\item Let $M\in \Mod R$ and $n\geqslant 0$. The \emph{$\Ext$-cograde} of $M$ with respect to $\omega$ is defined as
$\Ecograde_{\omega}M:=\inf\{i\geqslant 0\mid\Ext_{R}^i(\omega,M)\neq 0\}$; and {\it the strong $\Ext$-cograde} of $M$ with respect to $\omega$,
denoted by $\sEcograde_{\omega}M$, is said to be at least $n$ if $\Ecograde_{\omega}X\geqslant n$ for any quotient module $X$ of $M$.
Symmetrically, the (\emph{strong}) \emph{$\Ext$-cograde} of a module in $\Mod S^{op}$ is defined.
\item Let $N\in \Mod S$ and $n\geqslant 0$. The \emph{$\Tor$-cograde} of $N$ with respect to $\omega$ is defined as
$\Tcograde_{\omega}N:=\inf\{i\geqslant 0\mid\Tor^{S}_i(\omega,N)\neq 0\}$; and {\it the strong $\Tor$-cograde} of $N$ with respect to $\omega$,
denoted by $\sTcograde_{\omega}N$, is said to be at least $n$ if $\Tcograde_{\omega}Y\geqslant n$ for any submodule $Y$ of $N$.
Symmetrically, the (\emph{strong}) \emph{$\Tor$-cograde} of a module in $\Mod R^{op}$ is defined.
\end{enumerate}}
\end{definition}





Let $\mathcal{X}$ be a subclass of $\Mod R$ and $M\in \Mod R$. An exact sequence (of finite or infinite length):
$$\cdots \to X_n \to \cdots \to X_1 \to X_0 \to M \to 0$$
in $\Mod R$ is called an \emph{$\mathcal{X}$-resolution} of
$M$ if all $X_i$ are in $\mathcal{X}$. The \emph{$\mathcal{X}$-projective
dimension} $\mathcal{X}$-$\pd_RM$ of $M$ is defined as
$\inf\{n\mid$ there exists an $\mathcal{X}$-resolution
$$0 \to X_n \to \cdots \to X_1\to X_0 \to M\to 0$$
of $M$ in $\Mod R\}$. Dually, the notions of an \emph{$\mathcal{X}$-coresolution}
and the \emph{$\mathcal{X}$-injective dimension} $\mathcal{X}$-$\id_RM$ of $M$ are defined.


Let $\mathcal{F}$ be a subclass of $\Mod R$. A module $M\in\Mod R$ is said to have \emph{special $\mathcal{F}$-precover}
if there exists an exact sequence
$$0\to K \to F \to M \to 0$$ in $\Mod R$ with $F\in\mathcal{F}$ and $\Ext^1_R(F',K)=0$ for any $F'\in\mathcal{F}$.
The class $\mathcal{F}$ is called \emph{special precovering} if any module in $\Mod R$ has a special $\mathcal{F}$-precover.
Dually, the notions of \emph{special $\mathcal{F}$-preenvelopes} and \emph{special preenveloping classes} are defined (see \cite{EJ}).

\begin{definition}\label{def-2.6} {\rm (cf. \cite{GT})}
Let $\mathcal{U},\mathcal{V}$ be subclasses of $\Mod R$. The pair $(\mathcal{U},\mathcal{V})$ is called a \textit{cotorsion pair}
if $\mathcal{U}={^{\bot_1}\mathcal{V}}:=\{U\in\Mod R\mid \Ext_R^1(U,V)=0$ for any $V\in\mathcal{V}\}$ and
$\mathcal{V}={\mathcal{U}^{\bot_1}}:=\{V\in\Mod R\mid \Ext_R^1(U,V)=0$ for any $U\in\mathcal{U}\}$.
\end{definition}

The following is the Salce's lemma.

\begin{lemma}\label{lem-2.7} {\rm (cf. \cite[Lemma 2.2.6]{GT})}
Let $(\mathcal{U},\mathcal{V})$ be a cotorsion pair in $\Mod R$. Then the following statements are equivalent.
\begin{enumerate}
\item Any module in $\Mod R$ has a special $\mathcal{U}$-precover.
\item Any module in $\Mod R$ has a special $\mathcal{V}$-preenvelope.
\end{enumerate}
In this case, the cotorsion pair $(\mathcal{U},\mathcal{V})$ is called \emph{complete}.
\end{lemma}

\begin{definition}\label{def-2.8}
Let $\mathcal{X}$ be a subcategory of an abelian category $\mathcal{E}$ and $n\geqslant 1$. If there exists an exact sequence
$$0 \rightarrow N \rightarrow X_{0}\rightarrow \cdots \rightarrow X_{n-1}\to M \to 0$$
in $\mathcal{E}$ with all $X_i$ in $\mathcal{X}$, then $N$ is called an {\it $n$-$\mathcal{X}$-syzygy} of $M$
and $M$ is called an {\it $n$-$\mathcal{X}$-cosyzygy} of $N$.
\end{definition}

For subcategories $\mathcal{X},\mathcal{Y}$ of an abelian category $\mathcal{E}$ and $n\geqslant 1$, we write
$$\Omega^n_{\mathcal{X}}(\mathcal{Y}):=\{N\in\mathcal{A}\mid N\ \text{is\ an}\ n\text{-}\mathcal{X}\text{-syzygy\ of\ some\ object\ in}\ \mathcal{Y}\},$$
$$\coOmega^n_{\mathcal{X}}(\mathcal{Y}):=\{M\in\mathcal{A}\mid M\ \text{is\ an}\ n\text{-}\mathcal{X}\text{-cosyzygy\ of\ some\ object\ in}\ \mathcal{Y}\}.$$
In particular,
$\Omega^0_{\mathcal{X}}(\mathcal{Y})=\mathcal{Y}=\coOmega^0_{\mathcal{X}}(\mathcal{Y})$ and
$\Omega^{-1}_{\mathcal{X}}(\mathcal{Y})=0=\coOmega^{-1}_{\mathcal{X}}(\mathcal{Y})$.
For convenience, we write
$$\Omega^{n}_{\mathcal{A}}(S):=\Omega^{n}_{\mathcal{A}_{\omega}(S)}(\Mod S),\ \Omega^{n}_{\mathcal{I}_{\omega}}(S):=\Omega^{n}_{\mathcal{I}_{\omega}(S)}(\Mod S),$$
$$\Omega^{n}_{\mathcal{I}_{\omega}}(R^{op}):=\Omega^{n}_{\mathcal{I}_{\omega}(R^{op})}(\Mod R^{op}),$$
$$\coOmega^{n}_{\mathcal{B}}(R):=\coOmega^{n}_{\mathcal{B}_{\omega}(R)}(\Mod R),\ \coOmega^{n}_{\mathcal{F}_{\omega}}(R):=\coOmega^{n}_{\mathcal{F}_{\omega}(R)}(\Mod R),$$
$$\coOmega^{n}_{\mathcal{P}_{\omega}}(R):=\coOmega^{n}_{\mathcal{P}_{\omega}(R)}(\Mod R),\ \coOmega^{n}_{\mathcal{P}_{\omega}}(S^{op}):
=\coOmega^{n}_{\mathcal{P}_{\omega}(S^{op})}(\Mod S^{op}).$$

\begin{lemma}\label{lem-2.9}
We have
\begin{enumerate}
\item $\Omega^{1}_{\mathcal{I}_{\omega}}(S)=\acT^1_{\omega}(S)$.
\item $\coOmega^{1}_{\mathcal{P}_{\omega}}(R)=\cT^1_{\omega}(R)$.
\end{enumerate}
\end{lemma}

\begin{proof}
(1) By \cite[Proposition 3.8]{TH3}, we have $\acT^1_{\omega}(S)\subseteq\Omega^{1}_{\mathcal{I}_{\omega}}(S)$.
Now let $N\in \Omega^{1}_{\mathcal{I}_{\omega}}(S)$ and let
$f^0: N\rightarrowtail I^0$ be a monomorphism in $\Mod S$ with $I^0\in \mathcal{I}_\omega(S)$.
Then from the following commutative diagram
$$\xymatrix{N \ar@{{>}->}[r]^{f^0}\ar[d]^{\mu_{N}}
& I^0 \ar[d]^{\mu_{I^0}} \\
(\omega\otimes_S N)_* \ar[r]^{(1_\omega\otimes f^0)_*} & (\omega\otimes_S I^0)_*.}$$
with $\mu_{I^0}$ an isomorphism, we get that $\mu_N$ is a monomorphism and
$N\in \acT^1_{\omega}(S)$. It implies $\Omega^{1}_{\mathcal{I}_{\omega}}(S)\subseteq\acT^1_{\omega}(S)$.

(2) By \cite[Proposition 3.7]{TH1}, we have $\cT^1_{\omega}(R)\subseteq\coOmega^{1}_{\mathcal{P}_{\omega}}(R)$.
Now let $M\in \coOmega^{1}_{\mathcal{P}_{\omega}}(R)$ and let
$f_0:W_0\twoheadrightarrow M$ be an epimorphism in $\Mod R$ with $W_0\in \mathcal{P}_\omega(R)$.
Then from the following commutative diagram
$$\xymatrix{\omega\otimes_S{W_0}_* \ar[r]^{1_{\omega}\otimes {f_0}_*}\ar[d]^{\theta_{W_0}}
& \omega\otimes_SM_* \ar[d]^{\theta_M} \\
W_0 \ar@{>>}[r]^{f_0} & M}$$
with $\theta_{W_0}$ an isomorphism, we get that $\theta_M$ is an epimorphism and
$M\in \cT^1_{\omega}(R)$. It implies $\coOmega^{1}_{\mathcal{P}_{\omega}}(R)\subseteq\cT^1_{\omega}(R)$.
\end{proof}

Let $\mathcal{C},\mathcal{E}$ be abelian categories and $\Delta:\mathcal{C}\to\mathcal{E}$ a functor.
Recall that a sequence $\mathbb{T}$ in $\mathcal{C}$ is called {\it $\Delta$-exact} if $\Delta(\mathbb{T})$
is exact in $\mathcal{E}$.

\section{\bf (Strong) cograde conditions and double homological functors}

In this section, we study when $\Tor^S_i(\omega, \Ext^i_{R}(\omega,-))$ preserves epimorphisms
and $\Ext_{S^{op}}^i(\omega, \Tor_i^{R}(-,\omega))$ preserves monomorphisms in terms of the (strong)
cograde conditions of modules.

\vspace{0.2cm}

\noindent{\bf 3.1. Cograde conditions}

\vspace{0.2cm}

We begin with the following

\begin{lemma} \label{lem-3.1}
\begin{enumerate}
\item[]
\item Let $M\in \Mod R$ with the minimal injective resolution as (2.1). Then there exists an exact sequence
$$0\to \Ext^1_R(\omega,M)\stackrel{\lambda}{\longrightarrow} \cTr_\omega M \stackrel{\pi}{\longrightarrow} I^1(M)_*/\coOmega^{1}(M)_*\to 0 \eqno{(3.1)}$$
in $\Mod S$ such that $1_{\omega}\otimes \pi$ is an isomorphism.
\item Let $N\in \Mod S$ with the minimal flat resolution as (2.2). Then there exists an exact sequence
$$0\to \im(1_{\omega}\otimes f_1)\stackrel{\sigma}{\longrightarrow} \acTr_\omega N \stackrel{\tau}{\longrightarrow}
\Tor^S_1(\omega,N)\to 0 \eqno{(3.2)}$$
in $\Mod R$ such that $\sigma_*$ is an isomorphism.
\end{enumerate}
\end{lemma}

\begin{proof}
(1) Let $g^0=\alpha\cdot\beta$ (where $\beta:I^0(M)\twoheadrightarrow \coOmega^1(M)(=\im g^0)$ and $\alpha:\coOmega^1(M)\rightarrowtail I^1(M)$)
be the natural epic-monic decomposition of $g^0$. Then we have the following commutative diagram with exact columns and rows
$$\xymatrix{  &  &  & 0 \ar[d] &0\ar@{-->}[d] &  \\
0 \ar[r] & M_* \ar[r] \ar@{=}[d] & I^0(M)_*  \ar[r]^{\beta_*} \ar@{=}[d] & \coOmega^{1}(M)_* \ar[r] \ar[d]^{\alpha_*} &\Ext^1_R(\omega,M)\ar[r]\ar@{-->}[d]^{\lambda} & 0 \\
0 \ar[r] & M_* \ar[r] & I^0(M)_*  \ar[r]^{{g^0}_*}  & I^1(M)_* \ar[r]^{\gamma}\ar[d]^{\pi_1}& \cTr_\omega M\ar[r]\ar@{-->}[d]^{\pi} & 0\\
&  &  & C\ar[d]\ar@{=}[r] & C\ar@{-->}[d]\\
&  &  & 0 & 0}$$
in $\Mod S$, where $C=I^1(M)_*/\coOmega^{1}(M)_*$, $\pi_1$ is the natural epimorphism, $\lambda$ and $\pi$ are induced homomorphisms. The rightmost column
in the above diagram is exactly the exact sequence (3.1). Notice that
$$0\to \coOmega^{1}(M)_*\stackrel{\alpha_*}{\longrightarrow} I^1(M)_* \stackrel{{g^1}_*}{\longrightarrow} I^2(M)_*$$
is exact, so there exists a homomorphism $\delta: C\to I^2(M)_*$ in $\Mod S$ such that ${g^1}_*=\delta\cdot\pi_1$, and hence
${g^1}_*=\delta\cdot\pi_1=\delta\cdot\pi\cdot \gamma$.

By \cite[Lemma 4.1]{HW}, for any injective module $I\in \Mod R$, we have $\omega\otimes_SI_*\cong I$ canonically.
So the upper row in the following commutative diagram
$$\xymatrix{ \omega\otimes_SI^0(M)_* \ar[r]^{1_\omega\otimes {g^0}_*}  &  \omega\otimes_SI^1(M)_*
\ar[r]^{1_\omega\otimes {g^1}_*} \ar@{>>}[d]^{1_{\omega}\otimes\gamma}&  \omega\otimes_SI^2(M)_* \\
 &  \omega\otimes_S\cTr_\omega M \ar@{>>}[r]^{1_\omega\otimes \pi} & \omega\otimes_SC.  \ar[u]_{1_{\omega}\otimes\delta} }$$
is exact. Let $x\in \Ker (1_{\omega}\otimes \pi)$. Then there exists $y\in \omega\otimes_SI^1(M)_*$ such that $x=(1_{\omega}\otimes\gamma)(y)$.
It follows that $$(1_\omega\otimes {g^1}_*)(y)=(1_{\omega}\otimes\delta)\cdot (1_\omega\otimes \pi)\cdot (1_{\omega}\otimes\gamma)(y)
=(1_{\omega}\otimes\delta)\cdot (1_\omega\otimes \pi)(x)=0.$$ So $y\in\Ker (1_\omega\otimes {g^1}_*)=\im (1_\omega\otimes {g^0}_*)$,
and hence there exists $z\in \omega\otimes_SI^0(M)_*$ such that $y=(1_\omega\otimes {g^0}_*)(z)$. Thus
$$x=(1_{\omega}\otimes\gamma)(y)=(1_{\omega}\otimes\gamma)\cdot(1_\omega\otimes {g^0}_*)(z)=(1_{\omega}\otimes(\gamma\cdot {g^0}_*))(z)=0,$$
which implies that $1_{\omega}\otimes \pi$ is a monomorphism, and hence an isomorphism.

(2) Let $f_0=\alpha'\cdot\beta'$ (where $\beta':F_1(N)\twoheadrightarrow \Omega^1_{\mathcal{F}}(N)(=\im f_0)$
and $\alpha':\Omega^1_{\mathcal{F}}(N)\rightarrowtail F_0(N)$)
be the natural epic-monic decomposition of $f_0$. Then we have the following commutative diagram with exact columns and rows
$$\xymatrix{  &  0 \ar@{-->}[d] & 0\ar[d] &  & &  \\
&  \im(1_{\omega}\otimes f_1) \ar@{=}[r]\ar@{-->}[d]^{\sigma} & \im(1_{\omega}\otimes f_1)\ar[d]^{\sigma_1} &  & &  \\
0 \ar[r] & \acTr_\omega N \ar[r]^{\eta} \ar@{-->}[d]^{\tau} & \omega\otimes_S F_1(N) \ar[r]^{1_{\omega}\otimes f_0}
\ar[d]^{1_{\omega}\otimes\beta'} & \omega\otimes_SF_0(N) \ar[r] \ar@{=}[d] &\omega\otimes_SN\ar[r]\ar@{=}[d] & 0 \\
0 \ar[r] & \Tor_1^S(\omega,N) \ar[r]\ar@{-->}[d] & \omega\otimes_S\Omega_{\mathcal{F}}^1(N)\ar[r]^{1_{\omega}\otimes\alpha'}
\ar[d]  & \omega\otimes_SF_0(N) \ar[r]^{\gamma}& \omega\otimes_SN\ar[r] & 0\\
& 0 & 0 &  & }$$
in $\Mod R$, where $\sigma$ and $\tau$ are induced homomorphisms. The leftmost column
in the above diagram is exactly the exact sequence (3.2). Notice that
$$\omega\otimes_SF_2(N)\stackrel{1_{\omega}\otimes f_1}{\longrightarrow}
\omega\otimes_SF_1(N)\stackrel{1_{\omega}\otimes\beta'}{\longrightarrow}\omega\otimes_S\Omega_{\mathcal{F}}^1(N) \to 0$$
is exact, so there exists a homomorphism $\phi: \omega\otimes_SF_2(N)\to \im(1_{\omega}\otimes f_1)$
in $\Mod R$ such that $1_{\omega}\otimes f_1=\sigma_1\cdot\phi$, and hence
$1_{\omega}\otimes f_1=\sigma_1\cdot\phi=\eta\cdot\sigma\cdot\phi$.

By \cite[Lemma 4.1]{HW}, for any flat module $F\in \Mod S$, we have $F\cong(\omega\otimes_SF)_*$ canonically.
So the upper row in the following commutative diagram is exact.
$$\xymatrix{ (\omega\otimes_SF_2(N))_* \ar[r]^{(1_{\omega}\otimes f_1)_*} \ar[d]^{\phi_*}
&  (\omega\otimes_SF_1(N))_* \ar[r]^{(1_{\omega}\otimes f_0)_*} &  (\omega\otimes_SF_0(N))_*  \\
(\im(1_{\omega}\otimes f_1))_* \ar@{{>}->}[r]^{\sigma_*} &  (\acTr_\omega N)_*. \ar@{{>}->}[u]_{\eta_*} }$$
Let $x\in (\acTr_\omega N)_*$. Since $((1_{\omega}\otimes f_0)_* \cdot\eta_*)(x)=(((1_{\omega}\otimes f_0)\cdot \eta)_*)(x)=0$,
we have that $\eta_*(x)\in \Ker(1_{\omega}\otimes f_0)_*=\im(1_{\omega}\otimes f_1)_*$ and
there exists $y\in (\omega\otimes_SF_2(N))_*$ such that $\eta_*(x)=(1_{\omega}\otimes  f_1)_*(y)$. Thus
$$\eta_*(x)=(1_{\omega}\otimes  f_1)_*(y)=(\eta_*\cdot\sigma_*\cdot \phi_*)(y).$$ As $\eta^*$ is monic,
we have $x=\sigma_*(\phi_*(y))$. It means that $\sigma_*$ is an epimorphism, and hence an isomorphism.
\end{proof}

The following two lemmas are useful in this section.

\begin{lemma} \label{lem-3.2}
Assume that $\coOmega^n(R)\subseteq \cT^m_{\omega}(R)$ with $m,n\geqslant 0$.
Then the following statements are equivalent.
\begin{enumerate}
\item $\Tcograde_{\omega}\Ext^{n+1}_{R}(\omega,M)\geqslant m$ for any $M\in \Mod R$.
\item $\coOmega^{n+1}(R)\subseteq \cT^{m+1}_{\omega}(R)$.
\end{enumerate}
\end{lemma}

\begin{proof}
Because any injective module in $\Mod R$ is in $\cT^{1}_{\omega}(R)$ by \cite[Lemma 2.5(2)]{TH1}, we have that
$\coOmega^{n+1}(R)\subseteq \cT^{1}_{\omega}(R)$ for any $n\geqslant 0$, and the case for $m=0$ follows.
Now suppose that $m\geqslant 1$ and $M\in \Mod R$.
By Lemma ~\ref{lem-3.1}(1), there exists an exact sequence
$$0\to \Ext^1_R(\omega,\coOmega^n(M))\stackrel{\lambda}{\longrightarrow} \cTr_\omega \coOmega^n(M) \stackrel{\pi}{\longrightarrow} C\to 0$$
in $\Mod S$ such that $1_{\omega}\otimes \pi$ is an isomorphism, where $C={I^{n+1}(M)}_*/{\coOmega^{n+1}(M)}_*$.
Because $\coOmega^n(R)\subseteq \cT^m_{\omega}(R)$ by assumption, we have that both
$\cTr_\omega \coOmega^n(M)$ and $\cTr_\omega \coOmega^{n+1}(M)$ are in ${{\omega_S}^{\top_m}}$.
It yields that
$$\Tor_{i}^S(\omega, \Ext^{n+1}_R(\omega,M))\cong \Tor_{i}^S(\omega, \Ext^1_R(\omega,\coOmega^n(M)))\cong \Tor_{i+1}^S(\omega, C)$$
for any $0\leqslant i\leqslant m-1$. In addition, we also have an exact sequence
$$0\to C\to {I^{n+2}(M)}_*\to \cTr_\omega\coOmega^{n+1}(M)\to 0$$
in $\Mod S$. By \cite[Corollary 6.1]{HW}, we have ${I^{n+2}(M)}_*\in {{\omega_S}^{\top}}$.
So $$\Tor_{i}^S(\omega, \Ext^{n+1}_R(\omega,M))\cong \Tor_{i+1}^S(\omega, C)\cong\Tor_{i+2}^S(\omega, \cTr_\omega\coOmega^{n+1}(M))$$
for any $0\leqslant i\leqslant m-1$. Thus we conclude that $\Tor_{0\leqslant i\leqslant m-1}^S(\omega, \Ext^{n+1}_R(\omega,M))=0$
if and only if $\cTr_\omega \coOmega^{n+1}(M)\in{{\omega_S}^{\top_{m+1}}}$,
and if and only if $\coOmega^{n+1}(M)\in \cT^{m+1}_{\omega}(R)$. The proof is finished.
\end{proof}

\begin{lemma} \label{lem-3.3}
Assume that $\Omega^n_{\mathcal{F}}(S)\subseteq \acT^m_{\omega}(S)$ with $m,n\geqslant 0$.
Then the following statements are equivalent.
\begin{enumerate}
\item $\Ecograde_{\omega}\Tor_{n+1}^{S}(\omega,N)\geqslant m$ for any $N\in \Mod S$.
\item $\Omega^{n+1}_{\mathcal{F}}(S)\subseteq \acT^{m+1}_{\omega}(S)$.
\end{enumerate}
\end{lemma}

\begin{proof}
Because any flat module in $\Mod S$ is in $\acT^{1}_{\omega}(S)$ by \cite[Corollary 3.5(1)]{TH3}, we have that
$\Omega^{n+1}_{\mathcal{F}}(S)\subseteq \acT^{1}_{\omega}(S)$ for any $n\geqslant 0$, and the case for $m=0$ follows.
Now suppose that $m\geqslant 1$ and $N\in \Mod S$.
By Lemma ~\ref{lem-3.1}(2), there exists an exact sequence
$$0\to \im(1_{\omega}\otimes f_{n+1})\stackrel{\sigma}{\rightarrow} \acTr_\omega\Omega^n_{\mathcal{F}}(N) \stackrel{\tau}{\rightarrow}
\Tor^S_1(\omega,\Omega^n_{\mathcal{F}}(N))\to 0$$
in $\Mod R$ such that $\sigma_*$ is an isomorphism. Because $\Omega^n_{\mathcal{F}}(S)\subseteq \acT^m_{\omega}(S)$
by assumption, we have that both $\acTr_\omega \Omega^n_{\mathcal{F}}(N)$ and
$\acTr_\omega \Omega^{n+1}_{\mathcal{F}}(N)$ are in ${_R\omega^{\perp_m}}$. It yields that
$$\Ext^{i}_{R}(\omega, \Tor_{n+1}^{S}(\omega,N))\cong \Ext^{i}_{R}(\omega, \Tor_{1}^{S}(\omega,\Omega^n_{\mathcal{F}}(N)))
\cong \Ext^{i+1}_{R}(\omega, \im(1_{\omega}\otimes f_{n+1}))$$
for any $0\leqslant i\leqslant m-1$. In addition, we also have an exact sequence
$$0\to \acTr_\omega\Omega^{n+1}_{\mathcal{F}}(N) \to \omega\otimes_SF_{n+2}(N) \to \im(1_{\omega}\otimes f_{n+1})\to 0$$
in $\Mod R$. By \cite[Corollary 6.1]{HW}, we have $\omega\otimes_SF_{n+2}(N)\in{_R\omega^{\perp}}$.
So $$\Ext^{i}_{R}(\omega, \Tor_{n+1}^{S}(\omega,N))
\cong \Ext^{i+1}_{R}(\omega, \im(1_{\omega}\otimes f_{n+1}))\cong\Ext^{i+2}_{R}(\omega, \acTr_\omega\Omega^{n+1}_{\mathcal{F}}(N))$$
for any $0\leqslant i\leqslant m-1$. Thus we conclude that $\Ext^{0\leqslant i\leqslant m-1}_{R}(\omega, \Tor_{n+1}^{S}(\omega,N))=0$
if and only if $\acTr_\omega \Omega^{n+1}_{\mathcal{F}}(N)\in{_R\omega^{\perp_{m+1}}}$,
and if and only if $\Omega^{n+1}_{\mathcal{F}}(N)\in \acT^{m+1}_{\omega}(S)$. The proof is finished.
\end{proof}

Let $\mathcal{T}\subseteq \mathcal{W}$ be subcategories of an abelian category $\mathcal{E}$. Recall that $\mathcal{T}$ is called a
{\it generator} (resp. {\it cogenerator})
for $\mathcal{W}$ if for any $W\in \mathcal{W}$, there exists an exact sequence
$$0\to W' \to T \to W \to 0\ (\text{resp.}\ 0\to W \to T \to W' \to 0)$$
in $\mathcal{E}$ with $T\in \mathcal{T}$ and $W'\in \mathcal{W}$.

\begin{lemma} \label{lem-3.4}
\begin{enumerate}
\item[]
\item $\mathcal{P}_{\omega}(R)$ is a generator for $\mathcal{B}_{\omega}(R)$.
\item $\coOmega^{n}(R)\subseteq\coOmega^{n}_{\mathcal{B}}(R)=\coOmega^{n}_{\mathcal{F}_{\omega}}(R)=\coOmega^{n}_{\mathcal{P}_{\omega}}(R)$
for any $n\geqslant 1$.
\end{enumerate}
\end{lemma}

\begin{proof}
(1) Let $M\in\mathcal{B}_{\omega}(R)$. Then by \cite[Theorem 3.9 and Proposition 3.7]{TH1},
there exists an exact sequence
$$\cdots \to W_2 \to W_1 \to W_0 \to M \to 0$$
in $\Mod R$ with all $W_i\in \mathcal{P}_\omega(R)$ such that it remains exact after applying the functor $\Hom_R(\omega,-)$.
Put $M_1:=\im(W_1 \to W_0)$. Then $M_1\in\cT_{\omega}(R)$ by \cite[Proposition 3.7]{TH1}. Because
both $M$ and $W_0$ are in ${_R\omega}^{\bot}$, we have $M_1\in{_R\omega}^{\bot}$.
So $M_1\in\mathcal{B}_{\omega}(R)$ by \cite[Theorem 3.9]{TH1}.

(2) Let $n\geqslant 1$. By \cite[Lemma 4.1]{HW}, we have that $\mathcal{B}_{\omega}(R)$ contains all injective left $R$-modules,
which yields $\coOmega^{n}(R)\subseteq\coOmega^{n}_{\mathcal{B}}(R)$.
Because ${\mathcal{B}_{\omega}}(R)\supseteq {\mathcal{F}_{\omega}}(R)\supseteq{\mathcal{P}_{\omega}}(R)$
by \cite[Corollary 6.1]{HW},
we have $\coOmega^{n}_{\mathcal{B}}(R)\supseteq\coOmega^{n}_{\mathcal{F}_{\omega}}(R)\supseteq\coOmega^{n}_{\mathcal{P}_{\omega}}(R)$.
Because $\mathcal{B}_{\omega}(R)$ is closed under extensions by \cite[Theorem 6.2]{HW}, we have
$\coOmega^{n}_{\mathcal{B}}(R)=\coOmega^{n}_{\mathcal{P}_{\omega}}(R)$ by (1) and \cite[Corollary 5.4(2)]{H4}.
\end{proof}

In the following result, we characterize when the double functor $\Tor^S_i(\omega, \Ext^i_{R}(\omega,-))$ preserves epimorphisms
in terms of the $\Tor$-cograde conditions of $\Ext$-modules.

\begin{theorem} \label{thm-3.5} The conditions (1)--(3) below are equivalent for any $n, k\geqslant 0$. If $k\geqslant 1$,
then (1)--(4) are equivalent.
\begin{enumerate}
\item $\Tcograde_{\omega}\Ext^{i+k}_{R}(\omega,M)\geqslant i$ for any $M\in \Mod R$ and $1\leqslant i\leqslant n$.
\item $\Tor^S_i(\omega, \Ext^i_{R}(\omega,f))$ is an epimorphism for any epimorphism $f: B\twoheadrightarrow C$
in $\Mod R$ with $B,C\in \coOmega^{k+1}_{\mathcal{P}_{\omega}}(R)$ and $0\leqslant i\leqslant n-1$.
\item $\Tor^S_i(\omega, \Ext^i_{R}(\omega,f))$ is an epimorphism for any epimorphism $f: B\twoheadrightarrow C$
in $\Mod R$ with $B,C\in \coOmega^{k+1}(R)$ and $0\leqslant i\leqslant n-1$.
\item $\coOmega^{i+k}(R)\subseteq \cT^{i+1}_{\omega}(R)$ for any $1\leqslant i\leqslant n$.
\end{enumerate}
\end{theorem}

\begin{proof}
By using induction on $i$, $(1)\Leftrightarrow (4)$ follows from Lemma ~\ref{lem-3.2}.

$(1)\Rightarrow (2)$ Let $f: B\twoheadrightarrow C$ be an epimorphism in $\Mod R$ with $B,C\in \coOmega^{k+1}_{\mathcal{P}_{\omega}}(R)$.
Then $C=\coOmega^{k+1}_{\mathcal{P}_{\omega}}(C')$ for some $C'\in \Mod R$. By (1), we have
$$\Tor^S_i(\omega, \Ext^i_{R}(\omega,C))\cong\Tor^S_{i}(\omega, \Ext^{i+k+1}_{R}(\omega,C'))=0$$
for any $1\leqslant i\leqslant n-1$. Thus $\Tor^S_i(\omega, \Ext^i_{R}(\omega,f))$ is epic.
In the following, we will show that $1_\omega\otimes f_*$ is epic.

If $k\geqslant 1$, then $\coOmega^{k}_{\mathcal{P}_{\omega}}(R)\subseteq \cT^{1}_{\omega}(R)$ by Lemma \ref{lem-2.9}(2).
So $\coOmega^{k+1}_{\mathcal{P}_{\omega}}(R)\subseteq \cT^{2}_{\omega}(R)$ by Lemma ~\ref{lem-3.2}, and hence $B,C\in \cT^{2}_{\omega}(R)$.
It follows that $1_\omega\otimes f_*\cong f$ and $1_\omega\otimes f_*$ is epic.

Now suppose $k=0$. We have an epimorphism $p:W\twoheadrightarrow B$ in $\Mod R$ with $W\in \Add_R\omega$.
From the exact sequence
$$0\to M_1\to W\stackrel{f\cdot p}{\longrightarrow}C\to 0$$ in $\Mod R$ with $M_1=\Ker (f\cdot p)$,
we get the following exact sequence
$$W_*\stackrel{(f\cdot p)_*}{\longrightarrow}C_*\to \Ext^1_R(\omega,M_1)\to 0$$
in $\Mod S$. By (1). $\omega\otimes_S\Ext^1_R(\omega,M_1)=0$.
So $(1_\omega\otimes f_*)\cdot (1_\omega\otimes p_*)=1_\omega\otimes (f\cdot p)_*$ is epic,
which implies that $1_\omega\otimes f_*$ is also epic.

By Lemma ~\ref{lem-3.4}(2), we have $(2)\Rightarrow (3)$.

$(3)\Rightarrow (1)$ Let $M\in \Mod R$. From the exact sequence
$$0\to \coOmega^{k}(M)\to I^k(M)\stackrel{f}{\longrightarrow}\coOmega^{k+1}(M)\to 0$$ in $\Mod R$,
we get the following exact sequence
$$I^k(M)_*\stackrel{f_*}{\longrightarrow}\coOmega^{k+1}(M)_*\to \Ext^{k+1}_R(\omega, M)\to 0$$ in $\Mod S$.
Since $1_\omega\otimes f_*$ is an epimorphism by (2), we have that $\omega\otimes_S\Ext^{k+1}_R(\omega, M)=0$
and $\Tcograde_{\omega}\Ext^{k+1}_{R}(\omega,M)\geqslant 1$. In addition, for any $1\leqslant i\leqslant n-1$,
$$0=\Tor^S_i(\omega, \Ext^i_{R}(\omega,I^k(M)))\stackrel{\Tor^S_i(\omega, \Ext^i_{R}(\omega,f))}{\longrightarrow}
\Tor^S_i(\omega, \Ext^i_{R}(\omega,\coOmega^{k+1}(M)))$$ is epic by (3), so we have
$$\Tor^S_i(\omega, \Ext^{i+k+1}_{R}(\omega,M))\cong \Tor^S_i(\omega, \Ext^i_{R}(\omega,\coOmega^{k+1}(M)))=0.$$
Thus we conclude that $\Tcograde_{\omega}\Ext^{i+k+1}_{R}(\omega,M)\geqslant i+1$ for any $0\leqslant i\leqslant n-1$.
\end{proof}

\begin{lemma} \label{lem-3.6}
\begin{enumerate}
\item[]
\item $\mathcal{I}_{\omega}(S)$ is a cogenerator for $\mathcal{A}_{\omega}(S)$.
\item $\Omega^{n}_{\mathcal{F}}(S)\subseteq\Omega^{n}_{\mathcal{A}}(S)=\Omega^{n}_{\mathcal{I}_{\omega}}(S)$ for any $n\geqslant 1$.
\end{enumerate}
\end{lemma}

\begin{proof}
(1) Let $N\in\mathcal{A}_{\omega}(S)$. Then by \cite[Theorem 3.11(1)]{TH3},
there exists an $(\omega\otimes_S-)$-exact exact sequence
$$0\to N \to U^0 \to U^1 \to U^2 \to \cdots$$
in $\Mod S$ with all $U^i\in \mathcal{I}_\omega(S)$. Put $N^1:=\im(U^0 \to U^1)$. Then $N^1\in\acT_{\omega}(S)$ by \cite[Corollary 3.9]{TH3}.
Because both $N$ and $U^0$ are in ${\omega_S}^{\top}$, we have $N^1\in{\omega_S}^{\top}$.
So $N^1\in\mathcal{A}_{\omega}(S)$ by \cite[Theorem 3.11(1)]{TH3} again.

(2) Let $n\geqslant 1$. By \cite[Lemma 4.1]{HW}, we have that $\mathcal{A}_{\omega}(S)$ contains all flat left $S$-modules,
which yields $\Omega^{n}_{\mathcal{F}}(S)\subseteq\Omega^{n}_{\mathcal{A}}(S)$.
Because $\mathcal{A}_{\omega}(S)$ is closed under extensions by \cite[Theorem 6.2]{HW}, we have
$\Omega^{n}_{\mathcal{A}}(S)=\Omega^{n}_{\mathcal{I}_{\omega}}(S)$ by (1) and \cite[Corollary 5.4(1)]{H4}.
\end{proof}

In the following result, we characterize when the double functor $\Ext_{R}^i(\omega, \Tor_i^{S}(\omega,-))$ preserves monomorphisms
in terms of the $\Ext$-cograde conditions of $\Tor$-modules.

\begin{theorem} \label{thm-3.7}
The conditions (1)--(3) below are equivalent for any $n, k\geqslant 0$. If $k\geqslant 1$, then (1)--(4) are equivalent.
\begin{enumerate}
\item $\Ecograde_{\omega}\Tor_{i+k}^{S}(\omega,N)\geqslant i$ for any $N\in \Mod S$ and $1\leqslant i\leqslant n$.
\item $\Ext_{R}^i(\omega, \Tor_i^{S}(\omega,g))$ is a monomorphism for any monomorphism $g: B'\rightarrowtail C'$ in $\Mod S$
with $B', C'\in \Omega^{k+1}_{\mathcal{I}_{\omega}}(S)$ and $0\leqslant i\leqslant n-1$.
\item $\Ext_{R}^i(\omega, \Tor_i^{S}(\omega,g))$ is a monomorphism for any monomorphism $g: B'\rightarrowtail C'$ in $\Mod S$
with $B', C'\in \Omega^{k+1}_{\mathcal{F}}(S)$ and $0\leqslant i\leqslant n-1$.
\item $\Omega^{i+k}_{\mathcal{F}}(S)\subseteq \acT^{i+1}_{\omega}(S)$ for any $1\leqslant i\leqslant n$.
\end{enumerate}
\end{theorem}

\begin{proof}
By using induction on $i$, $(1)\Leftrightarrow (4)$ follows from Lemma ~\ref{lem-3.3}.

$(1)\Rightarrow (2)$
Let $g: B'\rightarrowtail C'$ be a monomorphism in $\Mod S$ with $B', C'\in \Omega^{k+1}_{\mathcal{I}_{\omega}}(S)$.
Then $B'=\Omega^{k+1}_{\mathcal{I}_{\omega}}(B'')$ for some $B''\in \Mod S$. By (1), we have
$$\Ext_{R}^{i}(\omega, \Tor_{i}^{S}(\omega,B'))\cong\Ext_{R}^{i}(\omega, \Tor^{S}_{i+k+1}(\omega,B''))=0$$
for any $1\leqslant i\leqslant n-1$. Thus $\Ext_{R}^i(\omega, \Tor_i^{S}(\omega,g))$ is a monic.
In the following, we will show that $(1_\omega\otimes g)_*$ is monic.

If $k\geqslant 1$, then $\Omega^{k}_{\mathcal{I}_{\omega}}(S)\subseteq \acT^{1}_{\omega}(S)$ by Lemma \ref{lem-2.9}(1). So
$\Omega^{k+1}_{\mathcal{I}_{\omega}}(S)\subseteq \acT^{2}_{\omega}(S)$ by Lemma ~\ref{lem-3.3},
and hence $B',C'\in\acT^{2}_{\omega}(S)$. It follows that $(1_{\omega}\otimes g)_*\cong g$ and $(1_{\omega}\otimes g)_*$ is monic.

Now suppose $k=0$. We have a monomorphism $i: C'\rightarrowtail U$ in $\Mod S$ with $U\in \mathcal{I}_\omega(S)$.
From the exact sequence $$0\to B'\stackrel{i\cdot g}{\longrightarrow} U \to L_1\to 0$$
in $\Mod S$ with $L_1=\Coker(i\cdot g)$, we get the following exact sequence
$$0\to \Tor^S_1(\omega,L_1)\to \omega\otimes_S B'\stackrel{1_{\omega}\otimes(i\cdot g)}{\longrightarrow} \omega\otimes_SU$$
in $\Mod R$. By (1), $(\Tor^S_1(\omega,L_1))_*=0$. So $(1_{\omega}\otimes i)_*\cdot (1_{\omega}\otimes g)_*=(1_\omega\otimes(i\cdot g))_*$
is monic, which implies that $(1_{\omega}\otimes g)_*$ is also monic.

By Lemma ~\ref{lem-3.6}(2), we have $(2)\Rightarrow (3)$.

$(3)\Rightarrow (1)$ Let $N\in \Mod S$. From the exact sequence
$$0\to \Omega^{k+1}_{\mathcal{F}}(N)\stackrel{g}{\longrightarrow} F_k(N)\to\Omega^{k}_{\mathcal{F}}(N)\to 0$$ in $\Mod S$,
we get the following exact sequence
$$0\to \Tor^S_{k+1}(\omega,N)\to \omega\otimes_S \Omega^{k+1}_{\mathcal{F}}(N)\stackrel{1_\omega\otimes g}{\longrightarrow} \omega\otimes_SF_k(N)$$
in $\Mod R$. Since $(1_\omega\otimes g)_*$ is a monomorphism by (2), we have that $(\Tor^S_{k+1}(\omega,N))_*=0$ and
$\Ecograde_{\omega}\Tor_{k+1}^{S}(\omega,N)\geqslant 1$. In addition, for any $1\leqslant i\leqslant n-1$,
$$\Ext_{R}^i(\omega, \Tor_i^{S}(\omega,\Omega^{k+1}_{\mathcal{F}}(N)))\stackrel{\Ext_{R}^i(\omega,\Tor_i^{S}(\omega,g))}{\longrightarrow}
\Ext_{R}^i(\omega, \Tor_i^{S}(\omega,F_k(N)))=0$$ is monic by (3), so we have
$$\Ext_{R}^i(\omega, \Tor_{i+k+1}^{S}(\omega,N))\cong\Ext_{R}^i(\omega, \Tor_i^{S}(\Omega^{k+1}_{\mathcal{F}}(\omega,N))=0.$$
Thus we conclude that $\Ecograde_{\omega}\Tor_{i+k+1}^{S}(\omega,N)\geqslant i+1$ for any $0\leqslant i\leqslant n-1$.
\end{proof}

\vspace{0.2cm}

\noindent{\bf 3.2. Strong cograde conditions}

\vspace{0.2cm}

Compare the following result with Theorem ~\ref{thm-3.5}.

\begin{theorem} \label{thm-3.8}
For any $n\geqslant 1$ and $k\geqslant 0$, the following three statements are equivalent.
\begin{enumerate}
\item $\sTcograde_{\omega}\Ext^{i+k}_{R}(\omega,M)\geqslant i$ for any $M\in \Mod R$ and $1\leqslant i\leqslant n$.
\item For any exact sequence
$$0\to A\to B\stackrel{f}{\longrightarrow}C\to 0$$ in $\Mod R$ with
$A\in \Omega^{i-1}_{\mathcal{P}_\omega}(\coOmega^{i+k-1}_{\mathcal{P}_\omega}(R))$,
$\Tor^S_i(\omega, \Ext^i_{R}(\omega,f))$ is an epimorphism for any $0\leqslant i\leqslant n-1$.
\item For any exact sequence
$$0\to A\to B\stackrel{f}{\longrightarrow}C\to 0$$ in $\Mod R$ with $A\in \Omega^{i-1}_{\mathcal{P}_\omega}(\coOmega^{i+k-1}(R))$,
$\Tor^S_i(\omega, \Ext^i_{R}(\omega,f))$ is an epimorphism for any $0\leqslant i\leqslant n-1$.
\end{enumerate}
Moreover, if $k=0$, then any of the above statements is equivalent to the following
\begin{enumerate}
\item[(4)] For any exact sequence
$$0\to A\to B\stackrel{f}{\longrightarrow}C\to 0$$ in $\Mod R$,
$\Tor^S_i(\omega, \Ext^i_{R}(\omega,f))$ is an epimorphism for any $0\leqslant i\leqslant n-1$.
\end{enumerate}
\end{theorem}

\begin{proof}
$(1)\Rightarrow (2)$ Let $A=\Omega^{i-1}_{\mathcal{P}_\omega}(\coOmega^{i+k-1}_{\mathcal{P}_\omega}(A'))$ with $A'\in\Mod R$.
For any $i\geqslant 0$, by dimension-shifting we have an exact sequence
$$\Ext^{i+k}_R(\omega,A')\stackrel{g}{\longrightarrow} \Ext^{i}_R(\omega,B)\stackrel{\Ext^{i}_R(\omega,f)}{\longrightarrow}
\Ext^{i}_R(\omega,C)\to \Ext^{i+k+1}_R(\omega,A')$$ in $\Mod S$, which induces exact sequences
$$\Tor^S_i(\omega, \Ext^i_{R}(\omega,B))\stackrel{a}{\longrightarrow} \Tor^S_i(\omega, \im (\Ext^i_{R}(\omega,f)))\to
\Tor^S_{i-1}(\omega, \Ker (\Ext^i_{R}(\omega,f)))$$ and
$$\Tor^S_i(\omega, \im (\Ext^i_{R}(\omega,f)))\stackrel{b}{\longrightarrow} \Tor^S_i(\omega, \Ext^i_{R}(\omega,C))\to
\Tor^S_i(\omega, \Coker (\Ext^i_{R}(\omega,f)))$$ in $\Mod R$.
Since $\Coker (\Ext^{i}_R(\omega,f))\subseteq \Ext^{i+k+1}_R(\omega,A')$, by (1) we have
$$\Tor^S_i(\omega, \Coker (\Ext^i_{R}(\omega,f)))=0$$ for any $0\leqslant i \leqslant n-1$.
Moreover, it follows from (1) and the exact sequence
$$0\to \Ker g\to \Ext^{i+k}_R(\omega,A')\to \Ker (\Ext^i_{R}(\omega,f))\to 0$$
in $\Mod S$ that $\Tor^S_{i-1}(\omega, \Ker (\Ext^i_{R}(\omega,f)))=0$ for any $0\leqslant i \leqslant n-1$.
Thus $\Tor^S_i(\omega, \Ext^i_{R}(\omega,f))=b\cdot a$ is an epimorphism for any $0\leqslant i \leqslant n-1$.

By Lemma ~\ref{lem-3.4}(2), we have $(2)\Rightarrow (3)$.

$(3)\Rightarrow (1)$ Let $M\in \Mod R$. Fix $i$ ($1\leqslant i\leqslant n$) and an $S$-submodule $L$ of $\Ext^{i+k}_{R}(\omega,M)$.
Take an epimorphism $a:P\twoheadrightarrow L$ in $\Mod S$ with $P$ projective and $a'$ the composition
$$P\stackrel{a}{\twoheadrightarrow} L\hookrightarrow \Ext^{i+k}_{R}(\omega,M).$$
Then $a'$ can be lifted to $b: P\to \coOmega^{i+k}(M)_*$. Take the following pull-back diagram
$$\xymatrix{ 0 \ar[r] & \coOmega^{i+k-1}(M) \ar[r]^{d} \ar@{=}[d] & X  \ar[r]^{c} \ar[d] & \omega\otimes_SP \ar[r] \ar[d]^{b'} & 0 \\
0 \ar[r] & \coOmega^{i+k-1}(M) \ar[r] & I^{i+k-1}(M)  \ar[r]  & \coOmega^{i+k}(M) \ar[r]& 0, \\
 & & {\rm Diagram}\ (3.3) &  }$$
where $b'$ is the composition
$$\omega\otimes_SP\stackrel{1_\omega\otimes b}{\longrightarrow}\omega\otimes_S \coOmega^{i+k}(M)_*
\stackrel{\theta_{\coOmega^{i+k}(M)}}{\longrightarrow}\coOmega^{i+k}(M).$$
It induces the following commutative diagram with exact rows
{\tiny $$\xymatrix{ 0 \ar[r] & \coOmega^{i+k-1}(M)_* \ar[r]^{d_*} \ar@{=}[d] & X_*  \ar[r]^{c_*} \ar[d]
& (\omega\otimes_SP)_* \ar[r] \ar[d]^{{b'}_*} & L \ar[r] \ar[d] & 0 \\
0 \ar[r] & \coOmega^{i+k-1}(M)_* \ar[r] & I^{i+k-1}(M)_*  \ar[r]  & \coOmega^{i+k}(M)_* \ar[r]& \Ext^{i+k}_{R}(\omega,M)\ar[r]&0. \\}$$}

In the following, we will proceed by induction on $i$. Let $i=1$. Since $1_{\omega}\otimes c_*$ is epic by (3),
we have that $\omega\otimes_S L=0$ and $\sTcograde_{\omega}\Ext^{1+k}_{R}(\omega,M)\geqslant 1$.

Assume that the statement (1) holds for any $1\leqslant i \leqslant n-1$.
Now consider the case for $i=n$. By the induction hypothesis, we have that $\sTcograde_{\omega}\Ext^{i+k}_{R}(\omega,M)\geqslant i$
for any $1\leqslant i\leqslant n-1$ and $\sTcograde_{\omega}\Ext^{n+k}_{R}(\omega,M)\geqslant n-1$. Then
$\coOmega^{n+k-1}(M)\in \cT^{n-1}_{\omega}(R)$ by Lemma ~\ref{lem-3.2}. Because $\omega\otimes_SP\in \cT^{n-1}_{\omega}(R)$
by \cite[Proposition 3.7]{TH1}, it follows from \cite[Lemma 4.3]{TH5} that $X$
in the diagram (3.3) is in $\cT^{n-1}_{\omega}(R)$. By \cite[Proposition 3.7]{TH1} again, there exists
$\Hom_R(\Add_R\omega,-)$-exact exact sequences
$$0\to Y' \to W'_{n-2} \to \cdots \to W'_0 \to \coOmega^{n+k-1}(M) \to 0$$
and
$$0\to Y \to W_{n-2} \to \cdots \to W_0 \to X \to 0$$
in $\Mod R$ with all $W'_j,W_j$ in $\Add_R\omega$. Then both $Y$ and $Y'$ are in ${_R\omega}^{\bot_{n-1}}$
and we get the following commutative diagram
{\tiny $$\xymatrix{ 0 \ar[r] & Y' \ar[r] \ar@{-->}[d]^{g} & W'_{n-2} \ar[r] \ar@{-->}[d] & \cdots \ar[r]
&  W'_{0} \ar[r] \ar@{-->}[d] &  \coOmega^{n+k-1}(M) \ar[r] \ar[d]^{d} & 0 \\
0 \ar[r] & Y \ar[r] & W_{n-2}  \ar[r]  & \cdots \ar[r] & W_{0} \ar[r]& X \ar[r]&0.}$$}
We can guarantee that $g$ is a monomorphism by adding a direct summand in $\Add_R\omega$ (for example $W'_{n-2}$) to $Y$ and $W_{n-2}$.
Thus we get an exact sequence
$$0\to Y'\stackrel{g}{\longrightarrow} Y \stackrel{h}{\longrightarrow} Z\to 0$$ in $\Mod R$ with $Z=\Coker g$. Since
$$\Coker(\Ext^{n-1}_R(\omega,h))\cong \Ker(\Ext^{n}_R(\omega,g))\cong \Ker(\Ext^{1}_R(\omega,d))\cong \Coker c_*\cong L,$$
we obtain $L\cong \Ext^{n-1}_R(\omega,Z)$. Since $Y'\in \Omega^{n-1}_{\mathcal{P}_\omega}(\coOmega^{n+k-1}(R))$, by (3) we get that
$\Tor^S_{n-1}(\omega, \Ext^{n-1}_{R}(\omega,h))$ is epic. So
$\Tor_{n-1}^S(\omega,L)=0$ and $\sTcograde_{\omega}\Ext^{n+k}_{R}(\omega,M)$ $\geqslant n$.

When $k=0$, the proof of $(3)\Rightarrow(1)\Rightarrow (2)$ is in fact that of $(4)\Leftrightarrow (1)$
by just removing the first sentence and putting $A'=A$ in the beginning of the proof of $(1)\Rightarrow (2)$,
\end{proof}

Compare the following result with Theorem ~\ref{thm-3.7}.

\begin{theorem} \label{thm-3.9}
For any $n\geqslant 1$ and $k\geqslant 0$, the following three statements are equivalent.
\begin{enumerate}
\item $\sEcograde_{\omega}\Tor_{i+k}^{S}(\omega,N)\geqslant i$ for any $N\in \Mod S$ and $1\leqslant i\leqslant n$.
\item For any exact sequence
$$0\to A\stackrel{g}{\longrightarrow} B\to C\to 0$$ in $\Mod S$ with $C\in \coOmega^{i-1}_{\mathcal{I}_\omega}(\Omega^{i+k-1}_{\mathcal{I}_\omega}(S))$,
$\Ext_{R}^i(\omega, \Tor_i^{S}(\omega,g))$ is a monomorphism for any $0\leqslant i\leqslant n-1$.
\item For any exact sequence
$$0\to A\stackrel{g}{\longrightarrow} B\to C\to 0$$ in $\Mod S$ with $C\in \coOmega^{i-1}_{\mathcal{I}_\omega}(\Omega^{i+k-1}_{\mathcal{F}}(S))$,
$\Ext_{R}^i(\omega, \Tor_i^{S}(\omega,g))$ is a monomorphism for any $0\leqslant i\leqslant n-1$.
\end{enumerate}
Moreover, if $k=0$, then any of the above statements is equivalent to the following
\begin{enumerate}
\item[(4)] For any exact sequence
$$0\to A\stackrel{g}{\longrightarrow} B\to C\to 0$$ in $\Mod S$,
$\Ext_{R}^i(\omega, \Tor_i^{S}(\omega,g))$ is a monomorphism for any $0\leqslant i\leqslant n-1$.
\end{enumerate}
\end{theorem}

\begin{proof}
$(1)\Rightarrow (2)$ Let $C=\coOmega^{i-1}_{\mathcal{I}_\omega}(\Omega^{i+k-1}_{\mathcal{I}_\omega}(C'))$ with $C'\in \Mod S$.
For any $i\geqslant 0$, by dimension shifting we have an exact sequence
$$\Tor_{i+k+1}^S(\omega,C')\to \Tor_{i}^S(\omega,A)\stackrel{\Tor_{i}^S(\omega,g)}{\longrightarrow}\Tor_{i}^S(\omega,B)
\stackrel{f}{\longrightarrow} \Tor_{i+k}^S(\omega,C')$$ in $\Mod R$, which induces exact sequences
$$\Ext_{R}^i(\omega, \Ker (\Tor_{i}^S(\omega,g)))\to \Ext_{R}^i(\omega, \Tor_{i}^S(\omega,A))
\stackrel{a}{\longrightarrow} \Ext_{R}^i(\omega, \im (\Tor_{i}^S(\omega,g)))$$
and
$$\Ext_{R}^{i-1}(\omega, \Coker (\Tor_{i}^S(\omega,g)))\to \Ext_{R}^i(\omega, \im (\Tor_{i}^S(\omega, g)))
\stackrel{b}{\longrightarrow} \Ext_{R}^i(\omega, \Tor_{i}^S(\omega,B))$$ in $\Mod S$.
Since $\Ker (\Tor_{i}^S(\omega,g))$ is an $R$-quotient module of $\Tor_{i+k+1}^S(\omega,C')$, by (1) we have
$$\Ext_{R}^i(\omega, \Ker (\Tor_{i}^S(\omega,g)))=0.$$
Moreover, it follows from (1) and the exact sequence
$$0\to \Coker(\Tor_{i}^S(\omega,g))\to \Tor_{i+k}^S(\omega,C')\to \Coker f\to 0$$ in $\Mod R$ that
$\Ext_{R}^{i-1}(\omega, \Coker (\Tor_{i}^S(\omega,g)))=0$ for any $0\leqslant i \leqslant n-1$. Thus
$\Ext_{R}^i(\omega, \Tor_i^{S}(\omega,g))=b\cdot a$ is a monomorphism for any $0\leqslant i \leqslant n-1$.

By Lemma ~\ref{lem-3.6}(2), we have $(2)\Rightarrow (3)$.

$(3)\Rightarrow (1)$ Let $N\in \Mod S$. Fix $i$ ($1\leqslant i\leqslant n$) and an $R$-quotient module $H$ of
$\Tor_{i+k}^{S}(\omega,N)$. Take a monomorphism $a:H\rightarrowtail I$ in $\Mod R$ with $I$ injective
and $a'$ the composition
$$\Tor_{i+k}^{S}(\omega,N)\twoheadrightarrow H\stackrel{a}{\rightarrowtail} I.$$ Then $a'$ can be extended to
$b:\omega\otimes_S\Omega^{i+k}_{\mathcal{F}}(N)\to I$. Take the following push-out diagram:
{$$\xymatrix{ 0 \ar[r] & \Omega^{i+k}_{\mathcal{F}}(N) \ar[r] \ar[d]^{b'} & F_{i+k-1}(N)  \ar[r] \ar[d]
& \Omega^{i+k-1}_{\mathcal{F}}(N) \ar[r] \ar@{=}[d] & 0 \\
0 \ar[r] & I_* \ar[r]^{c} & Y  \ar[r]^{d}  & \Omega^{i+k-1}_{\mathcal{F}}(N) \ar[r]& 0,\\
&  &{\rm Diagram}\ (3.5) &  }$$}
where $b'$ is the composition
$$\Omega^{i+k}_{\mathcal{F}}(N)\stackrel{\mu_{\Omega^{i+k}_{\mathcal{F}}(N)}}{\longrightarrow}
(\omega\otimes_S\Omega^{i+k}_{\mathcal{F}}(N))_*\stackrel{b_*}{\longrightarrow} I_*.$$
It induces the following commutative diagram with exact rows
{\tiny $$\xymatrix{ 0 \ar[r] & \Tor_{i+k}^{S}(\omega,N) \ar[r] \ar[d] & \omega\otimes_S\Omega^{i+k}_{\mathcal{F}}(N)
\ar[r] \ar[d]^{1_{\omega}\otimes b'} & \omega\otimes_SF_{i+k-1}(N) \ar[r] \ar[d]
& \omega\otimes_S\Omega^{i+k-1}_{\mathcal{F}}(N) \ar[r] \ar@{=}[d] & 0 \\
0 \ar[r] & H \ar[r] & \omega\otimes_SI_* \ar[r]^{1_{\omega}\otimes c}
& \omega\otimes_SY \ar[r]^{1_{\omega}\otimes d}& \omega\otimes_S\Omega^{i+k-1}_{\mathcal{F}}(N)\ar[r]&0.}$$}

In the following, we will proceed by induction on $i$. Let $i=1$. Since $(1_{\omega}\otimes c)_*$ is monic by (2), we have that $H_*=0$
and $\sEcograde_{\omega}\Tor_{1+k}^{S}(\omega,N)\geqslant 1$.

Assume that the statement (1) holds for any $1\leqslant i \leqslant n-1$.
Now consider the case for $i=n$. By the induction hypothesis, we have that $\sEcograde_{\omega}\Tor_{i+k}^{S}(\omega,N)\geqslant i$
for any $1\leqslant i\leqslant n-1$ and $\sEcograde_{\omega}\Tor_{n+k}^{S}(\omega,N)\geqslant n-1$.
Then $\Omega^{n+k-1}_{\mathcal{F}}(N)\in \acT^{n-1}_{\omega}(S)$ by Lemma ~\ref{lem-3.3}.
Because $I_*\in\acT^{n-1}_{\omega}(S)$ by \cite[Propposition]{TH3},
it follows from the dual result of \cite[Lemma 4.3]{TH5} that $Y$ in the diagram (3.5) is in $\acT^{n-1}_{\omega}(S)$.
By \cite[Propposition]{TH3} again, there exist $(\omega\otimes_S-)$-exact exact sequences
$$0 \to Y \to U^{0} \to \cdots \to  U^{n-2} \to  X \to 0$$
and
$$0\to \Omega^{n+k-1}_{\mathcal{F}}(L) \to V^{0} \to \cdots \to V^{n-2} \to X' \to 0$$
in $\Mod S$ with all $U^i, V^i$ in $\mathcal{I}_\omega(S)$. Then both $X$ and $X'$ are in ${\omega_S}^{\top_{n-1}}$
and we get the following commutative diagram
{\tiny $$\xymatrix{ 0 \ar[r] & Y \ar[r] \ar[d]^{d} & U^{0} \ar[r] \ar@{-->}[d] & \cdots \ar[r]
&  U^{n-2} \ar[r] \ar@{-->}[d] &  X \ar[r] \ar@{-->}[d]^{f} & 0 \\
0 \ar[r] & \Omega^{n+k-1}_{\mathcal{F}}(L) \ar[r] & V^{0}  \ar[r]  & \cdots \ar[r] &  V^{n-2} \ar[r]& X' \ar[r]&0.}$$}
We can guarantee that $f$ is an epimorphism by adding a direct summand in $\mathcal{I}_{\omega}(S)$ (for example $V^{n-2}$) to $X$ and $U^{n-2}$.
Thus we get an exact sequence
$$0\to Z\stackrel{h}{\longrightarrow} X \stackrel{f}{\longrightarrow} X'\to 0$$ in $\Mod S$ with $Z=\Ker f$. Since
$$\Ker(\Tor_{n-1}^S(\omega,h))\cong \Coker(\Tor_{n}^S(\omega,f))\cong \Coker(\Tor_1^S(\omega,d))\cong \Ker(1_{\omega}\otimes c),$$
we obtain $H\cong \Tor_{n-1}^S(\omega,Z)$. Since $X'\in \coOmega^{n-1}_{\mathcal{I}_\omega}(\Omega^{n+k-1}_{\mathcal{F}}(S))$,
by (3) we get that $\Ext_{R}^{n-1}(\omega, \Tor_{n-1}^{S}(\omega,h))$ is a monomorphism.
So $\Ext_{R}^{n-1}(\omega,H)=0$ and $\sEcograde_{\omega}\Tor_{n+k}^{S}(\omega,N)\geqslant n$.

When $k=0$, the proof of $(3)\Rightarrow (1)\Rightarrow (2)$ is in fact that of $(4)\Leftrightarrow (1)$
by just removing the first sentence and putting $C'=C$ in the beginning of the proof of $(1)\Rightarrow (2)$,
\end{proof}

\section {\bf (Quasi) $n$-cograde condition}

In this section, we introduce and study the (quasi) $n$-cograde condition of semidualizing bimodules.

\vspace{0.2cm}

\noindent{\bf 4.1. The $n$-cograde condition}

\vspace{0.2cm}

\begin{definition} \label{def-4.1}
{\rm For any $n\geqslant 1$, $\omega$ is said to satisfy the {\it right $n$-cograde condition}
if $\sEcograde_\omega\Tor_{i}^S(\omega,N)\geqslant i$ for any $N\in \Mod S$ and $1\leqslant i\leqslant n$;
and $\omega$ is said to satisfy the {\it left $n$-cograde condition}
if $\sEcograde_\omega\Tor_{i}^R(M',\omega)\geqslant i$ for any $M'\in \Mod R^{op}$ and $1\leqslant i\leqslant n$.}
\end{definition}

As a consequence of Theorems \ref{thm-3.8} and \ref{thm-3.9}, we get the following equivalent characterizations
for $\omega$ satisfying the right $n$-cograde condition.

\begin{corollary} \label{cor-4.2}
For any $n\geqslant 1$, the following statements are equivalent.
\begin{enumerate}
\item $\sTcograde_\omega\Ext^i_R(\omega,M)\geqslant i$ for any $M\in \Mod R$ and $1\leqslant i\leqslant n$.
\item $\sEcograde_\omega\Tor_i^S(\omega,N)\geqslant i$ for any $N\in \Mod S$ and $1\leqslant i\leqslant n$.
\item $\Tor^S_i(\omega,\Ext^i_R({\omega},-))$ preserves epimorphisms in $\Mod R$ for $0\leqslant i\leqslant n-1$.
\item $\Ext^i_R(\omega,\Tor^S_i({\omega},-))$ preserves monomorphisms in $\Mod S$ for $0\leqslant i\leqslant n-1$.
\item For any exact sequence
$$0\to A\to B\stackrel{f}{\longrightarrow}C\to 0$$ in $\Mod R$ with
$A\in \Omega^{i-1}_{\mathcal{P}_\omega}(\coOmega^{i-1}_{\mathcal{P}_\omega}(R))$,
$\Tor^S_i(\omega, \Ext^i_{R}(\omega,f))$ is an epimorphism for any $0\leqslant i\leqslant n-1$.
\item For any exact sequence
$$0\to A\to B\stackrel{f}{\longrightarrow}C\to 0$$ in $\Mod R$ with $A\in \Omega^{i-1}_{\mathcal{P}_\omega}(\coOmega^{i-1}(R))$,
$\Tor^S_i(\omega, \Ext^i_{R}(\omega,f))$ is an epimorphism for any $0\leqslant i\leqslant n-1$.
\item For any exact sequence
$$0\to A\stackrel{g}{\longrightarrow} B\to C\to 0$$ in $\Mod S$ with $C\in \coOmega^{i-1}_{\mathcal{I}_\omega}(\Omega^{i-1}_{\mathcal{I}_\omega}(S))$,
$\Ext_{R}^i(\omega, \Tor_i^{S}(\omega,g))$ is a monomorphism for any $0\leqslant i\leqslant n-1$.
\item For any exact sequence
$$0\to A\stackrel{g}{\longrightarrow} B\to C\to 0$$ in $\Mod S$ with $C\in \coOmega^{i-1}_{\mathcal{I}_\omega}(\Omega^{i-1}_{\mathcal{F}}(S))$,
$\Ext_{R}^i(\omega, \Tor_i^{S}(\omega,g))$ is a monomorphism for any $0\leqslant i\leqslant n-1$.
\end{enumerate}
\end{corollary}

\begin{proof}
By \cite[Theorem 6.9]{TH2}, we have $(1)\Leftrightarrow (2)$. By Theorems \ref{thm-3.8} and \ref{thm-3.9},
we have $(1)\Leftrightarrow (3)\Leftrightarrow (5)\Leftrightarrow (6)$ and $(2)\Leftrightarrow (4)\Leftrightarrow (7)\Leftrightarrow (8)$
respectively.
\end{proof}

Symmetrically, we have the following equivalent characterizations for $\omega$ satisfying the left $n$-cograde condition.

\begin{corollary} \label{cor-4.3}
For any $n\geqslant 1$, the following statements are equivalent.
\begin{enumerate}
\item $\sTcograde_\omega\Ext^i_{S^{op}}(\omega,N')\geqslant i$ for any $N'\in \Mod S^{op}$ and $1\leqslant i\leqslant n$.
\item $\sEcograde_\omega\Tor_i^R(M',\omega)\geqslant i$ for any $M'\in \Mod R^{op}$ and $1\leqslant i\leqslant n$.
\item $\Tor^R_i(\Ext^i_{S^{op}}({\omega},-),\omega)$ preserves epimorphisms in $\Mod S^{op}$ for $0\leqslant i\leqslant n-1$.
\item $\Ext^i_{S^{op}}(\omega,\Tor^{R}_i(-,{\omega}))$ preserves monomorphisms in $\Mod R^{op}$ for $0\leqslant i\leqslant n-1$.
\item For any exact sequence
$$0\to A\to B\stackrel{f}{\longrightarrow}C\to 0$$ in $\Mod S^{op}$ with
$A\in \Omega^{i-1}_{\mathcal{P}_\omega}(\coOmega^{i-1}_{\mathcal{P}_\omega}(S^{op}))$,
$\Tor^{R}_i(\Ext^i_{S^{op}}(\omega,f),\omega)$ is an epimorphism for any $0\leqslant i\leqslant n-1$.
\item For any exact sequence
$$0\to A\to B\stackrel{f}{\longrightarrow}C\to 0$$ in $\Mod S^{op}$ with $A\in \Omega^{i-1}_{\mathcal{P}_\omega}(\coOmega^{i-1}(S^{op}))$,
$\Tor^{R}_i(\Ext^i_{S^{op}}(\omega,f),\omega)$ is an epimorphism for any $0\leqslant i\leqslant n-1$.
\item For any exact sequence
$$0\to A\stackrel{g}{\longrightarrow} B\to C\to 0$$ in $\Mod R^{op}$ with $C\in \coOmega^{i-1}_{\mathcal{I}_\omega}(\Omega^{i-1}_{\mathcal{I}_\omega}(R^{op}))$,
$\Ext^i_{S^{op}}(\omega,\Tor^{R}_i(g,{\omega}))$ is a monomorphism for any $0\leqslant i\leqslant n-1$.
\item For any exact sequence
$$0\to A\stackrel{g}{\longrightarrow} B\to C\to 0$$ in $\Mod R^{op}$ with $C\in \coOmega^{i-1}_{\mathcal{I}_\omega}(\Omega^{i-1}_{\mathcal{F}}(R^{op}))$,
$\Ext^i_{S^{op}}(\omega,\Tor^{R}_i(g,{\omega}))$ is a monomorphism for any $0\leqslant i\leqslant n-1$.
\end{enumerate}
\end{corollary}

In the following, we will establish the left-right symmetry of the $n$-cograde condition.

\begin{lemma}\label{lem-4.4}
Let
$$0\to A\to B\to C\to 0$$ be an exact sequence in $\Mod R$ such that $A$ is superfluous in $B$.
Then the following assertions hold.
\begin{enumerate}
\item Let $L\in \Mod R^{op}$. If $L'\otimes_R C=0$ for any submodule $L'$ of $L$, then $L\otimes_R B=0$.
\item Let $M\in \Mod R$. If $\Hom_{R}(C, M')=0$ for any quotient module $M'$ of $M$, then $\Hom_{R}(B, M)=0$.
\end{enumerate}
\end{lemma}

\begin{proof}
(1) If $L\otimes_R B\neq 0$, then there exists $x\in L$ such that $xR\otimes_RB\neq 0$.
Since $xR\cong R/I$ for some right ideal $I$ of $R$, we have that
$$B/IB\cong R/I\otimes_RB \cong xR\otimes_RB\neq 0$$
and $IB\lneqq B$. In view of the assumption that $A$ is superfluous in $B$, it follows that $IB+A\lneqq B$ and
$$xR\otimes_RC\cong R/I\otimes_RC\cong R/I\otimes_R B/A\cong \frac{B/A}{(IB+A)/A}\cong B/(IB+A)\neq 0.$$
It contradicts the assumption.

(2) If $\Hom_{R}(B, M)\neq 0$, then there exists a non-zero homomorphism $f \in \Hom_{R}(B, M)$.
Pick the kernel $L$ of $f$ such that $\im f$ $\cong$ $B/L$. Because $A$ is superfluous in $B$ and $f \neq 0$,
we have $A+L\lneqq B$. Then there exists a non-zero natural epimorphism $\pi: B/A(\cong C) \twoheadrightarrow B/(A+L)$.
Note that the inclusions $(A+L)/L \subseteq B/L\subseteq M$ induce an embedding homomorphism
$$i: \frac{B/L}{(A+L)/L}(\cong B/(A+L))\hookrightarrow \frac{M}{(A+L)/L}.$$ Then $0\neq i\cdot\pi \in
\Hom_{R}(C,\frac{M}{(A+L)/L})$, which contradicts the assumption.
\end{proof}

It is straightforward to verify the following observation.

\begin{lemma}\label{lem-4.5}
\begin{enumerate}
\item[]
\item If $P\in\Mod R$ is finitely generated projective, then $\pd_{S^{op}}P^*=\mathcal{P}_{\omega}(R)$-$\id_{R}P$.
\item If $Q\in\Mod S^{op}$ is finitely generated projective, then $\pd_{R}Q^*=\mathcal{P}_{\omega}(S^{op})$-$\id_{S^{op}}Q$.
\end{enumerate}
\end{lemma}

\begin{lemma}\label{lem-4.6}
Let $P\in\Mod R$ be finitely generated projective and $t\geqslant 0$. Then the following statements are equivalent.
\begin{enumerate}
\item $\pd_{S^{op}}P^*\leqslant t$.
\item $\mathcal{P}_{\omega}(R)$-$\id_{R}P\leqslant t$.
\item $\Ext^{t+1}_{S^{op}}(\omega,H)\otimes_RP=0$ for any $H\in\Mod S^{op}$.
\item $\Hom_{R}(P,\Tor^{S}_{t+1}(\omega,N))=0$ for any $N\in\Mod S$.
\end{enumerate}
\end{lemma}

\begin{proof}
By Lemma ~\ref{lem-4.5}(1), we have $(1)\Leftrightarrow (2)$.

$(1)\Leftrightarrow (3)$ Let $H\in\Mod S^{op}$ and
$${\bf I}:=\ 0\to  H\to I^0\to I^1\to \cdots \to I^i\to \cdots$$ be an injective resolution of $H$ in $\Mod S^{op}$.
Because $P\in\Mod R$ is finitely generated projective by assumption, the functor $-\otimes_R P$ is exact. Then we have
\begin{align*}
& \ \ \ \  \ \  \ \ \Ext^{t+1}_{S^{op}}(P^*,H)\\
&\ \ \ \ \ \cong H^{t+1}(\Hom_{S^{op}}(P^*,{\bf I}))\\
&\ \ \ \ \  \cong H^{t+1}(\Hom_{S^{op}}(\omega,{\bf I})\otimes_R P)\\
&\ \ \ \ \  \cong H^{t+1}(\Hom_{S^{op}}(\omega,{\bf I}))\otimes_RP\ {\rm (by\ [6,\ p.33,\ Excercise\ 3]})\\
&\ \ \ \ \  \cong \Ext^{t+1}_{S^{op}}(\omega,H)\otimes_R P.
\end{align*}
Now the assertion follows easily.

$(1)\Leftrightarrow (4)$ Since $\pd_{S^{op}}P^*=\fd_{S^{op}}P^*$, the assertion follows from \cite[Lemma 7.6]{TH4}.
\end{proof}

Recall from \cite{N} that a ring $R$ is called {\it semiregular} if $R/J(R)$ is von Neumann regular and idempotents
can be lifted modulo $J(R)$, where $J(R)$ is the Jacobson radical of $R$. The class of semiregular rings includes:
(1) von Neumann regular rings; (2) semiperfect rings; (3) left cotorsion rings; and (4) right cotorsion rings.
See \cite{GH} for the definitions of left cotorsion rings and right cotorsion rings.

If $R$ is a semiregular ring, then any finitely presented left or right $R$-module has a projective cover by \cite[Theorem 2.9]{N}.
In this case, since $_R\omega$ admits a degreewise finite $R$-projective resolution by Definition ~\ref{def-2.1}, we may assume that
$$\cdots \to P_i(\omega) \to \cdots \to P_1(\omega) \to P_0(\omega)\to {_R\omega} \to 0$$ is the minimal projective resolution
of $_R\omega$ in $\mod R$.  Put $\omega_{i}:=\im (P_i(\omega)\to P_{i-1}(\omega))$ for any $i\geqslant 1$ and $\omega_{0}:=\omega$.
Analogously, if $S$ is a semiregular ring, then we assume that
$$\cdots \to Q_i(\omega) \to \cdots \to Q_1(\omega) \to Q_0(\omega)\to {\omega_S} \to 0$$ is the minimal projective resolution
of $\omega_S$ in $\mod S^{op}$. By Lemma~ \ref{lem-4.6}, we have the following

\begin{proposition}\label{prop-4.7}
Let $R$ be a semiregular ring and $m,n\geqslant 1$. Then the following statements are equivalent.
\begin{enumerate}
\item $\pd_{S^{op}}P_i(\omega)^*\leqslant m-1$ for any $0\leqslant i \leqslant n-1$.
\item $\mathcal{P}_{\omega}(R)$-$\id_{R}P_i(\omega)\leqslant m-1$ for any $0\leqslant i \leqslant n-1$.
\item $\sTcograde_\omega \Ext^{m}_{S^{op}}(\omega,N')\geqslant n$ for any $N'\in\Mod S^{op}$.
\item $\sEcograde_\omega \Tor_{m}^{S}(\omega,N)\geqslant n$ for any $N\in\Mod S$.
\end{enumerate}
\end{proposition}

\begin{proof}
By \cite[Proposition 7.7]{TH4} and Lemma ~\ref{lem-4.6}, we have $(4)\Leftrightarrow (1)\Leftrightarrow (2)$.

$(3)\Rightarrow (1)$ We proceed by induction on $n$. Let $N'\in \Mod S^{op}$. Suppose $n=1$. Because
$\sTcograde_\omega \Ext^{m}_{S^{op}}(\omega,N')\geqslant 1$ by (3), we have $L'\otimes_R\omega=0$
for any submodule $L'$ of $\Ext^{m}_{S^{op}}(\omega,N')$ in $\Mod R^{op}$. It follows from Lemma ~\ref{lem-4.4}(1) that
$\Ext^{m}_{S^{op}}(\omega,N')\otimes_RP_0(\omega)=0$. Therefore by Lemma ~\ref{lem-4.6} we get
$\pd_{S^{op}}P_0(\omega)^*\leqslant m-1$ and the case for $n=1$ is proved.

Now suppose $n\geqslant 2$. Let $X$ be a submodule of $\Ext^{m}_{S^{op}}(\omega,N')$ in $\Mod R^{op}$.
By (3), we have $\Tor_{0\leqslant i\leqslant n-1\emph{}}^{R}(X,\omega)=0$. Then for any $0\leqslant i\leqslant n-2$,
we have
$$\Tor_1^{R}(X,\omega_{i})\cong \Tor_{i+1}^{R}(X,\omega)=0.$$
For any $i\geqslant 0$, from the exact sequence
$$0\to \omega_{i+1}\to P_{i}(\omega)\to \omega_{i}\to 0,$$
we get the following exact sequence
$$0\to \Tor_1^R(X,\omega_{i})\to X\otimes_R \omega_{i+1}\to X\otimes_R P_{i}(\omega).\eqno{(4.1)}$$
By the induction hypothesis, we have $\pd_{S^{op}}P_{i}(\omega)^*\leqslant m-1$ for any $0\leqslant i \leqslant n-2$.
Then it follows from Lemma ~\ref{lem-4.6} that $\Ext^{m}_{S^{op}}(\omega,N')\otimes_RP_{n-2}(\omega)=0$
and hence $X\otimes_RP_{n-2}(\omega)=0$. So it is derived from (4.1) that
$X\otimes_R\omega_{n-1}=0$. Notice that $P_{n-1}(\omega)$ is the projective cover of $\omega_{n-1}$,
so $\Ext^{m}_{S^{op}}(\omega,N')\otimes_RP_{n-1}(\omega)=0$ by Lemma ~\ref{lem-4.4}(1).
It follows from Lemma ~\ref{lem-4.6} that $\pd_{S^{op}}P_{n-1}(\omega)^*\leqslant m-1$.

$(1)\Rightarrow (3)$ Let $X$ be a submodule of $\Ext^{m}_{S^{op}}(\omega,N')$ in $\Mod R^{op}$. Then by (1) and Lemma ~\ref{lem-4.6},
we have $\Ext^{m}_{S^{op}}(\omega,N')\otimes_R(\oplus_{i=0}^{n-1}P_i(\omega))=0$,
and hence $X\otimes_R(\oplus_{i=0}^{n-1}P_i(\omega))=0$.
Since $\omega_{i}$ is a quotient module of $P_i(\omega)$ for any $i\geqslant 0$, we then have
$X\otimes_R(\oplus_{i=0}^{n-1}\omega_i)=0$.

If $n=1$, then $X\otimes_R\omega=0$ and
$\sTcograde_\omega \Ext^{m}_{S^{op}}(\omega,N')\geqslant 1$. If $n\geqslant 2$, then
from (4.1) we get $\Tor^{R}_1(X,\oplus_{i=0}^{n-2}\omega_{i})=0$.
Since $\Tor_{i+1}^{R}(X,\omega)\cong \Tor_{1}^{R}(X,\omega_{i})$ for any $i\geqslant 0$, we have that
$\Tor^{R}_{0\leqslant i\leqslant n-1}(X,\omega)=0$ and $\sTcograde_\omega \Ext^{m}_{S^{op}}(\omega,N')\geqslant n$.
\end{proof}

The following result means that the $n$-cograde condition is left-right symmetric.

\begin{theorem} \label{thm-4.8}
Let $R$ be semiregular and $n\geqslant 1$. Then the following statements are equivalent.
\begin{enumerate}
\item $\pd_{S^{op}}P_i(\omega)^*\leqslant i$ for any $0\leqslant i \leqslant n-1$.
\item $\mathcal{P}_{\omega}(R)$-$\id_{R}P_i(\omega)\leqslant i$ for any $0\leqslant i \leqslant n-1$.
\item $\sTcograde_\omega\Ext^i_R(\omega,M)\geqslant i$ for any $M\in \Mod R$ and $1\leqslant i\leqslant n$.
\item $\sEcograde_\omega\Tor_i^S(\omega,N)\geqslant i$ for any $N\in \Mod S$ and $1\leqslant i\leqslant n$.
\item $\sTcograde_\omega\Ext^i_{S^{op}}(\omega,N')\geqslant i$ for any $N'\in \Mod S^{op}$ and $1\leqslant i\leqslant n$.
\item $\sEcograde_\omega\Tor_i^R(M',\omega)\geqslant i$ for any $M'\in \Mod R^{op}$ and $1\leqslant i\leqslant n$.
\item $\Tor^S_i(\omega,\Ext^i_R({\omega},-))$ preserves epimorphisms in $\Mod R$ for $0\leqslant i\leqslant n-1$.
\item $\Ext^i_R(\omega,\Tor^S_i({\omega},-))$ preserves monomorphisms in $\Mod S$ for $0\leqslant i\leqslant n-1$.
\item $\Tor^R_i(\Ext^i_{S^{op}}({\omega},-),\omega)$ preserves epimorphisms in $\Mod S^{op}$ for $0\leqslant i\leqslant n-1$.
\item $\Ext^i_{S^{op}}(\omega,\Tor^{R}_i(-,{\omega}))$ preserves monomorphisms in $\Mod R^{op}$ for $0\leqslant i\leqslant n-1$.
\end{enumerate}
\end{theorem}

\begin{proof}
By Proposition ~\ref{prop-4.7}, we have $(1)\Leftrightarrow (2)\Leftrightarrow (4)\Leftrightarrow (5)$.
By Corollaries ~\ref{cor-4.2} and ~\ref{cor-4.3}, we have $(3)\Leftrightarrow (4)\Leftrightarrow (7)\Leftrightarrow (8)$
and $(5)\Leftrightarrow (6)\Leftrightarrow (9)\Leftrightarrow (10)$.
\end{proof}

As a consequence, we get the following

\begin{corollary} \label{cor-4.9}
Let $R$ and $S$ be semiregular and $n\geqslant 1$. Then the following statements are equivalent.
\begin{enumerate}
\item $\pd_{S^{op}}P_i(\omega)^*\leqslant i$ for any $0\leqslant i \leqslant n-1$.
\item $\pd_{R}Q_i(\omega)^*\leqslant i$ for any $0\leqslant i \leqslant n-1$.
\item $\mathcal{P}_{\omega}(R)$-$\id_{R}P_i(\omega)\leqslant i$ for any $0\leqslant i \leqslant n-1$.
\item $\mathcal{P}_{\omega}(S^{op})$-$\id_{S^{op}}Q_i(\omega)\leqslant i$ for any $0\leqslant i \leqslant n-1$.
\end{enumerate}
\end{corollary}

\begin{proof}
By the symmetric version of Proposition ~\ref{prop-4.7}, we have
$$(2)\Leftrightarrow (4)\Leftrightarrow
\sTcograde_\omega \Ext^{i}_{R}(\omega,M)\geqslant i\ {\rm for\ any}\ M\in\Mod R\ {\rm and}\ 1\leqslant i\leqslant n.$$
Now the assertion follows from Theorem \ref{thm-4.8}.
\end{proof}

\vspace{0.2cm}

\noindent{\bf 4.2. The quasi $n$-cograde condition}

\vspace{0.2cm}

\begin{definition} \label{def-4.10}
{\rm For any $n\geqslant 1$, $\omega$ is said to satisfy the {\it right quasi $n$-cograde condition}
if $\sEcograde_\omega\Tor_{i+1}^S(\omega,N)\geqslant i$ for any $N\in \Mod S$ and $1\leqslant i\leqslant n$;
and $\omega$ is said to satisfy the {\it left quasi $n$-cograde condition}
if $\sEcograde_\omega\Tor_{i+1}^R(M',\omega)\geqslant i$ for any $M'\in \Mod R^{op}$ and $1\leqslant i\leqslant n$.}
\end{definition}

It is trivial that $\omega$ satisfies the right (resp. left) quasi $n$-cograde conditions
if it satisfies the right (resp. left) $n$-cograde condition. But the converse does not hold true
in general, see Subsection 4.4 below.

The following lemma is useful in the sequel.

\begin{lemma} \label{lem-4.11}
For any $n\geqslant 0$, the following assertions hold.
\begin{enumerate}
\item Let $M\in \Mod R$. If $\Ecograde_{\omega}M\geqslant n$
and $\Tcograde_{\omega}\Ext^n_{R}(\omega,M)\geqslant n+1$, then $\Ecograde_{\omega}M\geqslant n+1$.
\item Let $N\in \Mod S$. If $\Tcograde_{\omega}N\geqslant n$
and $\Ecograde_{\omega}\Tor_n^{S}(\omega,N)\geqslant n+1$, then $\Tcograde_{\omega}N\geqslant n+1$.
\end{enumerate}
\end{lemma}

\begin{proof}
We proceed by induction on $n$.

(1) If $n=0$, then $\omega\otimes_S M_*=0$ by assumption. It follows from \cite[Lemma 6.1(1)]{TH2} that $M_*=0$
and $\Ecograde_{\omega}M\geqslant 1$.

Let $n\geqslant 1$. Consider an injective resolution
$$0\to M\to I^0\to \cdots \to I^n \to \cdots$$
of $M$ in $\Mod R$. Put $M'=\im(I^{n-1}\to I^{n})$. Since $\Ecograde_{\omega}M\geqslant n$ by the induction hypothesis,
applying the functor $(-)_*$ to the above exact sequence yields the following exact sequence
$$0\to {I^0}_*\to \cdots \to {I^{n-1}}_* \stackrel{g}{\longrightarrow} {M'}_*\to \Ext^n_R(\omega,M)\to 0$$
in $\Mod S$. Because $\Tcograde_{\omega}\Ext^n_{R}(\omega,M)\geqslant n+1$ by assumption,
we have $\Tor^S_{0\leqslant i\leqslant n}(\omega,\Ext^n_{R}(\omega,M))=0$. Then by \cite[Proposition VI.5.1]{CE},
we have
$$\Ext^i_S(\Ext^n_{R}(\omega,M), {I^j}_*)\cong \Hom_R(\Tor^S_i(\omega,\Ext^n_{R}(\omega,M)),I^j)=0$$
for any $0\leqslant i\leqslant n$ and $j\geqslant 0$, and hence
$$\Ext^1_S(\Ext^n_{R}(\omega,M), \im g)\cong \Ext^n_S(\Ext^n_{R}(\omega,M), {I^0}_*)=0.$$
It implies that the exact sequence
$$0\to \im g\to {M'}_*\to \Ext_R^n(\omega,M)\to 0$$
splits and hence $\Ext_R^n(\omega,M)$ is a direct summand of ${M'}_*$. Since ${M'}_*$
is adjoint 1-$\omega$-cotorsionfree, so is $\Ext_R^n(\omega,M)$. Thus, applying \cite[Proposition 3.2]{TH3},
the T-cograde condition on $\Ext_R^n(\omega,M)$ proves $\Ext_R^n(\omega,M)=0$.
Consequently we have $\Ecograde_{\omega}M \geqslant n+1$ and the assertion follows.

(2) If $n=0$, then $(\omega\otimes_SN)_*=0$ by assumption. It follows from
\cite[Lemma 6.1(2)]{TH2} that $\omega\otimes_SN=0$ and $\Tcograde_{\omega}N\geqslant 1$.

Let $n\geqslant 1$. Consider a projective resolution
$$\cdots \to P_n \to \cdots \to P_0 \to N\to 0$$ of $N$ in $\Mod S$. Put $N'=\im(P_n\to P_{n-1})$.
Since $\Tcograde_{\omega}N\geqslant n$ by the induction hypothesis, applying the functor $\omega\otimes_S-$ to
the above exact sequence yields the following exact sequence
$$0\to \Tor_n^S(\omega,N)\to \omega\otimes_SN'\stackrel{f}{\longrightarrow} \omega\otimes_SP_{n-1}\to \cdots \to \omega\otimes_SP_{0}\to 0$$
in $\Mod R$. Because $\Ecograde_{\omega}\Tor_n^{S}(\omega,N)\geqslant n+1$ by assumption, we have
\linebreak
$\Ext^{0\leqslant i\leqslant n}_{R}(\omega,\Tor_n^{S}(\omega,N))=0$. Notice that $\omega\otimes_SP\in\Add_{R}\omega$
for any projective module $P$ in $\Mod S$, so $\Ext^{0\leqslant i\leqslant n}_{R}(\omega\otimes_SP_{j}, \Tor_n^{S}(\omega,N))=0$
for any $j\geqslant 0$, and hence
$$\Ext^1_{R}(\im f, \Tor_n^{S}(\omega,N))\cong \Ext^n_{R}(\omega\otimes_SP_0, \Tor_n^{S}(\omega,N))=0.$$
It induces an exact sequence
$$\Hom_{R}(\omega\otimes_SN',\Tor_n^{S}(\omega,N))\to \Hom_{R}(\Tor_n^{S}(\omega,N),\Tor_n^{S}(\omega,N))\to 0.$$
Because $\omega\otimes_SN'\in\cT_{\omega}^1(R)$ by \cite[Lemma 6.1(2)]{TH2}, there exists an epimorphism
$U\twoheadrightarrow \omega\otimes_SN'$ in $\Mod R$ with $U\in\Add_{R}\omega$ by \cite[Lemma 3.6(1)]{TH1}.
Because $(\Tor_n^{S}(\omega,N))_*=0$, we have $\Hom_{R}(U,\Tor_n^{S}(\omega,N))=0$. It follows that
$\Hom_{R}(\omega\otimes_SN',\Tor_n^{S}(\omega,N))=0$ and $\Hom_{R}(\Tor_n^{S}(\omega,N),\Tor_n^{S}(\omega,N))=0$,
which implies $\Tor_n^{S}(\omega,N)=0$. So $\Tcograde_{\omega}N \geqslant n+1$ and the assertion follows.
\end{proof}

We have the following equivalent characterizations for $\omega$ satisfying the right quasi $n$-cograde condition.

\begin{proposition} \label{prop-4.12}
For any $n\geqslant 1$, the following statements are equivalent.
\begin{enumerate}
\item $\sEcograde_{\omega}\Tor_{i+1}^{S}(\omega,N)\geqslant i$ for any $N\in \Mod S$ and $1\leqslant i\leqslant n$.
\item $\Tcograde_{\omega}\Ext^i_{R}(\omega,M)\geqslant i$ for any $M\in \Mod R$ and $1\leqslant i\leqslant n$.
\item For any exact sequence
$$0\to A\stackrel{g}{\longrightarrow} B\to C\to 0$$ in $\Mod S$ with $C\in \coOmega^{i-1}_{\mathcal{I}_\omega}(\Omega^{i}_{\mathcal{I}_\omega}(S))$,
$\Ext_{R}^i(\omega, \Tor_i^{S}(\omega,g))$ is a monomorphism for any $0\leqslant i\leqslant n-1$.
\item For any exact sequence
$$0\to A\stackrel{g}{\longrightarrow} B\to C\to 0$$ in $\Mod S$ with $C\in \coOmega^{i-1}_{\mathcal{I}_\omega}(\Omega^{i}_{\mathcal{F}}(S))$,
$\Ext_{R}^i(\omega, \Tor_i^{S}(\omega,g))$ is a monomorphism for any $0\leqslant i\leqslant n-1$.
\item $\Tor^S_i(\omega, \Ext^i_{R}(\omega,f))$ is an epimorphism for any epimorphism $f: B\twoheadrightarrow C$
in $\Mod R$ with $B,C\in \coOmega^{1}_{\mathcal{P}_{\omega}}(R)$ and $0\leqslant i\leqslant n-1$.
\item $\Tor^S_i(\omega, \Ext^i_{R}(\omega,f))$ is an epimorphism for any epimorphism $f: B\twoheadrightarrow C$
in $\Mod R$ with $B,C\in \coOmega^{1}(R)$ and $0\leqslant i\leqslant n-1$.
\item $\coOmega^{i}(R)\subseteq \cT^{i+1}_{\omega}(R)$ for any $1\leqslant i\leqslant n$.
\end{enumerate}
\end{proposition}

\begin{proof}
By Theorems \ref{thm-3.9} and \ref{thm-3.5}, we have $(1)\Leftrightarrow (3)\Leftrightarrow (4)$ and
$(2)\Leftrightarrow (5)\Leftrightarrow (6)\Leftrightarrow (7)$
respectively. In the following, we will prove $(1)\Leftrightarrow (2)$ by induction on $n$.

$(1)\Rightarrow (2)$ Let $M\in \Mod R$. By Lemma ~\ref{lem-3.1}(1), for any $n\geqslant 1$, there exist exact sequences
$$0\to \Ext^n_R(\omega, M)\stackrel{\lambda}{\longrightarrow}
\cTr_{\omega}\coOmega^{n-1}(M)\stackrel{\pi}{\longrightarrow} C\to 0, \eqno{(4.2)}$$
$$0\to C\to I^{n+1}(M)_*\to \cTr_{\omega}\coOmega^{n}(M)\to 0 \eqno{(4.3)}$$
in $\Mod S$ such that $1_\omega\otimes \pi$ is an isomorphism, where $C=I^n(M)_*/\coOmega^n(M)_*$.
Because $I^{n+1}(M)_*\in {\omega_S}^{\top}$ by \cite[Corollary 6.1]{HW},
it follows from the exact sequence (4.3) that $\Tor_i^S(\omega,C)\cong
\Tor_{i+1}^S(\omega,\cTr_{\omega}\coOmega^{n}(M))$ for any $i\geqslant 1$.

If $n=1$, then from the exact sequence (4.2) we get an exact sequence
$$\Tor_{2}^S(\omega, \cTr_{\omega}\coOmega^{1}(M))(\cong\Tor_1^S(\omega,C))\to \omega\otimes_S\Ext^1_{R}(\omega,M)\to 0$$
in $\Mod R$. Because $\sEcograde_{\omega}\Tor_{2}^S(\omega, \cTr_{\omega}\coOmega^{1}(M))\geqslant 1$ by assumption,
we have $\Ecograde_{\omega}\omega\otimes_S\Ext^1_{R}(\omega,M))\geqslant 1$.
It is derived from Lemma ~\ref{lem-4.11}(2) that $\Tcograde_{\omega}\Ext^1_{R}(\omega,M)\geqslant 1$.

Now suppose $n\geqslant 2$. Then $\Tcograde_{\omega}\Ext^i_{R}(\omega,M)\geqslant i$ for any $1\leqslant i\leqslant n-1$
and $\Tcograde_{\omega}\Ext^n_{R}(\omega,M)\geqslant n-1$ by the induction hypothesis. It follows from Theorem ~\ref{thm-3.5} that
$\coOmega^{i}(R)\subseteq\cT^i_{\omega}(R)$ for any $1\leqslant i\leqslant n$. So $\coOmega^{n-1}(M)\in \cT^{n-1}_{\omega}(R)$,
and hence $\cTr_{\omega}\coOmega^{n-1}(M)$ $\in {\omega_S}^{\top_{n-1}}$.
Thus from the exact sequences (4.2) and (4.3) we get the following exact sequence
$$\Tor_{n+1}^S(\omega, \cTr_{\omega}\coOmega^{n}(M))\to \Tor_{n-1}^S(\omega, \Ext^n_{R}(\omega,M))\to 0.$$
By (1), we have $\Ecograde_{\omega}\Tor_{n-1}^S(\omega, \Ext^n_{R}(\omega,M))\geqslant n$. It follows from Lemma ~\ref{lem-4.11}(2) that
$\Tcograde_{\omega}\Ext^n_{R}(\omega,M)\geqslant n$.

$(2) \Rightarrow (1)$ Let $N\in \Mod S$ and $X$ a quotient module of $\Tor_{n+1}^{S}(\omega,N)$ in $\Mod R$,  and let
$\beta:\Tor_{1}^{S}(\omega,\Omega_{\mathcal{F}}^{n}(N))(\cong \Tor_{n+1}^{S}(\omega,N))\twoheadrightarrow X$ be an epimorphism in $\Mod R$.
By Lemma ~\ref{lem-3.1}(2), we have an exact sequence
$$0\to \im (1_{\omega}\otimes f_{n+1})\stackrel{\sigma}{\longrightarrow}
\acTr_\omega \Omega_{\mathcal{F}}^{n}(N) \stackrel{\tau}{\longrightarrow} \Tor_{n+1}^{S}(\omega,N)\to 0$$
in $\Mod R$ such that $\sigma_*$ is an isomorphism. Then we get an exact sequence
$$0\to \Ker f\stackrel{\eta}{\longrightarrow} \acTr_{\omega}\Omega_{\mathcal{F}}^{n}(N)\stackrel{f}{\longrightarrow} X\to 0 \eqno{(4.4)}$$
in $\Mod R$, where $f=\beta\cdot \tau$. It is easy to see that $\eta_*$ is an isomorphism.

Let $n=1$. Because $\Omega_{\mathcal{F}}^{1}(N)\in \acT^1_{\omega}(S)$ by \cite[Corollary 3.5(1)]{TH3},
we have $\acTr_{\omega}\Omega_{\mathcal{F}}^{1}(N)\in{_R\omega}^{\bot_1}$.
Then the exact sequence (4.4) gives that $X_*\cong \Ext_R^{1}(\omega, \Ker f)$. So $\Tcograde_{\omega}X_*\geqslant 1$
by (2), and hence $\Ecograde_{\omega}X\geqslant 1$ by Lemma ~\ref{lem-4.11}(1). The case for $n=1$ is proved.

Now suppose $n\geqslant 2$. Then $\sEcograde_{\omega}\Tor_{i+1}^{S}(\omega,N)\geqslant i$ for any $1\leqslant i\leqslant n-1$
and $\sEcograde_{\omega}\Tor_{n+1}^{S}(\omega,N)\geqslant n-1$ by the induction hypothesis. So $\Ecograde_{\omega}X\geqslant n-1$.

By Theorem ~\ref{thm-3.7}, we have $\Omega^{i}_{\mathcal{F}}(S)\subseteq\acT^i_{\omega}(S)$ for any $1\leqslant i\leqslant n$.
So $\Omega^{n}_{\mathcal{F}}(N)\in \acT^n_{\omega}(S)$ and $\acTr_{\omega}\Omega^{n}_{\mathcal{F}}(N)\in{_R\omega}^{\bot_n}$.
It follows from the exact sequence (4.4) that $\Ext_R^{n-1}(\omega, X)\cong \Ext_R^{n}(\omega, \Ker f)$.
Then by (2), we have $\Tcograde_{\omega}\Ext^{n-1}_{R}(\omega,X)=\Tcograde_{\omega}\Ext^n_{R}(\omega,\Ker f)\geqslant n$.
Thus $\Ecograde_{\omega}X\geqslant n$ by Lemma ~\ref{lem-4.11}(1).
\end{proof}

We also have the following

\begin{proposition} \label{prop-4.13}
For any $n\geqslant 1$, the following statements are equivalent.
\begin{enumerate}
\item $\sTcograde_{\omega}\Ext^{i+1}_{S^{op}}(\omega,N')\geqslant i$ for any $N'\in \Mod S^{op}$ and $1\leqslant i\leqslant n$.
\item $\Ecograde_{\omega}\Tor_i^{R}(M', \omega)\geqslant i$ for any $M'\in \Mod R^{op}$ and $1\leqslant i\leqslant n$.
\item For any exact sequence
$$0\to A'\to B'\stackrel{f'}{\longrightarrow}C'\to 0$$ in $\Mod S^{op}$ with
$A\in \Omega^{i-1}_{\mathcal{P}_\omega}(\coOmega^{i}_{\mathcal{P}_\omega}(S^{op}))$,
$\Tor^R_i(\Ext^i_{S^{op}}(\omega,f'),\omega)$ is an epimorphism for any $0\leqslant i\leqslant n-1$.
\item For any exact sequence
$$0\to A'\to B'\stackrel{f'}{\longrightarrow}C'\to 0$$ in $\Mod S^{op}$ with $A\in \Omega^{i-1}_{\mathcal{P}_\omega}(\coOmega^{i}(S^{op}))$,
$\Tor^R_i(\Ext^i_{S^{op}}(\omega,f'),\omega)$ is an epimorphism for any $0\leqslant i\leqslant n-1$.
\item $\Ext_{S^{op}}^i(\omega, \Tor_i^{R}(g',\omega))$ is a monomorphism for any monomorphism $g': B'\rightarrowtail C'$ in $\Mod R^{op}$
with $B', C'\in \Omega^{1}_{\mathcal{I}_{\omega}}(R^{op})$ and $0\leqslant i\leqslant n-1$.
\item $\Ext_{S^{op}}^i(\omega, \Tor_i^{R}(g',\omega))$ is a monomorphism for any monomorphism $g': B'\rightarrowtail C'$ in $\Mod R^{op}$
with $B', C'\in \Omega^{1}_{\mathcal{F}}(R^{op})$ and $0\leqslant i\leqslant n-1$.
\item $\Omega^{i}_{\mathcal{F}}(R^{op})\subseteq \acT^{i+1}_{\omega}(R^{op})$ for any $1\leqslant i\leqslant n$.
\end{enumerate}
\end{proposition}

\begin{proof}
By the symmetric versions of Theorems \ref{thm-3.8} and \ref{thm-3.7}, we have $(1)\Leftrightarrow (3)\Leftrightarrow (4)$
and $(2)\Leftrightarrow (5)\Leftrightarrow (6)\Leftrightarrow (7)$ respectively.
In the following, we will prove $(1)\Leftrightarrow (2)$ by induction on $n$.

$(1)\Rightarrow (2)$ Let $M'\in \Mod R^{op}$ and let
$$\cdots \to F_{i+1}(M')\stackrel{f_i}{\longrightarrow} F_{i}(M')\to \cdots \stackrel{f_0}{\longrightarrow} F_{0}(M')\to M' \to 0$$
be the minimal flat resolution of $M'$ in $\Mod R^{op}$.
By Lemma ~\ref{lem-3.1}(2), for any $n\geqslant 1$, there exist exact sequences
$$0\to \im (1_{\omega}\otimes f_{n})\stackrel{\sigma}{\longrightarrow}
\acTr_\omega \Omega_{\mathcal{F}}^{n-1}(M') \stackrel{\tau}{\longrightarrow} \Tor_{n}^{R}(M',\omega)\to 0,\eqno{(4.5)}$$
$$0\to \acTr_{\omega}\Omega_{\mathcal{F}}^{n}(M')\to F_{n+1}(M')\otimes_R\omega\to
\im (1_{\omega}\otimes f_{n})\to 0 \eqno{(4.6)}$$
in $\Mod S^{op}$ such that $\sigma_*$ is an isomorphism. Because $F_{n+1}(M')\otimes_R\omega\in{\omega_S}^{\bot}$
by \cite[Corollary 6.1]{HW}, it follows from the exact sequence (4.6) that $\Ext^i_{S^{op}}(\omega,\im (1_{\omega}\otimes f_{n}))
\cong \Ext^{i+1}_{S^{op}}(\omega,\acTr_{\omega}\Omega_{\mathcal{F}}^{n}(M'))$ for any $i\geqslant 1$.

If $n=1$, then from the exact sequence (4.5) we get an exact sequence
$$0\to (\Tor_{1}^{R}(M',\omega))_*\to \Ext^{2}_{S^{op}}(\omega,\acTr_{\omega}\Omega_{\mathcal{F}}^{1}(M'))
(\cong \Ext^1_{S^{op}}(\omega,\im (1_{\omega}\otimes f_{n})))$$
in $\Mod R^{op}$. Because $\sTcograde_{\omega}\Ext^{2}_{S^{op}}(\omega,\acTr_{\omega}\Omega_{\mathcal{F}}^{1}(M'))\geqslant 1$
by assumption, we have $\Tcograde_{\omega}(\Tor_{1}^{R}(M',\omega))_*\geqslant 1$.
It is derived from Lemma ~\ref{lem-4.11}(1) that $\Ecograde_{\omega}\Tor_1^{R}(M',\omega)\geqslant 1$.

Now suppose $n\geqslant 2$. Then $\Ecograde_{\omega}\Tor_i^{R}(M',\omega)\geqslant i$ for any $1\leqslant i\leqslant n-1$
and $\Ecograde_{\omega}\Tor_n^{R}(M',\omega)\geqslant n-1$ by the induction hypothesis. It follows from Theorem ~\ref{thm-3.7} that
$\Omega^{i}_{\mathcal{F}}(R^{op})\subseteq \acT^i_{\omega}(R^{op})$ for any $1\leqslant i\leqslant n$.
So $\Omega^{n-1}_{\mathcal{F}}(M')\in \acT^{n-1}_{\omega}(R^{op})$ and
$\acTr_{\omega}\Omega^{n-1}_{\mathcal{F}}(M')\in {\omega_S}^{\bot_{n-1}}$.
Thus from the exact sequences (4.5) and (4.6) we get the following exact sequence
$$0\to \Ext^{n-1}_{S^{op}}(\omega, \Tor^R_n(M',\omega))\to \Ext^{n+1}_{S^{op}}(\omega,\acTr_{\omega}\Omega^{n}_{\mathcal{F}}(M')).$$
By (1), we have $\Tcograde_{\omega}\Ext^{n-1}_{S^{op}}(\omega, \Tor^R_n(M',\omega))\geqslant n$. It follows from Lemma ~\ref{lem-4.11}(1)
that $\Ecograde_{\omega}\Tor_n^{R}(M',\omega)\geqslant n$.

$(2)\Rightarrow (1)$ Let $N'\in \Mod S^{op}$ and $Y$ a submodule of $\Ext^{n+1}_{S^{op}}(\omega,N')$ in $\Mod R^{op}$,
and let $\alpha: Y\rightarrowtail \Ext^{1}_{S^{op}}(\omega,\coOmega^{n}(N'))(\cong \Ext^{n+1}_{S^{op}}(\omega,N'))$ be a monomorphism in $\Mod R^{op}$.
By Lemma ~\ref{lem-3.1}(1), we have an exact sequence
$$0\to \Ext^{n+1}_{S^{op}}(\omega,N')\stackrel{\lambda}{\longrightarrow} \cTr_\omega \coOmega^{n}(N')
\stackrel{\pi}{\longrightarrow} {I^{n+1}(N')}_*/\coOmega^{n+1}(N')_*\to 0$$
in $\Mod R^{op}$ such that $\pi \otimes 1_{\omega}$ is an isomorphism. Then we get an exact sequence
$$0\to Y \stackrel{g}{\longrightarrow}\cTr_\omega \coOmega^{n}(N')\stackrel{\rho}{\longrightarrow} \Coker g\to 0 \eqno{(4.7)}$$
in $\Mod R^{op}$, where $g=\lambda\cdot\alpha$. It is easy to see that $\rho \otimes 1_{\omega}$ is an isomorphism.

Let $n=1$. Because $\coOmega^{1}(N')\in \cT^1_{\omega}(S^{op})$ by \cite[Lemma 2.5(2)]{TH1}, we have
$\cTr_{\omega}\coOmega^{1}(N')\in{\omega_S}^{\top_1}$. Then the exact sequence (4.7) gives that
$Y\otimes_R \omega\cong \Tor^R_1(\Coker g,\omega)$. So $\Ecograde_{\omega}Y\otimes_R \omega\geqslant 1$ by (2), and hence
$\Tcograde_{\omega}Y\geqslant 1$ by Lemma ~\ref{lem-4.11}(2). The case for $n=1$ is proved.

Now suppose $n\geqslant 2$. Then $\sTcograde_{\omega}\Ext^{i+1}_{S^{op}}(\omega,N')\geqslant i$ for any $1\leqslant i\leqslant n-1$
and $\sTcograde_{\omega}\Ext^{n+1}_{S^{op}}(\omega,N')\geqslant n-1$ by the induction hypothesis. So $\Tcograde_{\omega}Y$ $\geqslant n-1$.

By Theorem ~\ref{thm-3.5}, we have $\coOmega^{i}(R^{op})\subseteq \cT^i_{\omega}(R^{op})$ for any $1\leqslant i\leqslant n$.
So $\coOmega^{n}(N')\in \cT^i_{\omega}(R^{op})$ and $\cTr_\omega \coOmega^{n}(N')\in {_R\omega}^{\top_n}$.
It follows from the exact sequence (4.7) that $\Tor^R_{n-1}(Y,\omega)\cong \Tor^R_{n}(\Coker g,\omega)$. Then by (2),
we have $\Ecograde_{\omega}\Tor^R_{n-1}(Y,\omega)=\Ecograde_{\omega}\Tor^R_{n}(\Coker g,\omega)\geqslant n$.
Thus $\Tcograde_{\omega}Y$ $\geqslant n$ by Lemma ~\ref{lem-4.11}(2).
\end{proof}

Now we are in a position to state the following

\begin{theorem} \label{thm-4.14}
Let $R$ be semiregular and $n\geqslant 1$. Then the following statements are equivalent.
\begin{enumerate}
\item $\pd_{S^{op}}P_i(\omega)^*\leqslant i+1$ for any $0\leqslant i\leqslant n-1$.
\item $\mathcal{P}_{\omega}(R)$-$\id_{R}P_i(\omega)\leqslant i+1$ for any $0\leqslant i\leqslant n-1$.
\item $\sTcograde_{\omega}\Ext^{i+1}_{S^{op}}(\omega,N')\geqslant i$ for any $N'\in \Mod S^{op}$ and $1\leqslant i\leqslant n$.
\item $\sEcograde_{\omega}\Tor_{i+1}^{S}(\omega,N)\geqslant i$ for any $N\in \Mod S$ and $1\leqslant i\leqslant n$.
\item $\Tcograde_{\omega}\Ext^i_{R}(\omega,M)\geqslant i$ for any $M\in \Mod R$ and $1\leqslant i\leqslant n$.
\item $\Ecograde_{\omega}\Tor_i^{R}(M', \omega)\geqslant i$ for any $M'\in \Mod R^{op}$ and $1\leqslant i\leqslant n$.
\item $\Tor^S_i(\omega, \Ext^i_{R}(\omega,f))$ is an epimorphism for any epimorphism $f: B\twoheadrightarrow C$
in $\Mod R$ with $B,C\in \coOmega^{1}(R)$ and $0\leqslant i\leqslant n-1$.
\item $\Ext_{S^{op}}^i(\omega, \Tor_i^{R}(f',\omega))$ is a monomorphism for any monomorphism $f': B'\rightarrowtail C'$ in $\Mod R^{op}$
with $B', C'\in \Omega^{1}_{\mathcal{F}}(R^{op})$ and $0\leqslant i\leqslant n-1$.
\item For any exact sequence
$$0\to A'\to B'\stackrel{g'}{\longrightarrow}C'\to 0$$ in $\Mod S^{op}$ with $A'\in \Omega^{i-1}_{\mathcal{P}_\omega}(\coOmega^{i}(S^{op}))$,
$\Tor^R_i(\Ext^i_{S^{op}}(\omega,g'),\omega)$ is an epimorphism for any $0\leqslant i\leqslant n-1$.
\item For any exact sequence
$$0\to A\stackrel{g}{\longrightarrow} B\to C\to 0$$ in $\Mod S$ with $C\in \coOmega^{i-1}_{\mathcal{I}_\omega}(\Omega^{i}_{\mathcal{F}}(S))$,
$\Ext_{R}^i(\omega, \Tor_i^{S}(\omega,g))$ is a monomorphism for any $0\leqslant i\leqslant n-1$.
\item $\coOmega^{i}(R)\subseteq \cT^{i+1}_{\omega}(R)$ for any $1\leqslant i\leqslant n$.
\item $\Omega^{i}_{\mathcal{F}}(R^{op})\subseteq \acT^{i+1}_{\omega}(R^{op})$ for any $1\leqslant i\leqslant n$.
\end{enumerate}
\end{theorem}

\begin{proof}
By Proposition ~\ref{prop-4.7}, we have $(1)\Leftrightarrow (2) \Leftrightarrow (3)\Leftrightarrow (4)$.
By Propositions ~\ref{prop-4.12} and ~\ref{prop-4.13},
we have $(4)\Leftrightarrow (5)\Leftrightarrow (7)\Leftrightarrow (10)\Leftrightarrow (11)$ and
$(3)\Leftrightarrow (6)\Leftrightarrow (8)\Leftrightarrow (9)\Leftrightarrow(12)$ respectively.
\end{proof}

For the right quasi 1-cograde condition, we have some additional interesting equivalent characterizations.

\begin{proposition} \label{prop-4.15}
Let $R$ be a semiregular ring. Then the following statements are equivalent.
\begin{enumerate}
\item $\pd_{S^{op}}P_0(\omega)^*\leqslant 1$.
\item $\sEcograde_{\omega}\Tor_2^{S}(\omega,N)\geqslant 1$ for any $N\in \Mod S$.
\item $\theta_M$ is a superfluous epimorphism for any $M\in\coOmega^{1}(R)$.
\item $\mu_{M'}$ is an essential monomorphism for any $M'\in\Omega^{1}_{\mathcal{F}}(R^{op})$.
\end{enumerate}
\end{proposition}

\begin{proof}
By Theorem \ref{thm-4.14}, we have $(1)\Leftrightarrow (2)$.

$(1)\Rightarrow (3)$ Let $M\in\coOmega^{1}(R)$. By \cite[Lemma 2.5(2)]{TH1}, we have $\coOmega^{1}(R)\subseteq\cT^{1}_{\omega}(R)$.
So $M\in\cT^{1}_{\omega}(R)$ and $\theta_M$ is an epimorphism. Because
$\Ker \theta_M\cong$ \linebreak $\Tor_2^S(\omega, \cTr_{\omega}M)$ by \cite[Proposition 3.2]{TH1}, we have
$$\Hom_R(P_0(\omega), \Ker \theta_M)\cong \Hom_R(P_0(\omega), \Tor_2^S(\omega, \cTr_{\omega}M))=0$$
by (1) and Lemma ~\ref{lem-4.6}.
It follows easily that $X_{*}=0$ for any quotient module $X$ of $\Ker \theta_M$. Let $A$ be a submodule of $\omega\otimes_S M_*$ in $\Mod R$
such that $\Ker \theta_M+A=\omega\otimes_S M_*$. Then $(\Ker \theta_M+A)/A(\cong \Ker \theta_M/(A\cap \Ker \theta_M))$
is isomorphic to a quotient module of $\Ker \theta_M$, and so $((\Ker \theta_M+A)/A)_*=0$. Since $\omega\otimes_S M_*\in\cT^{1}_{\omega}(R)$
by \cite[Lemma 6.1(2)]{TH2}, $(\Ker \theta_M+A)/A\in\cT^{1}_{\omega}(R)$ by \cite[Corollary 3.8]{TH1}. It follows that
$\theta_{(\Ker \theta_M+A)/A}:\omega\otimes_S((\Ker \theta_M+A)/A)_*\to (\Ker \theta_M+A)/A$ is epic and $(\Ker \theta_M+A)/A=0$.
It induces that $A=\Ker \theta_M+A=\omega\otimes_S M_*$ and $\theta_M$ is a superfluous epimorphism.

$(3)\Rightarrow (2)$ Let $f:B\twoheadrightarrow C$ be an epimorphism in $\Mod R$ with $B,C\in\coOmega^{1}(R)(\subseteq\cT^{1}_{\omega}(R))$.
Then $\theta_C\cdot(1_{\omega}\otimes f_*)=f\cdot\theta_B$ is epic.
Because $\theta_C$ is a superfluous epimorphism by (3), it follows from \cite[Corollary 5.15]{AF}
that $1_{\omega}\otimes f_*$ is epic. Now the assertion follows from Theorem \ref{thm-4.14}.

$(1)\Rightarrow (4)$ Let $M'\in\Omega^{1}_{\mathcal{F}}(R^{op})$. By \cite[Corollary 3.5(1)]{TH3},
we have $\Omega^{1}_{\mathcal{F}}(R^{op})\subseteq\acT^{1}_{\omega}(R^{op})$. So $M'\in\acT^{1}_{\omega}(R^{op})$
and $\mu_{M'}$ is a monomorphism. Because $\Coker \mu_{M'}\cong \Ext^2_{S^{op}}(\omega, \acTr_{\omega}M')$
by \cite[Proposition 3.2]{TH3}, we have
$$\Coker \mu_{M'}\otimes_RP_0(\omega)\cong \Ext^{2}_{S^{op}}(\omega,\acTr_{\omega}M')\otimes_RP_0(\omega)=0$$
by (1) and Lemma ~\ref{lem-4.6}.
It follows easily that $Y\otimes_R \omega=0$ for any submodule $Y$ of $\Coker \mu_{M'}$. Let $A'$ be a submodule of $(M'\otimes_R \omega)_*$ in $\Mod R^{op}$
with $A'\cap M'=0$. Then $A'\cong A'/A'\cap M'\cong (A'+M')/M'$ is isomorphic to a submodule of $\Coker \mu_{M'}$, and so $A'\otimes_R \omega=0$.
Since $(M'\otimes_R \omega)_*\in\acT^{1}_{\omega}(R^{op})$ by \cite[Lemma 6.1(1)]{TH2}, $A'\in\acT^{1}_{\omega}(R^{op})$ by \cite[Corollary 3.3(1)]{TH3}.
It follows that $\mu_{A'}:A'\to (A'\otimes_R \omega)_*$ is monic, It induces that $A'=0$ and $\mu_{M'}$ is an essential monomorphism.

$(4)\Rightarrow (2)$ Let $g: B'\rightarrow C'$ be a monomorphism in $\Mod R^{op}$
with $B',C'\in\Omega^{1}_{\mathcal{F}}(R^{op})(\subseteq\acT^{1}_{\omega}(R^{op}))$.
Then $(g\otimes 1_\omega)_*\cdot \mu_{B'}=\mu_{C'}\cdot g$ is monic. Because $\mu_{B'}$ is an essential monomorphism by (4),
it follows from \cite[Corollary 5.13]{AF} that $(g\otimes  1_{\omega})_*$ is monic. Now the assertion follows from Theorem \ref{thm-4.14}.
\end{proof}

\vspace{0.2cm}

\noindent{\bf 4.3. The equivalence of certain cograde condition of modules}

\vspace{0.2cm}

We have the following facts: for the strong $\Tor$-cograde condition of modules
in Theorem ~\ref{thm-3.8}(1) and the strong $\Ext$-cograde condition of modules in Theorem ~\ref{thm-3.9}(1),
they are equivalent when $k=0$ by Theorems ~\ref{thm-4.8}; but they are not equivalent when $k=1$
by Theorem ~\ref{thm-4.14} and Subsection 4.4 below. Also from Theorem ~\ref{thm-4.14} and Subsection 4.4 below
we know that the $\Tor$-cograde condition of modules
in Theorem ~\ref{thm-3.5}(1) and the $\Ext$-cograde condition of modules in Theorem ~\ref{thm-3.7}(1) are
not equivalent when $k=0$. In this subsection, we will show that these two cograde conditions of modules are equivalent when $k=1$.

For any $i\geqslant 1$, by \cite[Proposition 3.8]{TH3} we have $\acT^i_{\omega}(S)\subseteq\Omega^{i}_{\mathcal{I}_{\omega}}(S)$.
The following result characterizes when they are identical.

\begin{proposition} \label{prop-4.16}
For any $n\geqslant 1$, the following statements are equivalent.
\begin{enumerate}
\item $\Ecograde_\omega\Tor_{i}^S(\omega,N)\geqslant i-1$ for any $N\in \coOmega^i_{\mathcal{A}}(S)$ and $1\leqslant i\leqslant n$.
\item $\Ecograde_\omega\Tor_{i}^S(\omega,N)\geqslant i-1$ for any $N\in \coOmega^i_{\mathcal{I}_{\omega}}(S)$ and $1\leqslant i\leqslant n$.
\item $\acT^i_{\omega}(S)=\Omega^{i}_{\mathcal{A}}(S)$ for any $1\leqslant i\leqslant n$.
\item $\acT^i_{\omega}(S)=\Omega^{i}_{\mathcal{I}_{\omega}}(S)$ for any $1\leqslant i\leqslant n$.
\end{enumerate}
\end{proposition}

\begin{proof}
Because $\mathcal{I}_{\omega}(S)\subseteq\mathcal{A}_{\omega}(S)$, we have $(1)\Rightarrow (2)$.
By Lemma ~\ref{lem-3.6}(2), we have $(3)\Leftrightarrow (4)$.

$(2)\Rightarrow (4)$ By \cite[Proposition 3.8]{TH3}, it suffices to prove $\Omega^{i}_{\mathcal{I}_{\omega}}(S)\subseteq\acT^i_{\omega}(S)$
for any $1\leqslant i\leqslant n$. We proceed by induction on $n$. The case for $n=1$ follows from Lemma \ref{lem-2.9}(1).

Now let $N\in \Omega^{n}_{\mathcal{I}_{\omega}}(S)$ with $n\geqslant 2$ and let
$$0\longrightarrow N\stackrel{f^0}{\longrightarrow}I^0 \stackrel{f^{1}}{\longrightarrow} \cdots \stackrel{f^{n-1}}{\longrightarrow} I^{n-1} \eqno{(4.8)}$$
be an exact sequence in $\Mod S$ with all $I^i$ in $\mathcal{I}_\omega(S)$. By the induction hypothesis, we have
$\im f^1\in \acT^{n-1}_{\omega}(S)$. Applying the functor $\omega\otimes_S-$ to (4.8) gives an exact sequence
$$0\to \Tor_n^S(\omega, \Coker f^{n-1})\longrightarrow \omega\otimes_S N \stackrel{1_\omega\otimes f^0}{\longrightarrow} \omega\otimes_S I^0
\longrightarrow \omega\otimes_S \im f^1 \to 0\eqno{(4.9)}$$
in $\Mod R$. Set $M:=\im ({1_\omega\otimes f^0)}$ and let $1_\omega\otimes f^0:=\alpha\cdot\pi$
(where $\pi:\omega\otimes_S N\twoheadrightarrow M$ and $\alpha:M\hookrightarrow \omega\otimes_S I^0$) be the natural epic-monic
decomposition of $1_\omega\otimes f^0$. Then we have the following commutative diagram with exact rows
$$\xymatrix{0\ar[r] & N\ar[r]^{f^0} \ar@{-->}[d]^{g}& I^0 \ar[r] \ar[d]^{\mu_{I^0}} & \im f^1 \ar[r] \ar[d]^{\mu_{\im f^1}} & 0\\
0 \ar[r]  & (M)_*  \ar[r]^{\alpha_*} & (\omega\otimes_S I^0)_* \ar[r] & (\omega\otimes_S \im f^1)_* \ar[r] & \Ext^1_R(\omega, M) \ar[r]& 0.\\
&  &  &{\rm Diagram}\ (4.10) &  }$$
Since $\mu_{\im f^1}$ is a monomorphism by the above argument,
it follows from the snake lemma that $g$ is an epimorphism. On the other hand, we have
$$\alpha_* \cdot \pi_* \cdot \mu_N=(\alpha\cdot\pi)_* \cdot \mu_N=(1_\omega\otimes f^0)_* \cdot \mu_N=\mu_{I^0} \cdot f^0= \alpha_* \cdot g.$$
As $\alpha_* $ is monic, we get that $\pi_* \cdot \mu_N=g$ and $\pi_*$ is epic.
Consider the following commutative diagram with exact rows
$$\xymatrix{ & &  N\ar[d]^{\mu_N} \ar@{=}[r]& N\ar[d]^{g} & & \\
0 \ar[r] & (\Tor_n^S(\omega, \Coker f^{n-1}))_* \ar[r] & (\omega\otimes_S N)_*
\ar[r]^{\pi_*}& {M}_*\ar[r] & 0. \\
&  &{\rm Diagram}\ (4.11) & }$$ Because $(\Tor_n^S(\omega, \Coker f^{n-1}))_*=0$ by assumption, we have that $\pi_*$ is an isomorphism. So
$\mu_{N}$ is epic by the diagram (4.11), and hence an isomorphism. Thus $N\in\acT^2_{\omega}(S)$ and the case for $n=2$ follows.

Now suppose $n\geqslant 3$. By the induction hypothesis, we have that $\im f^1\in \acT^{n-1}_{\omega}(S)$ and $\mu_{\im f^1}$ is an isomorphism.
So $\Ext^1_R(\omega, M)=0$ by the diagram (4.10). In addition, we have $\omega\otimes_S \im f^1\in {_R\omega}^{\bot_{n-3}}$
by \cite[Corollary 3.3(3)]{TH3}. Because $\Ecograde_\omega\Tor_n^S(\omega,\Coker f^{n-1})\geqslant n-1$ (by assumption)
and $\omega\otimes_S I^0\in{_R\omega}^{\bot}$, applying the dimension shifting to (4.9) we obtain
$\omega\otimes_S N\in{_R\omega}^{\bot_{n-2}}$.
Therefore we conclude that $N\in \acT^n_{\omega}(S)$ by \cite[Corollary 3.3(3)]{TH3} again.

$(3)\Rightarrow (1)$ We proceed by induction on $n$. The case for $n=1$ is trivial. Let $N\in \coOmega^{n}_{\mathcal{A}}(S)$
with $n\geqslant 2$. Then there exists an exact sequence
$$0\to H\to A_{n-1}\stackrel{f}{\longrightarrow} A_{n-2}\to \cdots \to A_0\to N\to 0$$
in $\Mod S$ with all $A_i$ in $\mathcal{A}_\omega(S)$. By (3), we have $H\in \acT^n_{\omega}(S)$. By the induction hypothesis,
we have that $\Ecograde_\omega\Tor_i^S(\omega,N)\geqslant i-1$ for any $1\leqslant i\leqslant n-1$ and
$\Ecograde_\omega\Tor_n^S(\omega,N)\geqslant n-2$.

Put $M:=\Ker (1_{\omega}\otimes f)$. Because $A_i\in \acT_{\omega}(S)$ by \cite[Theorem 3.11(1)]{TH3}, we obtain that
$M_*\cong H(\in \acT^n_{\omega}(S))$ and $M\in{_R\omega}^{\bot_{n-2}}$.
By \cite[Proposition 5.1]{TH4}, we have the following exact sequences
$$0\to \Tor_n^S(\omega,N)(\cong \Tor_2^S(\omega,\Coker f))\to \omega\otimes_SM_*\stackrel{\pi}{\longrightarrow} \im \theta_M\to 0, \eqno{(4.12)}$$
$$0\to \im \theta_M \stackrel{\lambda}{\longrightarrow} M \to \Tor_{n-1}^S(\omega,N)(\cong \Tor_1^S(\omega,\Coker f))\to 0 \eqno{(4.13)}$$
such that $\theta_M=\lambda\cdot \pi$. Since $\mu_{M_*}$ is an isomorphism, it follows from \cite[Lemma 6.1(1)]{TH2} that $(\theta_M)_*$
is also an isomorphism. Then both $\lambda_*$ and $\pi_*$ are isomorphisms.

From the exact sequence (4.13), we get $\im \theta_M\in{_R\omega}^{\bot_{n-2}}$.
Because $\omega\otimes_SM_*\in{_R\omega}^{\bot_{n-2}}$
by \cite[Corollary 3.3]{TH3}, from the exact sequence (4.12) it yields that
\linebreak
$\Ext_R^{n-2}(\omega, \Tor_n^S(\omega,N))=0$. Thus we have $\Ecograde_\omega\Tor_n^S(\omega,N)\geqslant n-1$.
\end{proof}

For any $i\geqslant 1$, by \cite[Proposition 3.7]{TH1} we have $\cT^i_{\omega}(R)\subseteq \coOmega^{i}_{\mathcal{P}_{\omega}}(R)$.
The following result characterizes when they are identical.

\begin{proposition} \label{prop-4.17}
For any $n\geqslant 1$, the following statements are equivalent.
\begin{enumerate}
\item $\Tcograde_\omega\Ext^i_R(\omega,M)\geqslant i-1$ for any $M\in \Omega^{i}_{\mathcal{B}}(R)$ and
$1\leqslant i\leqslant n$.
\item $\Tcograde_\omega\Ext^i_R(\omega,M)\geqslant i-1$ for any $M\in \Omega^{i}_{\mathcal{F}_{\omega}}(R)$ and
$1\leqslant i\leqslant n$.
\item $\Tcograde_\omega\Ext^i_R(\omega,M)\geqslant i-1$ for any $M\in \Omega^{i}_{\mathcal{P}_{\omega}}(R)$ and
$1\leqslant i\leqslant n$.
\item $\cT^i_{\omega}(R)=\coOmega^{i}_{\mathcal{B}}(R)$ for any $1\leqslant i\leqslant n$.
\item $\cT^i_{\omega}(R)=\coOmega^{i}_{\mathcal{F}_{\omega}}(R)$ for any $1\leqslant i\leqslant n$.
\item $\cT^i_{\omega}(R)=\coOmega^{i}_{\mathcal{P}_{\omega}}(R)$ for any $1\leqslant i\leqslant n$.
\end{enumerate}
\end{proposition}

\begin{proof}
Because ${\mathcal{B}_{\omega}}(R)\supseteq {\mathcal{F}_{\omega}}(R)\supseteq{\mathcal{P}_{\omega}}(R)$,
we have $(1)\Rightarrow (2)\Rightarrow (3)$.
By Lemma ~\ref{lem-3.4}(2), we have $(4)\Leftrightarrow (5)\Leftrightarrow (6)$.

$(3)\Rightarrow (6)$ By \cite[Proposition 3.7]{TH1}, it suffices to prove $\coOmega^{i}_{\mathcal{P}_{\omega}}(R)\subseteq\cT^i_{\omega}(R)$
for any $1\leqslant i\leqslant n$. We proceed by induction on $n$. The case for $n=1$ follows from Lemma \ref{lem-2.9}(2).

Now let $M\in \coOmega^{n}_{\mathcal{P}_{\omega}}(R)$ with $n\geqslant 2$ and let
$$W_{n-1}\stackrel{f_{n-1}}{\longrightarrow} \cdots \rightarrow W_{1}\stackrel{f_{1}}{\longrightarrow}
W_{0}\stackrel{f_{0}}{\longrightarrow} M\rightarrow 0\eqno{(4.14)}$$ be an exact sequence in $\Mod R$
with all $W_i$ in $\mathcal{P}_\omega(R)$. By the induction hypothesis, we have
$\im f_1\in \cT^{n-1}_{\omega}(R)$. Applying the functor $(-)_*$ to (4.14) gives an exact sequence
$$0\to (\im {f_1})_*\to {W_0}_*\stackrel{{f_0}_*}{\longrightarrow} M_*
\to \Ext^n_R(\omega, \Ker f_{n-1})\to 0.\eqno{(4.15)}$$
Set $N:=\im({f_0}_*)$ and let ${f_0}_*:=\alpha\cdot\pi$ (where $\pi:{W_0}_*\twoheadrightarrow N$ and $\alpha:N\hookrightarrow M_*$)
be the natural epic-monic decompositions of ${f_0}_*$.
Then we have the following commutative diagram with exact rows
$$\xymatrix{0\ar[r]& \Tor_1^S(\omega,N)\ar[r] & \omega\otimes_S(\im f_1)_*
\ar[r] \ar[d]^{\theta_{\im f_1}} & \omega\otimes_S{W_0}_*
\ar[r]^{1_{\omega}\otimes \pi} \ar[d]^{\theta_{W_0}}& \omega\otimes_SN \ar@{-->}[d]^g \ar[r] & 0\\
&0 \ar[r] &  \im f_1 \ar[r] & W_0 \ar[r]^{f_0} &  M \ar[r]& 0.\\
&  & {\rm Diagram}\ (4.16) &  }$$
So we have
$$\theta_M\cdot(1_{\omega}\otimes \alpha)\cdot(1_{\omega}\otimes \pi)
=\theta_M\cdot (1_{\omega}\otimes {f_0}_*)=f_0\cdot\theta_{W_0}=g\cdot(1_{\omega}\otimes \pi).$$
Because $1_{\omega}\otimes \pi$ is epic, we have $\theta_M\cdot (1_{\omega}\otimes \alpha)=g$
and the following commutative diagram with exact rows
$$\xymatrix{
\omega\otimes_SN
\ar[d]^g \ar[r]^{1_\omega\otimes \alpha}& \omega\otimes_SM_* \ar[r]
\ar[d]^{\theta_{M}}& \omega\otimes_S\Ext^n_R(\omega,\Ker f_{n-1}) \ar[r]& 0\\
M \ar@{=}[r]& M.\\
&\ \ \  \ \ \ \ \ \ \ \ \ \ \ \ \ {\rm Diagram}\ (4.17)}$$
Since $\theta_{\im f_1}$ is an epimorphism by the above argument,
it follows from the snake lemma that $g$ is an isomorphism. Thus $1_{\omega}\otimes \alpha$ is a monomorphism.
Because $\omega\otimes_S\Ext^n_R(\omega,\Ker f_{n-1})=0$ by assumption,
we have that $\theta_{M}$ is an isomorphism and $M\in\cT^2_{\omega}(R)$ by the diagram (4.17).
It means that the assertion holds true for $n=2$.
If $n\geqslant 3$, then the fact that $\im f_1\in \cT^{n-1}_{\omega}(R)$ implies $\theta_{\im f_1}$ is an isomorphism.
So $\Tor^S_1(\omega,N)=0$ by the diagram (4.16). In addition, we have $(\im f_1)_*\in{\omega_S}^{\top_{n-3}}$
by \cite[Corollary 3.4(3)]{TH1}. Because $\Tcograde_\omega\Ext^n_R(\omega, \Ker f_{n-1})\geqslant n-1$ by assumption,
applying the dimension shifting to (4.15) we obtain $M_*\in{\omega_S}^{\top_{n-2}}$.
Therefore we conclude that $M\in \cT^n_{\omega}(R)$ by \cite[Corollary 3.4(3)]{TH1} again.

$(4)\Rightarrow (1)$
We proceed by induction on $n$. The case for $n=1$ is trivial.
Let $M\in \Omega^{n}_{\mathcal{B}}(R)$ with $n\geqslant 2$ and let
$$0\to M\to B_{n-1}\to \cdots \to B_1 \stackrel{f}{\longrightarrow} B_0\to L\to 0$$ be an exact sequence
with all $B_i$ in $\mathcal{B}_\omega(R)$. By (4), we have $L\in \cT^n_{\omega}(R)$. By the induction hypothesis, we have
$\Tcograde_\omega\Ext^i_R(\omega,M)\geqslant i-1$ for any $1\leqslant i\leqslant n-1$ and
$\Tcograde_\omega\Ext^n_R(\omega,M)\geqslant n-2$.

Put $N:=\cTr_\omega\Ker f$. Because $B_i\in \cT_{\omega}(R)$ by \cite[Theorem 3.9]{TH1}, we obtain that
$\omega\otimes_S N\cong L(\in \cT^n_{\omega}(R))$ and $N\in{\omega_S}^{\top_{n-2}}$.
By \cite[Proposition 6.7]{TH2}, we have the following exact sequences
$$0\to \Ext^{n-1}_R(\omega,M)\to N\stackrel{\pi}{\longrightarrow} \im \mu_N\to 0, \eqno{(4.18)}$$
$$0\to \im \mu_N \stackrel{\lambda}{\longrightarrow} (\omega\otimes_S N)_* \to  \Ext^{n}_R(\omega,M)\to 0 \eqno{(4.19)}$$
such that $\mu_N =\lambda\cdot \pi$. Since $\theta_{\omega\otimes_S N}$ is an isomorphism,
it follows from \cite[Lemma 6.1(2)]{TH2} that $1_{\omega}\otimes \mu_N$ is also an isomorphism.
Then both $1_{\omega}\otimes \lambda$ and $1_{\omega}\otimes\pi$ are isomorphisms.

From the exact sequence (4.18), we get $\im \mu_N\in{\omega_S}^{\top_{n-2}}$.
Because $(\omega\otimes_S N)_*\in{\omega_S}^{\top_{n-2}}$
by \cite[Corollary 3.4]{TH1}, from the exact sequence (4.19) it yields that
\linebreak
$\Tor_{n-2}^S(\omega, \Ext^n_R(\omega,M))=0$.
Thus we have $\Tcograde_\omega\Ext^n_R(\omega,M)\geqslant n-1$.
\end{proof}

\begin{lemma} \label{lem-4.20}
For any $n\geqslant 0$, the following statements are equivalent.
\begin{enumerate}
\item $\omega\otimes\Ext^{2}_{R}(\omega,-)$ vanishes on $\Mod R$.
\item $(\Tor_{2}^{S}(\omega,-))_*$ vanishes on $\Mod S$.
\item $M_*\in \acT^2_{\omega}(S)$ for any $M\in \Mod R$.
\item $\omega\otimes_SN\in \cT^2_{\omega}(R)$ for any $N\in \Mod S$.
\end{enumerate}
\end{lemma}

\begin{proof}
By \cite[Corollary 6.6]{TH4}, we have $(3)\Leftrightarrow (4)$.

$(1)\Leftrightarrow (4)$ Assume that (1) holds true.
Let $N\in \Mod S$. By \cite[Lemma 6.1(2)]{TH2}, we have
$$\theta_{\omega\otimes_SN}\cdot(1_{\omega}\otimes\mu_{N})=1_{\omega\otimes_SN}.$$
It follows that $\theta_{\omega\otimes_SN}$ is a split epimorphism and
\begin{align*}
& \ \ \ \  \ \  \ \ \Ker \theta_{\omega\otimes_SN}\cong\Coker (1_{\omega}\otimes\mu_{N})\cong \omega\otimes_S\Coker\mu_N\\
&\ \ \ \ \ \cong\omega\otimes_S\Ext^{2}_{R}(\omega,\acTr_{\omega}N)\ \ \text{(by \cite[Corollary 5.2(2)]{TH4})}\\
&\ \ \ \ \ =0\ \ \text{(by (1)}).
\end{align*}
So $\theta_{\omega\otimes_SN}$ is a monomorphism, and hence an isomorphism.

Conversely, assume that (4) holds true. Let $M\in\Mod R$. By \cite[Lemma 6.1(2)]{TH2} again, we have
$$\theta_{\omega\otimes_S\cTr_{\omega}M}\cdot(1_{\omega}\otimes\mu_{\cTr_{\omega}M})=1_{\omega\otimes_S\cTr_{\omega}M}.$$
It follows that
\begin{align*}
& \ \ \ \ \ \ \ \omega\otimes_S\Ext^{2}_{R}(\omega,M)\cong\omega\otimes_S\Coker\mu_{\cTr_{\omega}M}\ \ \text{(by \cite[Corollary 5.3(2)]{TH4})}\\
&\ \ \ \cong\Coker(1_{\omega}\otimes\mu_{\cTr_{\omega}M})\cong \Ker\theta_{\omega\otimes_S\cTr_{\omega}M}\\
&\ \ \ =0\ \ \text{(by (4)}).
\end{align*}

$(2)\Leftrightarrow (3)$ Assume that (2) holds true.
Let $M\in \Mod R$. By \cite[Lemma 6.1(1)]{TH2}, we have
$$(\theta_M)_*\cdot\mu_{M_*}=1_{M_*}.$$
It follows that $\mu_{M_*}$ is a split monomorphism and
\begin{align*}
& \ \ \ \  \ \  \ \ \Coker \mu_{M_*}\cong\Ker (\theta_M)_*\cong (\Ker \theta_M)_*\\
&\ \ \ \ \ \cong(\Tor_{2}^{S}(\omega,\cTr_{\omega}M))_*\ \ \text{(by \cite[Proposition 3.2]{TH1})}\\
&\ \ \ \ \ =0\ \ \text{(by (2)}).
\end{align*}
So $\mu_{M_*}$ is an epimorphism, and hence an isomorphism.

Conversely, assume that (3) holds true. Let $N\in\Mod S$. By \cite[Lemma 6.1(1)]{TH2} again, we have
$$(\theta_{\acTr_{\omega}N})_*\cdot\mu_{(\acTr_{\omega}N)_*}=1_{(\acTr_{\omega}N)_*}.$$
It follows that
\begin{align*}
& \ \ \ \ \ \ \ (\Tor_{2}^{S}(\omega,N))_*\cong(\Ker \theta_{\acTr_{\omega}N})_*\ \ \text{(by \cite[Corollary 5.3(1)]{TH4})}\\
&\ \ \ \cong\Ker (\theta_{\acTr_{\omega}N})_*\cong \Coker\mu_{(\acTr_{\omega}N)_*}\\
&\ \ \ =0\ \ \text{(by (3)}).
\end{align*}
\end{proof}

The following result establishes the left-right symmetry of certain cograde condition of modules.

\begin{theorem} \label{thm-4.19}
For any $n\geqslant 1$, the following statements are equivalent.
\begin{enumerate}
\item $\Tcograde_{\omega}\Ext^{i}_{R}(\omega,M)\geqslant i-1$ for any $M\in \Mod R$ and $1\leqslant i\leqslant n$.
\item $\Ecograde_{\omega}\Tor_{i}^{S}(\omega,N)\geqslant i-1$ for any $N\in \Mod S$ and $1\leqslant i\leqslant n$.
\item $\coOmega^i(R)\subseteq\cT^i_{\omega}(R)=\coOmega^{i}_{\mathcal{B}}(R)$
for any $1\leqslant i\leqslant n$.
\item $\coOmega^i(R)\subseteq\cT^i_{\omega}(R)=\coOmega^{i}_{\mathcal{F}_{\omega}}(R)$
for any $1\leqslant i\leqslant n$.
\item $\coOmega^i(R)\subseteq\cT^i_{\omega}(R)=\coOmega^{i}_{\mathcal{P}_{\omega}}(R)$
for any $1\leqslant i\leqslant n$.
\item $\Omega^i_{\mathcal{F}}(S)\subseteq\acT^i_{\omega}(S)=\Omega^{i}_{\mathcal{A}}(S)$
for any $1\leqslant i\leqslant n$.
\item $\Omega^i_{\mathcal{F}}(S)\subseteq\acT^i_{\omega}(S)=\Omega^{i}_{\mathcal{I}_{\omega}}(S)$
for any $1\leqslant i\leqslant n$.
\end{enumerate}
\end{theorem}

\begin{proof}
By Theorem ~\ref{thm-3.5} and Proposition ~\ref{prop-4.17}, we have $(1)\Leftrightarrow(3)\Leftrightarrow (4) \Leftrightarrow (5)$.
By Theorem ~\ref{thm-3.7} and Proposition ~\ref{prop-4.16}, $(2)\Leftrightarrow(6)\Leftrightarrow (7)$.

In the following, we will prove $(1)\Leftrightarrow (2)$ by induction on $n$.
The case for $n=1$ is trivial and the case for $n=2$ follows from Lemma ~\ref{lem-4.20}. Now suppose $n\geqslant 3$.

$(1)\Rightarrow (2)$ Let $N\in \Mod S$. By the induction hypothesis, we have that
\linebreak
$\Ecograde_{\omega}\Tor_{i}^{S}(\omega,N)\geqslant i-1$ for any $1\leqslant i\leqslant n-1$ and
$\Ecograde_{\omega}\Tor_{n}^{S}(\omega,N)\geqslant n-2$.
By Lemma \ref{lem-3.1}(2), there exists an exact sequence
$$0\to \im(f_{n}\otimes 1_{\omega})\stackrel{\sigma}{\longrightarrow} \acTr_\omega \Omega_{\mathcal{F}}^{n-1}(N)
\stackrel{\tau}{\longrightarrow}\Tor^S_{n}(\omega,N)\to 0$$
in $\Mod R$ such that $\sigma_*$ is an isomorphism. By Theorem \ref{thm-3.7},
we have that $\Omega_{\mathcal{F}}^{n-1}(N)\in \acT_{\omega}^{n-1}(S)$ and
$\acTr_\omega \Omega_{\mathcal{F}}^{n-1}(N)\in{_R\omega}^{\bot_{n-1}}$.
So $$\Ext^{n-2}_R(\omega,\Tor^S_{n}(\omega,N))\cong \Ext^{n-1}_R(\omega,\im(f_{n}\otimes 1_{\omega}))$$
$$\cong \Ext^{n}_R(\omega,\acTr_{\omega}\Omega^{n}_{\mathcal{F}}(N)).$$
Then $\Tcograde_{\omega}\Ext^{n-2}_R(\omega,\Tor^S_{n}(\omega,N))\geqslant n-1$ by (1). It follows from Lemma ~\ref{lem-4.11}(1)
that $\Ecograde_{\omega}\Tor_{n}^{S}(\omega,N)\geqslant n-1$.

$(2)\Rightarrow (1)$ Let $M\in \Mod R$. By the induction hypothesis, we have that
\linebreak
$\Tcograde_{\omega}\Ext^{i}_{R}(\omega,M)\geqslant i-1$ for any $1\leqslant i\leqslant n-1$ and
$\Tcograde_{\omega}\Ext^{n}_{R}(\omega,M)\geqslant n-2$.
By Lemma \ref{lem-3.1}(1), there exists an exact sequence
$$0\to \Ext^{n+1}_R(\omega,M)\stackrel{\lambda}{\longrightarrow} \cTr_\omega \coOmega^n(M)
\stackrel{\pi}{\longrightarrow} I^{n+1}(M)_*/\coOmega^{n+1}(M)_*\to 0$$
in $\Mod S$ such that $1_{\omega}\otimes \pi$ is an isomorphism. By Theorem \ref{thm-3.5},
we have that $\coOmega^{n-1}(M)\in \cT_{\omega}^{n-1}(R)$ and $\cTr_\omega \coOmega^{n-1}(M)\in{\omega_S}^{\top_{n-1}}$.
So $$\Tor_{n-2}^S(\omega,\Ext^{n}_R(\omega,M))\cong \Tor_{n-1}^S(\omega,I^{n}(M)_*/\coOmega^{n}(M)_*)$$
$$\cong \Tor_{n}^S(\omega,\cTr_{\omega}\coOmega^{n}(M)).$$
Then $\Ecograde_{\omega}\Tor_{n-2}^S(\omega,\Ext^{n}_R(\omega,M))\geqslant n-1$ by (2). It follows from Lemma ~\ref{lem-4.11}(2)
that $\Tcograde_{\omega}\Ext^{n}_{R}(\omega,M)\geqslant n-1$.
\end{proof}

\vspace{0.2cm}

\noindent{\bf 4.4. Examples}

\vspace{0.2cm}

In this subsection, we give some examples for $\omega$ satisfying the (quasi) $n$-cograde condition.

Let $R$ be an artin algebra. Recall that $R$ is called {\it Auslander $n$-Gorenstein} if $\pd_{R^{op}}I^i(R_R)\leqslant i$
for any $0\leqslant i\leqslant n-1$; equivalently $\pd_{R}I^i(_RR)\leqslant i$
for any $0\leqslant i\leqslant n-1$ (\cite{FGR,IS}); and $R$ is called {\it left (resp. right) quasi $n$-Gorenstein}
if $\pd_{R}I^i(_{R}R)$ (resp. $\pd_{R^{op}}I^i(R_{R}) \leqslant i+1$ for any $0\leqslant i\leqslant n-1$
(\cite{H3}).

Let $D$ be the ordinary duality between $\mod R$ and $\mod R^{op}$. Then $D(R)$ is a semidualizing $(R,R)$-bimodule.
Because
$$\pd_{R}I^i(_RR)=\id_{R^{op}}P_i(D(_RR))=\pd_{R}\Hom_{R^{op}}(P_i(D(_RR)),D(R))\ \text{and}$$
$$\pd_{R^{op}}I^i(R_R)=\id_{R}P_i(D(R_R))=\pd_{R^{op}}\Hom_{R}(P_i(D(R_R)),D(R)),$$
we have

\begin{example} \label{exa-4.20}
{\rm \begin{enumerate}
\item[]
\item $R$ is Auslander $n$-Gorenstein if and only if $D(R)$ satisfies the $n$-cograde condition.
\item $R$ is left (resp. right) quasi $n$-Gorenstein if and only if $D(R)$ satisfies
the left (resp. right) quasi $n$-cograde condition.
\end{enumerate}}
\end{example}

So, if putting ${_R\omega_S}={_RD(R)_R}$ in Theorem ~\ref{thm-4.8} (resp. Theorem ~\ref{thm-4.14}),
then all the conditions there are equivalent to that $R$ is Auslander $n$-Gorenstein
(resp. right quasi $n$-Gorenstein). Note that the notion of quasi $n$-Gorenstein algebras is not
left-right symmetric (\cite[p.11]{AR1}). So,
contrary to the $n$-cograde condition, the quasi $n$-cograde condition is not left-right symmetric.

\begin{example} \label{exa-4.21}
Let $Q$ be the quiver
$$\xymatrix{
& 3  \ar[dl]_{\beta} \\
1 & & 5 \ar[dl]^{\gamma} \ar[ul]_{\alpha}. \\
& 4 \ar[dl]^{\varepsilon} \ar[ul]^{\delta}\\
2 }$$ and $R=KQ/<\beta\alpha-\delta\gamma, \varepsilon\gamma>$ with $K$ a field. Take
$$\omega:=\begin{array}{c} 0\\1 ~ 0 \\ 0\\ 0~\end{array}\oplus \begin{array}{c} 0\\1 ~ 0 \\ 1\\ 1~\end{array}\oplus
\begin{array}{c} 0\\1 ~ 0 \\ 1\\ 0~\end{array}\oplus \begin{array}{c} 1\\1 ~ 1 \\ 1\\ 0~\end{array}\oplus
\begin{array}{c} 0\\0 ~ 1 \\ 1\\ 0~\end{array}.$$
By \cite[Example VI.2.8(a)]{ASS}, we have that $\omega_R$ is a non-injective tilting module with $\pd_{R}\omega=1$.
Thus it is a semidualizing $(R,\End_R(\omega))$-bimodule. It is straightforward to verify that the projective cover
$P_0(\omega)$ of $\omega$ is $P(1)\oplus P(4)^2\oplus P(5)^2$. So $\mathcal{P}_\omega(R)$-$\id_{R}P_0(\omega)=0$,
and hence $\omega$ satisfies the left and right 1-cograde conditions by Theorem ~\ref{thm-4.8}.
Since $\pd_{R}\omega=1$, we have $\Ext_{R}^{\geqslant 2}(\omega,M)=0$ for any $M\in \Mod R$. By Theorem ~\ref{thm-4.8} again,
we have that $\omega$ satisfies the left and right $n$-cograde conditions for any $n\geqslant 1$.
\end{example}

\section {\bf Two cotorsion pairs}

In this section, we will construct two complete cotorsion pairs under any of the equivalent conditions in Theorem ~\ref{thm-4.19}.

For any $n\geqslant 0$, set $\mathcal{P}_\omega\text{-}\id^{\leqslant n}(R):=\{M\in \Mod R\mid\mathcal{P}_\omega(R){\text -}\id_R M\leqslant n \}$.

\begin{lemma} \label{lem-5.1}
Let $M\in {_R\omega^{\perp_{n-1}}}$ with $n\geqslant 1$. If $\Tcograde_{\omega}\Ext_R^n(\omega,M)\geqslant n-1$,
then there exists an exact sequence
$$0\to M\to X\to Y\to 0$$ in $\Mod R$
with $X\in {_R\omega^{\perp_{n}}}$ and $Y\in\mathcal{P}_\omega\text{-}\id^{\leqslant n-1}(R)$.
\end{lemma}

\begin{proof}
Let $M\in {_R\omega^{\perp_{n-1}}}$. From the exact sequence
$$0\rightarrow M \rightarrow I^0(M) \to \cdots\to I^{n-1}(M)\to \coOmega^{n}(M)\to 0$$
in $\Mod R$, we get the following commutative diagram with exact rows
{\footnotesize$$\xymatrix{ & P_{n-1} \ar[d] \ar[r]& \cdots \ar[r]  &   P_{0} \ar[d] \ar[r]& \Ext_R^n(\omega, M) \ar@{=}[d] \ar[r]& 0\\
I^0(M)_*   \ar[r]& I^1(M)_*   \ar[r]& \cdots \ar[r] &  \coOmega^{n}(M)_* \ar[r] & \Ext_R^n(\omega, M) \ar[r]& 0,\\
&  & &{\rm Diagram}\ (5.1) &  }$$}
where the upper sequence is a projective resolution of $\Ext_R^n(\omega, M)$ in $\Mod S$.
Taking the mapping cone of the diagram (5.1), we get an exact sequence
$$I^0(M)_*\oplus P_{n-1}\to \cdots \to I^{n-1}(M)_*\oplus P_{0}\to \coOmega^{n}(M)_* \to 0. \eqno{(5.2)}$$
Since $\Tcograde_{\omega}\Ext_R^n(\omega,M)\geqslant n-1$, we get an exact sequence
$$\omega\otimes_S P_{n-1}\to \cdots \to \omega\otimes_S P_{1} \to \omega\otimes_S P_{0}\to 0$$ in $\Mod R$.
Then we get the following commutative diagram with exact columns and rows
$$\xymatrix{ & 0 \ar[d]& 0 \ar@{-->}[d] & 0 \ar[d] & \\
0\ar@{-->}[r] & M\ar@{-->}[r] \ar[d]& X \ar@{-->}[r]\ar@{-->}[d] & Y \ar@{-->}[r]\ar[d] & 0\\
0\ar[r] & I^0(M)\ar[r] \ar[d]& I^0(M)\oplus (\omega \otimes_S P_{n-1}) \ar[r]\ar@{-->}[d] & \omega \otimes_S P_{n-1} \ar[r]\ar[d] & 0\\
& \vdots \ar[d]& \vdots\ar@{-->}[d] & \vdots\ar[d] & \\
0\ar[r] & I^{n-1}(M)\ar[r] \ar[d]& I^{n-1}(M)\oplus (\omega \otimes_S P_0) \ar[r]\ar@{-->}[d] & \omega \otimes_S P_{0} \ar[r]\ar[d] & 0\\
& \coOmega^{n}(M)\ar@{==}[r] \ar[d]& \coOmega^{n}(M) \ar@{-->}[d] & 0 & \\
& 0 & 0, &  & \\
&  & {\rm Diagram}\ (5.3) & &  }$$
where
$$X=\Ker(I^0(M)\oplus (\omega \otimes_S P_{n-1})\to I^1(M)\oplus (\omega \otimes_S P_{n-2}))\ {\rm and}$$
$$Y=\Ker(\omega\otimes_S P_{n-1} \to \omega\otimes_S P_{n-2}).$$
Then $Y\in\mathcal{P}_\omega\text{-}\id^{\leqslant n-1}(R)$.
From the exactness of (5.2) and the middle column in the diagram (5.3), we know that $X\in {_R\omega^{\perp_{n}}}$.
So the top row in the diagram (5.3) is the desired exact sequence.
\end{proof}

For any $n\geqslant 0$, set $\mathcal{I}_\omega\text{-}\pd^{\leqslant n}(S):=\{N\in \Mod S\mid\mathcal{I}_\omega(S){\text -}\pd_S N\leqslant n \}$.

\begin{lemma} \label{lem-5.2}
Let $N\in {{\omega_S}^{\top_{n-1}}}$ with $n\geqslant 1$. If $\Ecograde_{\omega}\Tor^S_n(\omega,N)\geqslant n-1$,
then there exists an exact sequence
$$0\to Y'\to X'\to N\to 0$$
in $\Mod S$ with $X'\in {{\omega_S}^{\top_{n}}}$ and $Y'\in\mathcal{I}_\omega\text{-}\pd^{\leqslant n-1}(S)$.
\end{lemma}

\begin{proof}
Let $N\in {{\omega_S}^{\top_{n-1}}}$. From the exact sequence
$$0\rightarrow \Omega^{n}_{\mathcal{F}}(N) \rightarrow F_{n-1}(N) \to  \cdots \to F_{0}(N)\to N\to 0$$
in $\Mod S$, we get the following commutative diagram with exact rows
{\tiny$$\xymatrix{ 0 \ar[r]& \Tor_n^S(\omega,N) \ar@{=}[d] \ar[r]& \omega\otimes_S\Omega^{n}_{\mathcal{F}}(N)
\ar[r] \ar[d]&\cdots \ar[r]  & \omega\otimes_S F_1(N) \ar[d] \ar[r]& \omega\otimes_SF_0(N)\\
 0 \ar[r]& \Tor_n^S(\omega,N)  \ar[r]& I^0 \ar[r]& \cdots \ar[r] & I^{n-1},\\
&  & &{\rm Diagram}\ (5.4) &  }$$}
where the lower sequence is an injective resolution of $\Tor_n^S(\omega,N)$ in $\Mod R$.
Taking the mapping cone of diagram (5.4), we get an exact sequence
$$\omega\otimes_S\Omega^{n}_{\mathcal{F}}(N)\to I^0\oplus(\omega\otimes_SF_{n-1}(N))\to \cdots \to
I^{n-1}\oplus(\omega\otimes_SF_{0}(N)). \eqno{(5.5)}$$
Since $\Ecograde_{\omega}\Tor^S_n(\omega,N)\geqslant n-1$, we get an exact sequence
$$0\to {I^0}_*\to {I^1}_*\to \cdots \to {I^{n-1}}_*$$ in $\Mod S$.
Then we get the following commutative diagram with exact columns and rows
$$\xymatrix{& & 0\ar@{-->}[d] & 0\ar[d] &  \\
& 0\ar[d]& \Omega^n_{\mathcal{F}}(N)\ar@{-->}[d]\ar@{==}[r] & \Omega^n_{\mathcal{F}}(N)\ar[d] &  \\
0\ar[r] & {I^0}_*\ar[r] \ar[d]& {I^0}_*\oplus F_{n-1}(N) \ar[r]\ar@{-->}[d] & F_{n-1}(N) \ar[r]\ar[d] & 0\\
& \vdots \ar[d]& \vdots\ar@{-->}[d] & \vdots\ar[d] & \\
0\ar[r] & {I^{n-1}}_*\ar[r] \ar[d]& {I^{n-1}}_*\oplus F_{0}(N) \ar[r]\ar@{-->}[d] & F_{0}(N) \ar[r]\ar[d] & 0\\
0\ar@{-->}[r]& Y'\ar[d]\ar@{-->}[r]& X'\ar@{-->}[d]\ar@{-->}[r] & N\ar[d]\ar@{-->}[r] & 0 \\
& 0& 0 & 0, & \\
&  & {\rm Diagram}\ (5.6) & &  }$$
where $$X'=\Coker({I^{n-2}}_*\oplus F_{1}(N)\to {I^{n-1}}_*\oplus F_{0}(N))\ \text{and}$$
$$Y'=\Coker({I^{n-2}}_*\to {I^{n-1}}_*).$$
Then $Y'\in\mathcal{I}_\omega\text{-}\pd^{\leqslant n-1}(S)$.
From the exactness of (5.5) and the middle column in the diagram (5.6), we know that $X'\in {{\omega_S}^{\top_{n}}}$.
So the bottom row in the diagram (5.6) is the desired exact sequence.
\end{proof}

\begin{lemma} \label{lem-5.3}
For any $n\geqslant 0$, we have
\begin{enumerate}
\item $\mathcal{P}_\omega\text{-}\id^{\leqslant n}(R)$ is closed under direct summands and closed under extensions.
\item $\mathcal{I}_\omega\text{-}\pd^{\leqslant n}(S)$ is closed under direct summands and closed under extensions.
\end{enumerate}
\end{lemma}

\begin{proof}
(1) By \cite[Lemma 4.6]{TH2}, $\mathcal{P}_\omega\text{-}\id^{\leqslant n}(R)$ is closed under direct summands.

Let $$0\to A \to B \to C \to 0$$ be an exact sequence in $\Mod R$ with $A,C\in\mathcal{P}_\omega\text{-}\id^{\leqslant n}(R)$.
It is easy to see that it is $\Hom_R(-,\mathcal{P}_{\omega}(R))$-exact.
Then $B\in\mathcal{P}_\omega\text{-}\id^{\leqslant n}(R)$ by the generalized horseshoe lemma (c.f. \cite[Lemma 3.1(2)]{H4}).

(2) By \cite[Lemma 4.7]{TH3}, $\mathcal{I}_\omega\text{-}\pd^{\leqslant n}(S)$ is closed under direct summands.

Let $$0\to A \to B \to C \to 0$$ be an exact sequence in $\Mod S$ with $A,C\in\mathcal{I}_\omega\text{-}\pd^{\leqslant n}(S)$.
It is easy to see that it is $(\omega\otimes_S-)$-exact; equivalently it is $\Hom_R(-,\mathcal{I}_{\omega}(S))$-exact
by \cite[p.298, Observation]{TH3}. Then $B\in\mathcal{I}_\omega\text{-}\pd^{\leqslant n}(S)$ by the generalized horseshoe lemma
(c.f. \cite[Lemma 3.1(1)]{H4}).
\end{proof}

\begin{proposition} \label{prop-5.4}
Let $n,k\geqslant 1$ and $\Tcograde_{\omega}\Ext^{i+k}_R(\omega,M)\geqslant i$ for any $M\in \Mod R$
and $1\leqslant i\leqslant n-1$. Then for any $M\in \Mod R$ and $0\leqslant i \leqslant n-1$, there exists an exact sequence
$$0\to \coOmega^{k-1}(M) \to X\to Y\to 0$$ in $\Mod R$ with $X\in {_R\omega^{\perp_{i+1}}}$ and
$Y\in\mathcal{P}_\omega\text{-}\id^{\leqslant i}(R)$.
\end{proposition}

\begin{proof}
We proceed by induction on $n$. The case for $n=1$ follows from Lemma ~\ref{lem-5.1}.
Now suppose $n\geqslant 2$. By the induction hypothesis, for any $0\leqslant i \leqslant n-2$ there exists an exact sequence
$$0\to \coOmega^{k-1}(M) \to X_i\to Y_i\to 0$$ in $\Mod R$
with $X_i\in {_R\omega^{\perp_{i+1}}}$ and $Y_i\in\mathcal{P}_\omega\text{-}\id^{\leqslant i}(R)$. Then
$$\Ext^{n}_R(\omega,X_{n-2})\cong\Ext^{n}_R(\omega,\coOmega^{k-1}(M))\cong \Ext^{n+k-1}_R(\omega,M).$$ So
$\Tcograde_{\omega}\Ext^{n}_R(\omega,X_{n-2})= \Tcograde_{\omega}\Ext^{n+k-1}_R(\omega,M) \geqslant n-1$ by assumption.
Applying Lemma ~\ref{lem-5.1}, we get an exact sequence
$$0\to X_{n-2}\to X_{n-1}\to Y_{n-1}\to 0$$
in $\Mod R$ with $X_{n-1}\in {_R\omega^{\perp_{n}}}$ and $Y_{n-1}\in\mathcal{P}_\omega\text{-}\id^{\leqslant n-1}(R)$.
Consider the following push-out diagram
$$\xymatrix{ & & 0 \ar[d] & 0 \ar[d]\\
0 \ar[r] & \coOmega^{k-1}(M) \ar[r] \ar@{=}[d] & X_{n-2}  \ar[r] \ar[d] & Y_{n-2}  \ar[r] \ar[d] & 0 \\
0 \ar[r] & \coOmega^{k-1}(M) \ar[r] & X_{n-1}  \ar[r] \ar[d] & Y  \ar[r] \ar[d] & 0 \\
 & & Y_{n-1} \ar@{=}[r] \ar[d] & Y_{n-1}  \ar[d] \\
 & & 0 & 0. }$$
By Lemma ~\ref{lem-5.3}(1), we have $Y\in \mathcal{P}_\omega\text{-}\id^{\leqslant n-1}(R)$.
So the middle row in this diagram is the desired sequence.
\end{proof}

\begin{proposition} \label{prop-5.5}
Let $n,k\geqslant 1$ and $\Ecograde_{\omega}\Tor_{i+k}^S(\omega,N)\geqslant i$ for any $N\in \Mod S$ and $1\leqslant i\leqslant n-1$.
Then for any $N\in \Mod S$ and $0\leqslant i \leqslant n-1$, there exists an exact sequence
$$0\to Y'\to X'\to \Omega^{k-1}_{\mathcal{F}}(N) \to 0$$
in $\Mod S$ with $X'\in {{\omega_S}^{\top_{i+1}}}$ and $Y'\in\mathcal{I}_\omega\text{-}\pd^{\leqslant i}(S)$.
\end{proposition}

\begin{proof}
We proceed by induction on $n$. The case for $n=1$ follows from Lemma ~\ref{lem-5.2}.
Now suppose $n\geqslant 2$. By the induction hypothesis, for any $0\leqslant i \leqslant n-2$ there exists an exact sequence
$$0\to Y_i'\to X_i'\to \Omega^{k-1}_{\mathcal{F}}(N) \to 0$$ in $\Mod S$
with $X_i'\in {{\omega_S}^{\top_{i+1}}}$ and $Y_i'\in\mathcal{I}_\omega\text{-}\pd^{\leqslant i}(S)$. Then
$$\Tor_{n}^S(\omega,X'_{n-2})\cong \Tor_{n}^S(\omega,\Omega^{k-1}_{\mathcal{F}}(N))\cong \Tor_{n+k-1}^S(\omega,N).$$
So $\Ecograde_{\omega}\Tor_{n}^S(\omega,X'_{n-2})=\Ecograde_{\omega}\Tor_{n+k-1}^S(\omega,N) \geqslant n-1$ by assumption.
Applying Lemma ~\ref{lem-5.2}, we get an exact sequence
$$0\to Y_{n-1}'\to X_{n-1}'\to X_{n-2}'\to 0$$ in $\Mod S$ with $X_{n-1}'\in {{\omega_S}^{\top_{n}}}$ and
$Y'_{n-1}\in\mathcal{I}_\omega\text{-}\pd^{n-1}(S)$. Consider the following pull-back diagram
$$\xymatrix{ & & 0 \ar[d] & 0 \ar[d]\\
0 \ar[r] & Y_{n-1}' \ar[r] \ar@{=}[d] & Y'  \ar[r] \ar[d] & Y_{n-2}'  \ar[r] \ar[d] & 0 \\
0 \ar[r] & Y_{n-1}' \ar[r] & X_{n-1}'  \ar[r] \ar[d] & X_{n-2}'  \ar[r] \ar[d] & 0 \\
 & & \Omega^{k-1}_{\mathcal{F}}(N) \ar@{=}[r] \ar[d] & \Omega^{k-1}_{\mathcal{F}}(N)  \ar[d] \\
 & & 0 & 0. }$$
By Lemma ~\ref{lem-5.3}(2), we have $Y'\in\mathcal{I}_\omega\text{-}\pd^{\leqslant n-1}(S)$.
So the middle column in this diagram is the desired sequence.
\end{proof}

Based on the equivalence of (1) and (2) in Theorem ~\ref{thm-4.19}, we have the following

\begin{theorem} \label{thm-5.6}
For any $n\geqslant 1$, we have
\begin{enumerate}
\item[(1)]  If one of the equivalent conditions in Theorem ~\ref{thm-4.19} is satisfied,
then the following statements are equivalent.
\begin{enumerate}
\item[(1.1)] $\pd_{S^{op}}\omega\leqslant n-1$.
\item[(1.2)] $\mathcal{P}_{\omega}(R)$-$\id_RR\leqslant n-1$.
\item[(1.3)] $\mathcal{P}_{\omega}(R)$-$\id_RP\leqslant n-1$ for any projective $P$ in $\Mod R$.
\item[(1.4)] $(\mathcal{P}_\omega\text{-}\id^{\leqslant n-1}(R), {_R\omega^{\perp_{n}}})$ forms a complete cotorsion pair.
\end{enumerate}
\end{enumerate}
\begin{enumerate}
\item[(2)] If one of the equivalent conditions in Theorem ~\ref{thm-4.19} is satisfied,
then the following statements are equivalent.
\begin{enumerate}
\item[(2.1)] $\mathcal{I}_{\omega}(S)$-$\pd_SQ\leqslant n-1$ for some injective cogenerator $Q$ in $\Mod S$.
\item[(2.2)] $\mathcal{I}_{\omega}(S)$-$\pd_SI\leqslant n-1$ for any injective module $I$ in $\Mod S$.
\item[(2.3)] $({{\omega_S}^{\top_{n}}}, \mathcal{I}_\omega\text{-}\pd^{\leqslant n-1}(S))$ forms a complete cotorsion pair.
\end{enumerate}
If $R$ and $S$ are artin algebras, then the statements (2.1)--(2.3) are equivalent to the following
\begin{enumerate}
\item[(2.4)] $\pd_R\omega\leqslant n-1$.
\end{enumerate}
\end{enumerate}
\end{theorem}

\begin{proof}
By Lemma ~\ref{lem-4.5}(1), we have $(1.1)\Leftrightarrow (1.2)$.

If $\mathcal{P}_{\omega}(R)$-$\id_RR\leqslant n-1$, then $\mathcal{P}_{\omega}(R)$-$\id_RF\leqslant n-1$ for any free module $F$ in $\Mod R$
by \cite[Proposition 5.1(b)]{HW}. It follows from Lemma ~\ref{lem-5.3}(1) that $\mathcal{P}_{\omega}(R)$-$\id_RP\leqslant n-1$
for any projective $P$ in $\Mod R$. This proves $(1.2)\Leftrightarrow (1.3)$.

$(1.3)\Rightarrow (1.4)$ It is easy to verify that $\Ext_R^1(A,B)=0$ for any $A\in \mathcal{P}_\omega\text{-}\id^{\leqslant n-1}(R)$
and $B\in {_R\omega^{\perp_{n}}}$.

Let $M\in \Mod R$. By Lemma \ref{lem-5.1} when $n=1$ or taking $k=1$ in Proposition ~\ref{prop-5.4} when $n\geqslant 2$,
we get an exact sequence
$$0\to M \to B\to A\to 0\eqno{(5.7)}$$
in $\Mod R$ with $B\in {_R\omega^{\perp_{n}}}$ and $A\in \mathcal{P}_\omega\text{-}\id^{\leqslant n-1}(R)$.
It implies that $M$ has a special ${_R\omega^{\perp_{n}}}$-preenvelope and ${_R\omega^{\perp_{n}}}$ is
special preenveloping in $\Mod R$. If $M\in({\mathcal{P}_\omega\text{-}\id^{\leqslant n-1}(R))^{\bot_1}}$, then the exact sequence (5.7) splits.
It follows that $M$ is a direct summand of $B$ and $M\in {_R\omega^{\perp_{n}}}$.

Let
$$0\to M_1\to P\to M\to 0$$ be an exact sequence in $\Mod R$ with $P$ projective. By (1.3),
we have $P\in\mathcal{P}_\omega\text{-}\id^{\leqslant n-1}(R)$. By Lemma \ref{lem-5.1} when $n=1$ or
by Proposition ~\ref{prop-5.4} when $n\geqslant 2$, we have an exact sequence
$$0\to M_1\to B'\to A'\to 0$$ in $\Mod R$ with $B'\in{_R\omega^{\perp_{n}}}$ and $A'\in \mathcal{P}_\omega\text{-}\id^{\leqslant n-1}(R)$.
Consider the following push-out diagram
$$\xymatrix{  & 0 \ar[d] & 0 \ar[d]\\
0 \ar[r] & M_1 \ar[r] \ar[d] & P  \ar[r] \ar[d] & M  \ar[r] \ar@{=}[d] & 0 \\
0 \ar[r] & B' \ar[d]  \ar[r] & A''  \ar[r] \ar[d] & M  \ar[r]  & 0 \\
& A' \ar@{=}[r] \ar[d] &A' \ar[d] \\
& 0 & 0.  }$$
Since $\mathcal{P}_\omega\text{-}\id^{\leqslant n-1}(R)$ is closed under extensions by Lemma ~\ref{lem-5.3}(1), it follows from the middle column
in the above diagram that $A''\in\mathcal{P}_\omega\text{-}\id^{\leqslant n-1}(R)$. If $M\in{^{\bot_1}{(_R\omega^{\perp_{n}}})}$,
then the middle row in the above diagram splits and $M$ is a direct summand of $A''$. By Lemma ~\ref{lem-5.3}(1),
we have $M\in\mathcal{P}_\omega\text{-}\id^{\leqslant n-1}(R)$. It follows from Lemma ~\ref{lem-2.7} that
$(\mathcal{P}_\omega\text{-}\id^{\leqslant n-1}(R),{_R\omega^{\perp_{n}}})$ forms a complete cotorsion pair.

$(1.4)\Rightarrow (1.2)$ By (1.4), we immediately have that ${_RR}\in\mathcal{P}_\omega\text{-}\id^{\leqslant n-1}(R)$
and $\mathcal{P}_{\omega}(R)$-$\id_RR\leqslant n-1$.

If $\mathcal{I}_{\omega}(S)$-$\pd_SQ\leqslant n-1$ for some injective cogenerator $Q$ in $\Mod S$,
then any direct product of $Q$ is in $\mathcal{I}_\omega\text{-}\pd^{\leqslant n-1}(S)$ by \cite[Proposition 5.1(c)]{HW}.
It follows from Lemma ~\ref{lem-5.3}(2) that $\mathcal{I}_{\omega}(S)$-$\pd_SI\leqslant n-1$ for any injective module $I$ in $\Mod S$.
This proves $(2.1)\Leftrightarrow (2.2)$.

$(2.2)\Rightarrow (2.3)$ It is easy to verify that $\Ext_S^1(C,D)=0$ for any $C\in {{\omega_S}^{\top_{n}}}$ and
$D\in \mathcal{I}_\omega\text{-}\pd^{\leqslant n-1}(S)$.

Let $N\in \Mod S$. By Lemma \ref{lem-5.2} when $n=1$ or taking $k=1$ in Proposition ~\ref{prop-5.5} when $n\geqslant 2$,
we get an exact sequence
$$0\to D \to C\to N\to 0\eqno{(5.8)}$$
in $\Mod S$ with $C\in{{\omega_S}^{\top_{n}}}$ and $D\in \mathcal{I}_\omega\text{-}\pd^{\leqslant n-1}(S)$. It implies that
$N$ has a special ${{\omega_S}^{\top_{n}}}$-precover and ${{\omega_S}^{\top_{n}}}$ is precovering in $\Mod S$. If
$N\in {^{\bot_1}(\mathcal{I}_\omega\text{-}\pd^{\leqslant n-1}(S)})$, then the exact sequence (5.8) splits.
It follows that $N$ is a direct summand of $C$ and $N\in{{\omega_S}^{\top_{n}}}$.

Let
$$0\to N\to I\to N_1\to 0$$ be an exact sequence in $\Mod S$ with $I$ injective. By (2.2), we have
$I\in \mathcal{I}_\omega\text{-}\pd^{\leqslant n-1}(S)$. By Lemma \ref{lem-5.2} when $n=1$ or
by Proposition ~\ref{prop-5.5} when $n\geqslant 2$, we have an exact sequence
$$0\to D'\to C'\to N_1\to 0$$ in $\Mod S$
with $C'\in{{\omega_S}^{\top_{n}}}$ and $D'\in \mathcal{I}_\omega\text{-}\pd^{\leqslant n-1}(S)$.
Consider the following pull-back diagram
$$\xymatrix{ & & 0 \ar[d] & 0 \ar[d]\\
&& N \ar@{=}[r] \ar[d]& N \ar[d]\\
0 \ar[r] & D' \ar[r] \ar@{=}[d] & D''  \ar[r] \ar[d] & I \ar[r] \ar[d] & 0 \\
0 \ar[r] & D'   \ar[r] & C'  \ar[r] \ar[d] & N_1  \ar[r] \ar[d] & 0 \\
 & & 0 & 0.  }$$
Since $\mathcal{I}_\omega\text{-}\pd^{\leqslant n-1}(S)$ is closed under extensions by Lemma ~\ref{lem-5.3}(2), it follows from
the middle row in the above diagram that $D''\in\mathcal{I}_\omega\text{-}\pd^{\leqslant n-1}(S)$.
If $N\in ({{\omega_S}^{\top_{n}}})^{\bot_1}$, then the middle column in the above diagram splits and $N$ is a direct summand of $D''$.
By Lemma ~\ref{lem-5.3}(2), we have $N\in\mathcal{I}_\omega\text{-}\pd^{\leqslant n-1}(S)$. It follows from Lemma ~\ref{lem-2.7} that
$({{\omega_S}^{\top_{n}}},\mathcal{I}_\omega\text{-}\pd^{\leqslant n-1}(S))$ forms a complete cotorsion pair.

$(2.3)\Rightarrow (2.2)$ For any injective module $I$ in $\Mod S$, by (2.3) we have that $I\in\mathcal{I}_\omega\text{-}\pd^{\leqslant n-1}(S)$
and $\mathcal{I}_{\omega}(S)$-$\pd_SI\leqslant n-1$.

If $R$ and $S$ are artin algebras, then $\pd_R\omega=\mathcal{I}_{\omega}(S)$-$\pd_SD(S_S)$ by \cite[Lemma 4.9]{TH3}.
Because $D(S_S)$ is an injective cogenerator in $\Mod S$, $(2.1)\Leftrightarrow (2.4)$ follows.
\end{proof}

\begin{observation} \label{ob-5.7}
Let $R$ be an artin algebra and ${_R\omega_S}={_RD(R)_R}$. Then we have
\begin{enumerate}
\item $\pd_R\omega=\id_{R^{op}}R$ and $\pd_{R^{op}}\omega=\id_{R}R$.
\item $\mathcal{P}_{\omega}(R)$ is exactly the subclass of $\Mod R$
consisting of injective modules. It implies that
\begin{enumerate}
\item[(2.1)] $\mathcal{P}_{\omega}(R)$-$\id_RM=\id_RM$ for any $M\in \Mod R$.
\item[(2.2)] $\mathcal{P}_\omega\text{-}\id^{\leqslant n}(R)=\mathcal{I}^{\leqslant n}(R):=\{M\in \Mod R\mid \id_RM\leqslant n\}$.
\end{enumerate}
\item $\mathcal{I}_{\omega}(R)$ is exactly the subclass of $\Mod R$ consisting of projective modules.
It implies that
\begin{enumerate}
\item[(3.1)] $\mathcal{I}_{\omega}(R)$-$\pd_RN=\pd_RN$ for any $N\in \Mod R$.
\item[(3.2)] $\mathcal{I}_\omega\text{-}\pd^{\leqslant n}(R)=\mathcal{P}^{\leqslant n}(R):=\{N\in \Mod R\mid \pd_RN\leqslant n\}$.
\end{enumerate}
\item By \cite[Proposition VI.5.3]{CE}, it is easy to see that ${{\omega_R}^{\top_{n+1}}}={^{\bot_{n+1}}{_RR}}$.
\item If $R$ is right quasi $(n-1)$-Gorenstein, then all conditions in Theorem ~\ref{thm-4.19} are satisfied;
see Theorem ~\ref{thm-4.14} and Example ~\ref{exa-4.20}(2).
\end{enumerate}
\end{observation}

As an application of Theorem ~\ref{thm-5.6}, we have the following

\begin{corollary} \label{cor-5.8}
Let $R$ be a right quasi $(n-1)$-Gorenstein artin algebra with $n\geqslant 1$. Then the following statements are equivalent.
\begin{enumerate}
\item[(1)] $\id_{R}R\leqslant n-1$.
\item[(2)] $\id_{R^{op}}R\leqslant n-1$.
\item[(3)] $(\mathcal{I}^{\leqslant n-1}(R), {_RD(R)^{\perp_{n}}})$ forms a complete cotorsion pair.
\item[(4)] $({^{\bot_{n}}{_RR}},\mathcal{P}^{\leqslant n-1}(R))$ forms a complete cotorsion pair.
\end{enumerate}
\end{corollary}

\begin{proof}
By Theorem ~\ref{thm-5.6} and Observation~ \ref{ob-5.7}, we have $(1)\Leftrightarrow (3)$ and $(2)\Leftrightarrow (4)$.

$(1)\Leftrightarrow (2)$ Let $\id_{R}R\leqslant n-1$. By \cite[Theorem 4.7]{AR2} and the symmetric version of \cite[Theorem]{H2},
we have $\id_{R^{op}}R\leqslant (n-1)+(n-2)=2n-3$. Conversely, let $\id_{R^{op}}R\leqslant n-1$. By \cite[Theorem 7.5]{TH4},
we have $\id_{R}R\leqslant n-1$. Now the assertion follows from \cite[Lemma A]{Z}.
\end{proof}

As a consequence of Corollary ~\ref{cor-5.8}, we have the following

\begin{corollary} \label{cor-5.9}
For any artin algebra $R$, the following conditions are equivalent.
\begin{enumerate}
\item[(1)] $\id_{R}R\leqslant 1$.
\item[(2)] $\id_{R^{op}}R\leqslant 1$.
\end{enumerate}
Furthermore, if $R$ is right quasi 1-Gorenstein, then they are equivalent to each of the following two statements.
\begin{enumerate}
\item[(3)] $(\mathcal{I}^{\leqslant 1}(R), {_RD(R)^{\perp_{2}}})$ forms a complete cotorsion pair.
\item[(4)] $({^{\bot_{2}}{_RR}},\mathcal{P}^{\leqslant 1}(R))$ forms a complete cotorsion pair.
\end{enumerate}
\end{corollary}

\begin{proof}
The first assertion follows from \cite[Corollary 2]{H2}.
If $R$ is right quasi 1-Gorenstein,  then we get the second assertion by putting $n=2$ in Corollary ~\ref{cor-5.8}.
\end{proof}

We use $\mathcal{I}(R)$ and $\mathcal{P}(R)$ to denote the subclasses of $\Mod R$ consisting of injective and projective modules respectively.
Putting $n=1$ in Corollary ~\ref{cor-5.8}, we have the following

\begin{corollary} \label{cor-5.10}
For any artin algebra $R$, the following statements are equivalent.
\begin{enumerate}
\item[(1)] $R$ is self-injective.
\item[(2)] $(\mathcal{I}(R), {_RD(R)^{\perp_{1}}})$ forms a complete cotorsion pair
$($in this case, ${_RD(R)^{\perp_{1}}}=\mathcal{I}(R)^{\perp_{1}})$.
\item[(3)] $({^{\bot_{1}}{_RR}},\mathcal{P}(R))$ forms a complete cotorsion pair
$($in this case, ${^{\bot_{1}}{_RR}}={^{\perp_{1}}}{\mathcal{P}(R)})$.
\end{enumerate}
\end{corollary}

\section {\bf Relative finitistic dimensions}

In this section, we introduce and study the finitistic $\mathcal{P}_\omega(R)$-injective dimension and the $\mathcal{I}_\omega(S)$-projective dimension
of rings.

The {\it finitistic $\mathcal{P}_\omega(R)$-injective dimension} ${\rm F}\mathcal{P}_{\omega}{\text -}\id R$ of $R$ is defined as
$${\rm F}\mathcal{P}_{\omega}{\text -}\id R:=\sup\{\mathcal{P}_\omega(R){\text -}\id_RM\mid
M\in \Mod R\ {\rm and}\ \mathcal{P}_\omega(R){\text -}\id_RM<\infty\};$$
and the {\it finitistic $\mathcal{I}_\omega(S)$-projective dimension} ${\rm F}\mathcal{I}_{\omega}{\text -}\pd S$ of $S$ is defined as
$${\rm F}\mathcal{I}_{\omega}{\text -}\pd S:=\sup\{\mathcal{I}_\omega(S){\text -}\pd_SN\mid
N\in \Mod S\ {\rm and}\ \mathcal{I}_\omega(S){\text -}\pd_SN<\infty\}.$$

\begin{lemma} \label{lem-6.1}
For any $n\geqslant 0$ and $k\geqslant 1$, we have
\begin{enumerate}
\item Let $\Tcograde_{\omega}\Ext^{i+k}_R(\omega,M)\geqslant i$ for any $M\in \Mod R$
and $1\leqslant i\leqslant n+1$. If ${\rm F}\mathcal{P}_{\omega}{\text -}\id R=n$, then $\pd_R\omega\leqslant n+k$.
\item Let $\Ecograde_{\omega}\Tor_{i+k}^S(\omega,N)\geqslant i$ for any $N\in \Mod S$
and $1\leqslant i\leqslant n+1$. If ${\rm F}\mathcal{I}_{\omega}{\text -}\pd S=n$, then $\pd_{S^{op}}\omega\leqslant n+k$.
\end{enumerate}
\end{lemma}

\begin{proof}
(1) Let $M\in \Mod R$. By Proposition ~\ref{prop-5.4}, there exists an exact sequence
$$0\to \coOmega^{k-1}(M) \to X\to Y\to 0$$ in $\Mod R$
with $X\in {_R{\omega}}^{\bot_{n+2}}$ and $\mathcal{P}_\omega(R)$-$\id_RY\leqslant n+1$.
If ${\rm F}\mathcal{P}_{\omega}{\text -}\id R=n$, then $\mathcal{P}_\omega(R)$-$\id_RY\leqslant n$. Thus we have that
$$\Ext^{n+k+1}_R(\omega,M)\cong \Ext^{n+2}_R(\omega,\coOmega^{k-1}(M))\cong \Ext^{n+1}_R(\omega,Y)=0$$ and $\pd_R\omega\leqslant n+k$.

(2) Let $N\in \Mod S$. By Proposition ~\ref{prop-5.5}, there exists an exact sequence
$$0\to Y'\to X'\to \Omega^{k-1}_{\mathcal{F}}(N) \to 0$$ in $\Mod S$ with $X'\in {\omega_S}^{\top_{n+2}}$
and $\mathcal{P}_\omega(R)$-$\id_SY'\leqslant n+1$.
If ${\rm F}\mathcal{I}_{\omega}{\text -}\pd S=n$, then $\mathcal{I}_\omega(R)$-$\pd_SY'\leqslant n$. Thus we have that
$$\Tor_{n+k+1}^S(\omega,N)\cong \Tor_{n+2}^S(\omega,\Omega^{k-1}_{\mathcal{F}}(N))\cong \Tor_{n+1}^S(\omega,Y')=0$$
and $\pd_{S^{op}}\omega=\fd_{S^{op}}\omega\leqslant n+k$.
\end{proof}

\begin{lemma} \label{lem-6.2}
For any $n\geqslant 0$, we have
\begin{enumerate}
\item Let ${\rm F}\mathcal{P}_{\omega}{\text -}\id R\leqslant n$ and $N\in \Mod S$. If $\Tcograde_{\omega}N\geqslant n+1$, then $N=0$.
\item Let ${\rm F}\mathcal{I}_{\omega}{\text -}\pd S\leqslant n$ and $H\in \Mod R$. If $\Ecograde_{\omega}H\geqslant n+1$, then $H=0$.
\end{enumerate}
\end{lemma}

\begin{proof}
(1) Consider a projective resolution
$$\cdots \to Q_{n+1}\to Q_n\to \cdots \to Q_0\to N\to 0$$ of $N$ in $\Mod S$. If $\Tcograde_{\omega}N\geqslant n+1$,
then we get an exact sequence
$$0\to M\to \omega\otimes_SQ_{n+1}\to \omega\otimes_SQ_{n}\to \cdots \to \omega\otimes_SQ_{1}\to \omega\otimes_SQ_{0}\to 0$$
in $\Mod R$, where $M=\Ker(\omega\otimes_SQ_{n+1}\to \omega\otimes_SQ_{n})$. By \cite[Corollary 3.5]{TH3}, $Q\cong(\omega\otimes_SQ)_*$ canonically
for any projective $Q$ in $\Mod S$, so $N\cong \Ext^{n+1}_R(\omega, M)$. Because ${\rm F}\mathcal{P}_{\omega}{\text -}\id R\leqslant n$
by assumption, we have that $\mathcal{P}_\omega(R)$-$\id_RM\leqslant n$ and $N\cong\Ext^{n+1}_R(\omega, M)=0$.

(2) Consider an injective resolution
$$0\to H\to I^0\to \cdots \to I^{n}\to I^{n+1}\to \cdots$$ of $H$ in $\Mod R$. If $\Ecograde_{\omega}H\geqslant n+1$,
then we get an exact sequence
$$0\to {I^0}_*\to \cdots \to {I^n}_*\to {I^{n+1}}_*\to N\to 0$$ in $\Mod S$, where $N=\Coker({I^n}_*\to {I^{n+1}}_*)$.
By \cite[Lemma 2.5(2)]{TH1}, $\omega\otimes_SI_*\cong I$ canonically for any injective $I$ in $\Mod R$, so $H\cong\Tor_{n+1}^S(\omega,N)$.
Because ${\rm F}\mathcal{I}_{\omega}{\text -}\pd S\leqslant n$ by assumption, we have that
$\mathcal{I}_\omega(R^{op})$-$\pd_SN\leqslant n$ and $H\cong\Tor_{n+1}^S(\omega,N)=0$.
\end{proof}

The following is the main result in this section.

\begin{theorem} \label{thm-6.3}
For any $k\geqslant 0$, we have
\begin{enumerate}
\item If $\Tcograde_{\omega}\Ext^{i+k}_R(\omega,M)\geqslant i$ for any $M\in \Mod R$ and $i\geqslant 1$,
then ${\rm F}\mathcal{P}_{\omega}{\text -}\id R\leqslant \pd_R\omega\leqslant{\rm F}\mathcal{P}_{\omega}{\text -}\id R+k$.
\item If $\Ecograde_{\omega}\Tor_{i+k}^S(\omega,N)\geqslant i$ for any $N\in \Mod S$ and $i\geqslant 1$,
then ${\rm F}\mathcal{I}_{\omega}{\text -}\pd S\leqslant \pd_{S^{op}}\omega\leqslant{\rm F}\mathcal{I}_{\omega}{\text -}\pd S+k$.
\end{enumerate}
\end{theorem}

\begin{proof}
(1) Let $\pd_R\omega=n(<\infty)$ and $M\in \Mod R$ with $\mathcal{P}_\omega(R)$-$\id_RM=m(<\infty)$.
Then there exists an exact sequence
$$0\to M\stackrel{f^0}{\longrightarrow} \omega^0\stackrel{f^1}{\longrightarrow} \omega^1\to \cdots \stackrel{f^m}{\longrightarrow}\omega^{m}\to 0$$
in $\Mod R$ with all $\omega^i$ in $\mathcal{P}_\omega(R)$. Since $\mathcal{P}_\omega(R)\subseteq \mathcal{B}_\omega(R)$ by \cite[Corollary 6.1]{HW},
we have $\mathcal{B}_\omega(R)$-$\id_RM\leqslant\mathcal{P}_\omega(R)$-$\id_RM<\infty$. If $m>n$,
then it follows from \cite[Theorem 4.2]{TH2} that
$\mathcal{B}_\omega(R)$-$\id_RM\leqslant n$ and $\im f^n \in \mathcal{B}_\omega(R)$. On the other hand, we have the following exact and split sequence
$$0\to (\im f^n)_* \to {\omega^n}_* \to \cdots \to {\omega^m}_*\to 0$$
in $\Mod S$ with all ${\omega^i}_*$ projective. So $(\im f^n)_*$ is projective, and hence $\im f^n \in \mathcal{P}_\omega(R)$ by \cite[Lemma 5.1(2)]{HW}.
It yields that $\mathcal{P}_\omega(R)$-$\id_RM\leqslant n$, a contradiction. This proves ${\rm F}\mathcal{P}_{\omega}{\text -}\id R\leqslant \pd_R\omega$.

In the following, we will prove $\pd_R\omega\leqslant{\rm F}\mathcal{P}_{\omega}{\text -}\id R+k$. The case for $k\geqslant 1$ follows from
Lemma ~\ref{lem-6.1}(1). Now suppose that $k=0$ and ${\rm F}\mathcal{P}_{\omega}{\text -}\id R=n(<\infty)$. Let $M\in \Mod R$. Then
$\Tcograde_{\omega}\Ext^{n+1}_R(\omega,M)\geqslant n+1$ by assumption. It follows from Lemma ~\ref{lem-6.2}(1) that $\Ext^{n+1}_R(\omega,M)=0$
and $\pd_R\omega\leqslant n$.

(2) Let $ \pd_{S^{op}}\omega=n(<\infty)$ and $N\in \Mod S$ with $\mathcal{I}_\omega(S)$-$\pd_SN=m(<\infty)$. Then there exists an exact sequence
$$0 \to U_m\stackrel{g_m}{\longrightarrow} \cdots \to U_1\stackrel{g_1}{\longrightarrow} U_0\stackrel{g_0}{\longrightarrow}N\to 0$$
in $\Mod S$ with all $U_i$ in $\mathcal{I}_\omega(S)$. Since $\mathcal{I}_\omega(S)\subseteq \mathcal{A}_\omega(S)$ by \cite[Corollary 6.1]{HW},
we have $\mathcal{A}_\omega(S)$-$\pd_SN<\infty$. If $m>n$, then it follows from the dual result of \cite[Theorem 4.2]{TH2} that
$\mathcal{A}_\omega(S)$-$\pd_SN\leqslant n$ and $\im g_n \in \mathcal{A}_\omega(S)$. On the other hand, we have the following exact and split sequence
$$0\to \omega\otimes_SU_m \to \cdots \to\omega\otimes_SU_n \to \omega\otimes_S\im g_n\to 0$$
in $\Mod R$ with all $\omega\otimes_SU_i$ injective. So $\omega\otimes_S\im g_n$ is injective, and hence $\im g_n \in \mathcal{I}_\omega(S)$
by \cite[Lemma 5.1(3)]{HW}. It yields that $\mathcal{I}_\omega(S)$-$\pd_SN\leqslant n$, a contradiction. This proves
${\rm F}\mathcal{I}_{\omega}{\text -}\pd S\leqslant \pd_{S^{op}}\omega$.

In the following, we will prove $\pd_{S^{op}}\omega\leqslant{\rm F}\mathcal{I}_{\omega}{\text -}\pd S+k$. The case for $k\geqslant 1$ follows from
Lemma ~\ref{lem-6.1}(2). Now suppose that $k=0$ and ${\rm F}\mathcal{I}_{\omega}{\text -}\pd S=n$. Let $N\in\Mod S$. Then
$\Ecograde_{\omega}\Tor_{n+1}^S(\omega,N)\geqslant n+1$ by assumption. It follows from Lemma ~\ref{lem-6.2}(2) that $\Tor_{n+1}^S(\omega,N)=0$
and $\pd_{S^{op}}\omega=\fd_{S^{op}}\omega\leqslant n$.
\end{proof}

Putting $k=0$ in Theorem~ \ref{thm-6.3}, we immediately get the following

\begin{corollary} \label{cor-6.4}
\begin{enumerate}
\item[]
\item If $\Tcograde_{\omega}\Ext^{i}_R(\omega,M)\geqslant i$ for any $M\in \Mod R$ and $i\geqslant 1$,
then ${\rm F}\mathcal{P}_{\omega}{\text -}\id R=\pd_R\omega$.
\item If $\Ecograde_{\omega}\Tor_{i}^S(\omega,N)\geqslant i$ for any $N\in \Mod S$ and $i\geqslant 1$,
then ${\rm F}\mathcal{I}_{\omega}{\text -}\pd S$ $=\pd_{S^{op}}\omega$.
\end{enumerate}
\end{corollary}

The following is an immediate consequence of Corollaries~ \ref{cor-4.2} and \ref{cor-6.4}.

\begin{corollary} \label{cor-6.5}
If $\omega$ satisfies the $n$-cograde condition for all $n$, then
$${\rm F}\mathcal{P}_{\omega}{\text -}\id R=\pd_R\omega\ {\rm and}\ {\rm F}\mathcal{I}_{\omega}{\text -}\pd S=\pd_{S^{op}}\omega.$$
\end{corollary}

Combining Theorem~ \ref{thm-4.19} with the case for $k=1$ in Theorem~ \ref{thm-6.3}, we get the following

\begin{corollary} \label{cor-6.6}
We have
$${\rm F}\mathcal{P}_{\omega}{\text -}\id R\leqslant \pd_R\omega\leqslant{\rm F}\mathcal{P}_{\omega}{\text -}\id R+1\ {\rm and}$$
$${\rm F}\mathcal{I}_{\omega}{\text -}\pd S\leqslant \pd_{S^{op}}\omega\leqslant{\rm F}\mathcal{I}_{\omega}{\text -}\pd S+1,$$
if either of the following conditions is satisfied.
\begin{enumerate}
\item $\Tcograde_{\omega}\Ext^{i+1}_R(\omega,M)\geqslant i$ for any $M\in \Mod R$ and $i\geqslant 1$.
\item $\Ecograde_{\omega}\Tor_{i+1}^S(\omega,N)\geqslant i$ for any $N\in \Mod S$ and $i\geqslant 1$.
\end{enumerate}
\end{corollary}

\begin{corollary} \label{cor-6.7}
If $\omega$ satisfies the right quasi $n$-cograde condition for all $n$, then
$${\rm F}\mathcal{P}_{\omega}{\text -}\id R=\pd_R\omega\ {\rm and}\
{\rm F}\mathcal{I}_{\omega}{\text -}\pd S\leqslant \pd_{S^{op}}\omega\leqslant{\rm F}\mathcal{I}_{\omega}{\text -}\pd S+1.$$
\end{corollary}

\begin{proof}
The former equality follows from Proposition~ \ref{prop-4.12} and Corollary~ \ref{cor-6.4}(1),
and the later inequalities follow from the definition of the right quasi $n$-cograde condition and
Corollary~ \ref{cor-6.6}.
\end{proof}

\begin{observation} \label{ob-6.8}
Let $R$ be an artin algebra and ${_R\omega_S}={_RD(R)_R}$. Then we have
\begin{enumerate}
\item By Observation 5.7, we have
$${\rm F}\mathcal{P}_{\omega}{\text -}\id R=\FID R:=\sup\{\id_RM\mid M\in \Mod R\ {\rm and}\ \id_RM<\infty\},$$
$${\rm F}\mathcal{I}_{\omega}{\text -}\pd S=\FPD R:=\sup\{\pd_RN\mid N\in \Mod R\ {\rm and}\ \pd_RN<\infty\}.$$
\item If $R$ is right (or left) quasi $n$-Gorenstein for all $n$, then $\id_{R^{op}}R=\id_RR$ (\cite[Corollary 4]{H2}).
\end{enumerate}
\end{observation}

As a consequence of the above results, we have the following

\begin{corollary} \label{cor-6.9}
Let $R$ be an artin algebra. Then we have
\begin{enumerate}
\item If $R$ satisfies the Auslander condition (that is, $R$ is Auslander $n$-Gorenstein for all $n$), then
$$\FPD R^{op}=\FID R^{op}={\id_{R^{op}}R}=\id_RR=\FPD R=\FID R.$$
\item If $R$ satisfies the right quasi Auslander condition (that is, $R$ is right quasi $n$-Gorenstein for all $n$), then
$$\FPD R\leqslant\FID R={\id_{R^{op}}R}=\id_{R}R\leqslant\FPD R+1.$$
\end{enumerate}
\end{corollary}

\begin{proof}
In view of Example ~\ref{exa-4.20}, Observations \ref{ob-5.7} and \ref{ob-6.8}, the assertions follow from
Corollaries~ \ref{cor-6.5} and \ref{cor-6.7} respectively.
\end{proof}

\vspace{0.5cm}

{\bf Acknowledgements.}
This research was partially supported by NSFC (Grant Nos. 11571164, 11501144),
a Project Funded by the Priority Academic Program Development of Jiangsu Higher Education Institutions
and NSF of Guangxi Province of China (Grant No. 2016GXNSFAA380151).
The authors thank the referee for the useful suggestions.

\providecommand{\bysame}{\leavevmode\hbox to3em{\hrulefill}\thinspace}
\providecommand{\MR}{\relax\ifhmode\unskip\space\fi MR }
\providecommand{\MRhref}[2]{%
\href{http://www.ams.org/mathscinet-getitem?mr=#1}{#2}
}
\providecommand{\href}[2]{#2}

\end{document}